\documentclass[article]{amsart}

\usepackage{amssymb,amsfonts,amsmath,amsthm}
\usepackage[all,arc]{xy}
\usepackage{enumerate}
\usepackage{mathrsfs}
\usepackage[toc,page]{appendix}
\usepackage[left=3cm, right=3cm, bottom=3cm]{geometry}
\usepackage{tabularx}
\usepackage{url}
\usepackage{color}
\usepackage{tikz-cd}  
\usepackage{mathdots} 
\usepackage{romannum} 

\newtheorem{thm}{Theorem}[section]

\newtheorem*{thm*}{Theorem}
\newtheorem*{cor*}{Corollary}
\newtheorem*{prop*}{Proposition}
\newtheorem{cor}[thm]{Corollary}
\newtheorem{prop}[thm]{Proposition}
\newtheorem{lem}[thm]{Lemma}

\theoremstyle{definition}

\newtheorem{exmp}[thm]{Example}

\newtheorem*{notn*}{Notation}

\theoremstyle{remark}
\newtheorem{rem}[thm]{Remark}

\newtheorem*{idea*}{Idea}

\newcommand{\Spec}{{\rm Spec}}

\makeatletter
\let\c@equation\c@thm
\makeatother
\numberwithin{thm}{section}
\numberwithin{equation}{section}

\title[Moduli Space of $\Lambda$-Modules on Projective Deligne-Mumford Stacks]{Moduli Space of $\Lambda$-Modules on Projective Deligne-Mumford Stacks}

\author{Hao Sun}

\begin{document}
\pagenumbering{arabic}
\maketitle
\begin{abstract}
In this paper, we define $\Lambda$-quot-functors on Deligne-Mumford stacks. We prove that the $\Lambda$-quot-functor is representable by an algebraic space. Then, we construct the moduli space of $\Lambda$-modules on a projective Deligne-Mumford stack. We prove that this moduli space is a quasi-projective scheme.
\end{abstract}

\flushbottom

\tableofcontents

\newpage

\renewcommand{\thefootnote}{\fnsymbol{footnote}}
\footnotetext[1]{MSC2010 Class: 14A20, 14C05, 14D20}

\section{Introduction}
The moduli space of Higgs bundles was first introduced by N. Hitchin \cite{Hit1987}. He gave an analytic construction of the moduli space for rank two Higgs bundles over Riemann surfaces. C. Simpson constructed the moduli space of $\Lambda$-modules on smooth projective varieties over $\mathbb{C}$ in an algebraic way, and Higgs bundle is a special case of $\Lambda$-modules \cite{Simp2}. Based on the non-abelian Hodge correspondence, there are many distinct approaches to study the moduli space of Higgs bundles. Our initial motivations grew out of our interest in the moduli space of parabolic Higgs bundles. C. Simpson provided a non-abelian Hodge correspondence in the non-compact case \cite{Simp1990}: (tame) parabolic Higgs bundles are in bijection with meromorphic flat connections with simple poles. It is well-known that the parabolic bundles can be understood as bundles on root stacks (also called orbifolds and $V$-manifolds) \cite{Bis97,FuSt,NaSt}, which can be extended to twisted parabolic Higgs bundles and twisted Higgs bundles on root stacks \cite{KSZ1901}. This brings up the question that whether the non-abelian Hodge correspondence can be extended to root stacks, or more generally, Deligne-Mumford stacks. Based on the correspondence between parabolic bundles and bundles over root stacks, the non-abelian Hodge correspondence holds on root stacks in case of $\text{GL}_n(\mathbb{C})$. However, for a general reductive (or semisimple) group $G$, the non-abelian Hodge correspondence is not known for parabolic $G$-Higgs bundles. P. Boalch pointed out that in this case, the correct choice to construct the non-abelian Hodge correspondence should be the parahoric Higgs $\mathscr{G}$-torsors instead of parabolic $G$-Higgs bundles (see \cite{Boalch2011,Boalch2018}). Similar to parabolic bundles, parahoric $\mathscr{G}$-torsors also correspond a certain type of bundles on root stacks (see \cite[\S 5]{BaSes}). Although the moduli space of parahoric $\mathscr{G}$-torsors on Riemann surfaces is constructed \cite{BaSes}, an algebraic construction of the moduli space of parahoric Higgs bundles on smooth projective varieties is still missing, and the existence of corresponding moduli space on root stacks (or Deligne-Mumford stacks) is unknown yet. This paper is a first step to deal with the above problems, and we give an algebraic construction of the moduli space of $\Lambda$-modules on projective Deligne-Mumford stacks.

Let $X$ be a scheme, and let $\Lambda$ be a sheaf of graded algebras over $X$. A $\Lambda$-module is a coherent sheaf equipped with a $\Lambda$-structure. The $\Lambda$-module is derived from the Higgs bundle. Usually, a Higgs bundle over a smooth projective variety is considered as a pair, which includes a vector bundle and a Higgs field. Similarly, $\Lambda$-modules can be defined in this way. A $\Lambda$-module over a smooth projective variety is also a pair $(F,\Phi)$, where $F$ is a coherent sheaf and $\Phi: \Lambda \rightarrow \mathcal{E}nd(F)$ gives a $\Lambda$-structure on $F$.

C. Simpson constructed the moduli space $\mathcal{M}(X)$ of coherent sheaves and the moduli space $\mathcal{M}_{\Lambda}(X)$ of $\Lambda$-modules on smooth projective varieties decades ago \cite{Simp2}. Afterwards, the moduli space $\mathcal{M}(\mathcal{X})$ of coherent sheaves on (projective) Deligne-Mumford stacks was constructed \cite{Nir,OlSt}, where $\mathcal{X}$ is a Deligne-Mumford stack.
\begin{center}
\begin{tikzcd}
\mathcal{M}_{\Lambda}(\mathcal{X}) \arrow[d, dashrightarrow] \arrow[r, dashrightarrow]   & \mathcal{M}_{\Lambda}(X) \arrow[d]   \\
\mathcal{M}(\mathcal{X}) \arrow[r] & \mathcal{M}(X)
\end{tikzcd}
\end{center}
Therefore, the only mysterious object left in the above diagram is $\mathcal{M}_{\Lambda}(\mathcal{X})$, \emph{the moduli space of $\Lambda$-modules on (projective) Deligne-Mumford stacks}.

C. Simpson constructed the moduli space of $p$-semistable $\Lambda$-modules on a smooth ``projective" Deligne-Mumford stack $\mathcal{X}$ over $\mathbb{C}$ by using a simplicial resolution of $\mathcal{X}$ \cite{Simp2010}. A smooth ``projective" Deligne-Mumford stack $\mathcal{X}$ here means that a Deligne-Mumford stack admits a smooth projective coarse moduli space and a surjective \'etale morphism $Y_0 \rightarrow \mathcal{X}$ such that $Y_0$ is a smooth projective variety over $\mathbb{C}$.

This paper aims at working on a more general type of Deligne-Mumford stacks, projective Deligne-Mumford stacks over an algebraic space, and constructing the moduli space of $\Lambda$-modules on these Deligne-Mumford stacks. The terminology \emph{projective Deligne-Mumford stacks} in this paper is different from Simpson's definition. Roughly speaking, a projective Deligne-Mumford stack is a tame Deligne-Mumford stack over an algebraic space such that its coarse moduli space is projective over the algebraic space. Compared with Simpson's object, we are working on arbitrary characteristic under the condition tameness, and we are not assuming that there exists an \'etale covering $Y_0 \rightarrow \mathcal{X}$ such that $Y_0$ is a smooth projective variety.

As we mentioned above, C. Simpson studied this problem a decade ago. His approach is based on the choice of a special simplicial resolution of the Deligne-Mumford stack and the existence of the moduli space of Higgs bundles on smooth projective varieties over $\mathbb{C}$. More precisely, we suppose that there exists a simplicial resolution $Y_{\bullet}=[\cdots Y_1 \rightrightarrows Y_0] \rightarrow \mathcal{X}$ of the Deligne-Mumford stack $\mathcal{X}$, where $Y_i$ are smooth projective varieties for $i \geq 0$. Then, the existence of the moduli space of $p$-semistable $\Lambda$-bundles on each $Y_i$ gives us the moduli space of $p$-semistable $\Lambda$-modules on the simplicial resolution $Y_{\bullet}$, and therefore the moduli space of $p$-semistable $\Lambda$-modules on $\mathcal{X}$. We refer the reader to \cite{Simp2010} for more details.

In our case, such a resolution may not exist, and therefore Simpson's approach may not work. As shown in the diagram, there is another possible way left for us. We can try to construct the moduli space of $\Lambda$-modules $\mathcal{M}_{\Lambda}(\mathcal{X})$ based on the existence of the moduli space of coherent sheaves $\mathcal{M}(\mathcal{X})$.

Now we review some results of the moduli space of coherent sheaves on Deligne-Mumford stacks. Let $S$ be an algebraic space over an algebraically closed field $k$, and let $\mathcal{X}$ be a Deligne-Mumford stack over $S$ with coarse moduli space $\pi : \mathcal{X} \rightarrow X$. Let $\mathcal{G}$ be a coherent sheaf over $\mathcal{X}$. The quot-functor $\widetilde{{\rm Quot}}(\mathcal{G},\mathcal{X})$ on $\mathcal{X}$ has been defined and studied by M. Olsson and J. Starr \cite{OlSt}. The quot-functor $\widetilde{{\rm Quot}}(\mathcal{G},\mathcal{X})$ is proved to be representable by an algebraic space ${\rm Quot}(\mathcal{G},\mathcal{X})$ \cite[Theorem 1.1]{OlSt}, which gives the existence of the moduli space of coherent sheaves on $\mathcal{X}$. Furthermore, when $S$ is an affine scheme, or a noetherian scheme of finite type and $\mathcal{X}$ is a projective Deligne-Mumford stack, each connected component of ${\rm Quot}(\mathcal{G},\mathcal{X})$ is an $S$-projective scheme \cite[Theorem 1.5]{OlSt} \cite[Theorem 2.17]{Nir}. The connected components of ${\rm Quot}(\mathcal{G},\mathcal{X})$ are parametrized by integer polynomials $P$. Denote by ${\rm Quot}(\mathcal{G},\mathcal{X},P)$ the $S$-projective scheme with respect to the polynomial $P$. Based on the result that the connected components of ${\rm Quot}(\mathcal{G},\mathcal{X})$ are $S$-projective schemes, Nironi studied the moduli problem of $p_{\mathcal{E}}$-semistable coherent sheaves on a projective Deligne-Mumford stack $\mathcal{X}$ using the geometric invariant theory \cite{Nir}, where $\mathcal{E}$ is a generating sheaf on $\mathcal{X}$. The modified Hilbert polynomial with respect to $\mathcal{E}$ and $\mathcal{O}_X(1)$ is defined as follows
\begin{align*}
P_{\mathcal{E}}(\mathcal{F},m)=\chi(\mathcal{X},\mathcal{F} \otimes \mathcal{E}^{\vee} \otimes \pi^* \mathcal{O}_X(m)), \quad m \gg 0,
\end{align*}
where $\mathcal{F}$ is a coherent sheaf on $\mathcal{X}$. The $p_{\mathcal{E}}$-semistability of $\mathcal{F}$ can be defined naturally (see \S 3.3). Clearly, the semistability depends on the choice of the generating sheaf $\mathcal{E}$, and this is the reason why we call it $p_{\mathcal{E}}$-semistability. We omit the subscript $\mathcal{E}$ and use the terminology $p$-semistability for simplicity in the main body of the paper. When $\mathcal{X}$ is a root stack, with a good choice of the generating sheaf $\mathcal{E}$, the $p_{\mathcal{E}}$-semistability is exactly the semistability of the corresponding parabolic bundles on the coarse moduli space $X$.

Based on the construction of the quot-functor and the moduli space of coherent sheaves on Deligne-Mumford stacks, we construct the quot-functor of $\Lambda$-modules on a Deligne-Mumford stack (over an algebraic space) first, and then study the the moduli space of $p_{\mathcal{E}}$-semistable $\Lambda$-modules on a projective Deligne-Mumford stack over an affine scheme (or a noetherian scheme of finite type). In summary, we consider two moduli problems in this paper:
\begin{enumerate}
    \item[\S 5] the quot-functor $\widetilde{{\rm Quot}}_{\Lambda}(\mathcal{G},\mathcal{X})$ of $\Lambda$-modules (also called the $\Lambda$-quot-functor), where $\mathcal{X}$ is a separated and locally finitely-presented Deligne-Mumford stack over an algebraic space $S$,
    \item[\S 6] the moduli problem $\widetilde{\mathcal{M}}_\Lambda^{ss}(\mathcal{E},\mathcal{O}_X(1),P)$ of $p_{\mathcal{E}}$-semistable $\Lambda$-modules with the modified Hilbert polynomial $P$ over $\mathcal{X}$, where $\mathcal{X}$ is a projective Deligne-Mumford stack over an affine scheme (or noetherian scheme of finite type) $S$, $\mathcal{E}$ is a generating sheaf on $\mathcal{X}$, $\mathcal{O}_X(1)$ is a polarization on the coarse moduli space $X$ and $P$ is an integer polynomial.
\end{enumerate}
Compared with the second moduli problem $\widetilde{\mathcal{M}}^{ss}(\mathcal{E},\mathcal{O}_X(1),P)$, the setup of the first problem is more general, in which the Deligne-Mumford stack $\mathcal{X}$ is not even a tame stack.

The prerequisite sections for studying the first moduli problem in \S 5 are \S 3.1 and \S 4.1, while the prerequisite sections for the second one in \S 6 are \S 3.2-\S 3.8, \S 4 and \S 5.4-\S 5.5. In conclusion,
\begin{enumerate}
\item Deligne-Mumford stacks considered in \S 3.1, \S 4.1 and \S 5.1-\S 5.3 are separated and locally finitely-presented over an algebraic space $S$.
\item We consider projective (or quasi-projective) Deligne-Mumford stacks over an affine scheme $S$ in \S 3.2-\S 3.8, \S 4, \S 5.4-\S 5.5 and \S 6.
\end{enumerate}

Here is the structure of the paper.

In \S 2, we review the definitions and some properties of tame Deligne-Mumford stacks and projective Deligne-Mumford stacks. As a tame (or projective) Deligne-Mumford stack $\mathcal{X}$, the natural map $\pi: \mathcal{X} \rightarrow X$ to its coarse moduli space $X$ induces an exact functor $\pi_*:{\rm QCoh}(\mathcal{X}) \rightarrow {\rm QCoh}(X)$, where ${\rm QCoh}$ is the category (or stack) of quasi-coherent sheaves. The functor $\pi_*$ may not be injective. If there exists an injective exact functor $F: {\rm QCoh}(\mathcal{X}) \rightarrow {\rm QCoh}(X)$, then ${\rm QCoh}(\mathcal{X})$ can be probably considered as a closed (or locally closed) subset of ${\rm QCoh}(X)$. In \S 2.2 and \S 2.3, we introduce the functor
\begin{align*}
F_{\mathcal{E}}: {\rm QCoh}(\mathcal{X}) \rightarrow {\rm QCoh}(X)
\end{align*}
where $\mathcal{E}$ is a generating sheaf. This functor $F_{\mathcal{E}}$ is proved to be an injective exact natural transformation for quot-functors \cite[Proposition 6.2]{OlSt}, and the injectivity of $F_{\mathcal{E}}$ also holds for quot-functors of $\Lambda$-modules (see Lemma \ref{508}). This functor plays an important role when we study the moduli space of $p_{\mathcal{E}}$-semistable coherent sheaves (see \S 3) and the moduli space of $p_{\mathcal{E}}$-semistable $\Lambda$-modules (see \S 6). In \S 2.5, we give the definition of moduli problems and representabilities we consider in this paper. A moduli problem is defined as a sheaf over the category of schemes over an algebraic space with respect to the big \'etale topology (or fppf topology). This definition is equivalent to consider a moduli problem as a category fibered in groupoids (CFG) satisfying the effective descent conditions \cite{CasaWise,Ol}. Given a moduli problem, it is important to understand whether there exists a coarse moduli space or a fine moduli space \cite{HaMo}. Furthermore, since we study the moduli problem related to coherent sheaves in this paper, we also give the definitions of co-representability and universal co-representability \cite{HuLe}.

In \S 3, we first review a general result that the quot-functor $\widetilde{{\rm Quot}}(\mathcal{G},\mathcal{X})$ is representable by an algebraic space ${\rm Quot}(\mathcal{G},\mathcal{X})$, where $\mathcal{X}$ is a Deligne-Mumford stack over an algebraic space $S$ (see \cite[Theorem 1.1]{OlSt} or Theorem \ref{301}). Afterwards, we give the definition of saturations, modified Hilbert polynomials, $p_{\mathcal{E}}$-stability condition, Harder-Narasimhan filtrations and Jordan-H\"{o}lder filtrations. Saturations (see Corollary \ref{304}) and modified Hilbert polynomials (see \S 3.3) of coherent sheaves over $\mathcal{X}$ are preserved under the functor $F_{\mathcal{E}}$, while the $p_{\mathcal{E}}$-stability and two filtrations are not preserved (see Remark \ref{305}). This is the reason why we have to study the $p_{\mathcal{E}}$-stability conditions in detail. Next, we review properties of boundedness of $p_{\mathcal{E}}$-semistable coherent sheaves over a projective Deligne-Mumford stack in \S 3.5 and \S 3.6. Note that there are two distinct properties called the \emph{boundedness}. Boundedness \Romannum{1} is related to the regularity and the existence of a universal family, while Boundedness \Romannum{2} is about the upper bound of the global sections of $p_{\mathcal{E}}$-semistable coherent sheaves. Langer studied Boundedness \Romannum{2} of $p$-semistable coherent sheaves on schemes in positive characteristic \cite{Lan}. Based on the geometric invariant theory, we give the construction of the moduli space of $p_{\mathcal{E}}$-semistable coherent sheaves on projective Deligne-Mumford stacks.
\begin{thm}[Theorem \ref{317}]
With respect to the situation above, we have the following results.
\begin{enumerate}
\item The moduli space $\mathcal{M}^{ss}(\mathcal{E},\mathcal{O}_X(1),P)$ is a projective $S$-scheme.
\item There exists a natural morphism
\begin{align*}
\widetilde{\mathcal{M}}^{ss}(\mathcal{E},\mathcal{O}_X(1),P) \rightarrow \mathcal{M}^{ss}(\mathcal{E},\mathcal{O}_X(1),P)
\end{align*}
such that $\mathcal{M}^{ss}(\mathcal{E},\mathcal{O}_X(1),P)$ universally co-represents $\widetilde{\mathcal{M}}^{ss}(\mathcal{E},\mathcal{O}_X(1),P)$. The points in the moduli space $\mathcal{M}^{ss}(\mathcal{E},\mathcal{O}_X(1),P)$ represent the $S$-equivalent classes of $p$-semistable sheaves.
\item $\mathcal{M}^{s}(\mathcal{E},\mathcal{O}_X(1),P)$ is a coarse moduli space of $\widetilde{\mathcal{M}}^{s}(\mathcal{E},\mathcal{O}_X(1),P)$.
\item If $x \in \mathcal{M}^{s}(\mathcal{E},\mathcal{O}_X(1),P)$ is a point such that $Q^s$ is smooth at the inverse image of $x$, then $\mathcal{M}^{s}(\mathcal{E},\mathcal{O}_X(1),P)$ is smooth at $x$.
\end{enumerate}
\end{thm}
Most of the materials in this section is included in \cite{Nir,OlSt}, but the construction of the moduli space we give is slightly different from that in \cite[\S 6]{Nir}. The first two statements of the theorem are also proved in \cite[Theorem 6.22]{Nir}. If the reader finds a proof of a statement in this section, it means that we have not seen it in any reference.

In \S 4, we give the definition of sheaves of graded algebras $\Lambda$ and $\Lambda$-modules. Sheaves of graded algebras over smooth projective varieties are introduced in \cite[\S 2]{Simp2}, and we generalize the definition to Deligne-Mumford stacks. Next, we define the $p_{\mathcal{E}}$-semistability of $\Lambda$-modules, $\Lambda$-Harder-Narasimhan filtrations and $\Lambda$-Jordan-H\"{o}lder filtrations. In \S 4.3, we prove the Boundedness \Romannum{2} of $\Lambda$-modules. Note that Boundedness \Romannum{1} of $\Lambda$-modules is not proved in this section, which is proved in \S 5.5, because the proof of this property depends on the representability of the $\Lambda$-quot-functor.

In \S 5, we consider the $\Lambda$-quot-functor
\begin{align*}
\widetilde{{\rm Quot}}_{\Lambda}(\mathcal{G},\mathcal{X}): (\text{Sch}/S)^{{\rm op}} \rightarrow \text{Set}.
\end{align*}
For each $S$-scheme $T$, $\widetilde{{\rm Quot}}_{\Lambda}(\mathcal{G},\mathcal{X})(T)$ is defined as the set of $\mathcal{O}_{\mathcal{X}_T}$-module quotients $\mathcal{G}_T \rightarrow \mathcal{F}_T$ such that
\begin{enumerate}
\item $\mathcal{F}_T \in \widetilde{{\rm Quot}}(\mathcal{G})(T)$,
\item $\mathcal{F}_T$ is a $\Lambda_{T}$-module.
\end{enumerate}

Here is the main result in this section.
\begin{thm}[Theorem \ref{501}]
Let $S$ be an algebraic space, which is locally of finite type over an algebraically closed field $k$, and let $\mathcal{X}$ be a separated and locally finitely-presented Deligne-Mumford stack over $S$. The $\Lambda$-quot-functor $\widetilde{{\rm Quot}}_{\Lambda}(\mathcal{G},\mathcal{X})$ is represented by a separated and locally finitely presented algebraic space.
\end{thm}
We apply a theorem by Artin (see \cite[Theorem 5.3]{Art} or Theorem \ref{502}) to prove this result. The theorem by Artin lists all necessary conditions, under which a moduli problem is representable by an algebraic space. These conditions are \emph{locally of finite presentation}, \emph{integrability} (or called \emph{effectivity}), \emph{separation}, \emph{deformation theory} and \emph{obstruction theory}. There are many good references about the first three conditions. We refer the reader to \cite{CasaWise,HaRy,Hall} for more details. The infinitesimal deformation theory of Hitchin pairs was studied by I. Biswas and S. Ramanan \cite{BisRam}. They constructed a two-term complex and proved that the first hypercohomology group of the two-term complex is exactly the tangent space of the moduli space of Hitchin pairs over a smooth projective curve. Based on this idea, we construct the deformation and obstruction theory for $\Lambda$-modules and prove that the theory satisfies all of the conditions in the Artin's theorem. In this paper, the deformation and obstruction theory follows from Artin's definition (see \cite[\S 5]{Art} or \S 5.2.4).

We make a brief review about Artin's theorem and necessary backgrounds in \S 5.2. We give the statement of the main result Theorem \ref{501} in \S 5.1, and the proof of this theorem is included in \S 5.3, where we check that $\Lambda$-quot-functors satisfy all of the conditions in the Artin's theorem. In \S 5.4, we consider the case that $S$ is an affine scheme (or noetherian scheme of finite type) and $\mathcal{X}$ is a projective Deligne-Mumford stack over $S$. Under this condition, we prove that the functor $F_{\mathcal{E}}:{\rm QCoh}(\mathcal{X}) \rightarrow {\rm QCoh}(X)$ induces a natural transformation
\begin{align*}
F_{\mathcal{E}}:\widetilde{{\rm Quot}}_{\Lambda}(\mathcal{G},\mathcal{X}) \rightarrow \widetilde{{\rm Quot}}_{F_{\mathcal{E}}(\Lambda)}(F_{\mathcal{E}}(\mathcal{G}),X),
\end{align*}
which is a monomorphism (see Lemma \ref{508}). Together with the following result,
\begin{prop}[Proposition \ref{5101}]
Let $X$ be a projective scheme, and let $G$ be a coherent sheaf on $X$. We fix a polynomial $P$. The $\Lambda$-quot-functor $\widetilde{{\rm Quot}}_{\Lambda}(G,X, P)$ is represented by a quasi-projective scheme.
\end{prop}
we prove the following theorem.
\begin{thm}[Theorem \ref{511}]
Let $S$ be an affine scheme (or a noetherian scheme of finite type) and let $\mathcal{X}$ be a projective Deligne-Mumford stack over $S$. The $\Lambda$-quot-functor  $\widetilde{\rm Quot}_{\Lambda}(\mathcal{G},\mathcal{X},P)$ with respect to a given integer polynomial $P$ is represented by a quasi-projective $S$-scheme.
\end{thm}
The proof of Proposition \ref{5101} will be used in proving Proposition \ref{603}. In \S 5.5, we prove the Boundedness \Romannum{1} of $\Lambda$-modules (see Corollary \ref{514}). Theorem \ref{511}, Boundedness \Romannum{1} of $\Lambda$-modules (see Corollary \ref{514}) and Boundedness \Romannum{2} of $\Lambda$-modules (see Proposition \ref{407}) will be used to construct the moduli space of $p_{\mathcal{E}}$-semistable $\Lambda$-modules in \S 6.

In \S 6, we focus on the moduli problem $\widetilde{\mathcal{M}}_{\Lambda}^{ss}(\mathcal{E},\mathcal{O}_X(1),P)$ of $p_{\mathcal{E}}$-semistable $\Lambda$-modules on projective Deligne-Mumford stacks. The version of this moduli problem on smooth projective varieties over $\mathbb{C}$ is studied in \cite{Simp2}. We construct the moduli space $\mathcal{M}_\Lambda^{ss}(\mathcal{E},\mathcal{O}_X(1),P)$ of $p_{\mathcal{E}}$-semistable $\Lambda$-modules on projective Deligne-Mumford stacks, and prove that this moduli space universally co-represents the moduli problem $\widetilde{\mathcal{M}}_{\Lambda}^{ss}(\mathcal{E},\mathcal{O}_X(1),P)$.

\begin{thm}[Theorem \ref{607}]
\leavevmode
\begin{enumerate}
\item The moduli space $\mathcal{M}_\Lambda^{ss}(\mathcal{E},\mathcal{O}_X(1),P)$ is a quasi-projective $S$-scheme.
\item There exists a natural morphism
\begin{align*}
\widetilde{\mathcal{M}}_{\Lambda}^{ss}(\mathcal{E},\mathcal{O}_X(1),P) \rightarrow \mathcal{M}_\Lambda^{ss}(\mathcal{E},\mathcal{O}_X(1),P)
\end{align*}
such that $\mathcal{M}_\Lambda^{ss}(\mathcal{E},\mathcal{O}_X(1),P)$ universally co-represents $\widetilde{\mathcal{M}}_\Lambda^{ss}(\mathcal{E},\mathcal{O}_X(1),P)$, and the points of $\mathcal{M}_\Lambda^{ss}(\mathcal{E},\mathcal{O}_X(1),P)$ represent the $S$-equivalent classes of $p$-semistable $\Lambda$-modules with modified Hilbert polynomial $P$.
\item $\mathcal{M}_\Lambda^{s}(\mathcal{E},\mathcal{O}_X(1),P)$ is a coarse moduli space of $\widetilde{\mathcal{M}}_\Lambda^{s}(\mathcal{E},\mathcal{O}_X(1),P)$, and its points represent isomorphism classes of $p$-stable $\Lambda$-modules.
\end{enumerate}
\end{thm}

\vspace{2mm}
\textbf{Acknowledgments}.
I would like to thank G. Kydonakis and L. Zhao for very helpful discussions about the moduli space of Higgs bundles. Thanks S. Casalaina-Martin for sharing the idea about the Artin's theorem, and P. Boalch for a very helpful description of the parahoric torsors and the non-abelian Hodge correspondence.
\vspace{2mm}

\section{Preliminaries}
\subsection{Tame Algebraic Stacks}
Let $k$ be an algebraically closed field, and let $S$ be an algebraic space, which is locally of finite type over an algebraically closed field $k$. An algebraic stack $\mathcal{X}$ over $S$ is a morphism $\mathcal{X} \rightarrow S$, which is also considered as a family of algebraic stacks over $S$. Let $\mathcal{X}$ be an algebraic stack over $S$ such that such that $\mathcal{X}$ is locally of finite presentation over $S$ with finite diagonal, where finite diagonal means that the natural map $I_{\mathcal{X}}\rightarrow \mathcal{X}$ is finite, where $I_{\mathcal{X}}$ is the inertia stack. Under these conditions, the algebraic stack $\mathcal{X}$ has a \emph{coarse moduli space} $X$ \cite[Theorem 11.1.2]{Ol}. Denote by $\pi: \mathcal{X} \rightarrow X$ the natural morphism. Note that there is a natural morphism $\rho:X \rightarrow S$ such that
\begin{center}
\begin{tikzcd}
\mathcal{X} \arrow[rd] \arrow[rr, "\pi"] &  & X \arrow[ld,"\rho"] \\
& S &
\end{tikzcd}
\end{center}

An algebraic stack $\mathcal{X}$, which has a coarse moduli space $\pi: \mathcal{X} \rightarrow X$, is \emph{tame} if the induced functor
\begin{align*}
\pi_*: {\rm QCoh}(\mathcal{X}) \rightarrow {\rm QCoh}(X)
\end{align*}
is exact, where ${\rm QCoh}(*)$ is the category of quasi-coherent sheaves over $*$, which is a scheme, an algebraic space or an algebraic stack. The category ${\rm QCoh}(*)$ has a natural stack structure \cite{Ol}. If ${\rm char} (k)=0$, the functor $\pi_*$ is always exact and the nontrivial case comes from the positive characteristic. In \cite{AbOlVis}, the authors studied the tame stack in detail and proved several equivalent conditions in \cite[Theorem 3.2]{AbOlVis}. We list some of the conditions as follows:
\begin{itemize}
\item The algebraic stack $\mathcal{X}$ is tame.
\item Let $k'$ be an algebraically closed field with a morphism $\Spec(k') \rightarrow S$. Let $\xi$ be an object in $\mathcal{X}(\Spec (k'))$. Then the automorphism group scheme ${\rm Aut}_{k'}(\xi) \rightarrow \Spec (k)'$ is linearly reductive.
\item There exists an fppf (or surjective \'etale) cover of the coarse moduli space $X' \rightarrow X$, a finite and finitely presented scheme $U$ over $X'$ and a linearly reductive group scheme $G \rightarrow X'$ acting on $U$ together with an isomorphism
    \begin{align*}
    \mathcal{X} \times_{X} X' \cong [U/G].
    \end{align*}
\end{itemize}

\subsection{Two Functors: $F_{\mathcal{E}}$ and $G_{\mathcal{E}}$}
Let $\mathcal{X}$ be a tame algebraic stack. Let $\mathcal{E}$ be a locally free sheaf over $\mathcal{X}$. We define two functors $F_{\mathcal{E}} : {\rm QCoh}(\mathcal{X}) \rightarrow {\rm QCoh}(X)$ and $G_{\mathcal{E}}:  {\rm QCoh}(X) \rightarrow {\rm QCoh}(\mathcal{X})$ as follows
\begin{align*}
F_{\mathcal{E}}(\mathcal{F})= \pi_* \mathcal{H}om_{\mathcal{O}_{\mathcal{X}}}(\mathcal{E},\mathcal{F}), \quad G_{\mathcal{E}}(F)=\pi^*F \otimes \mathcal{E},
\end{align*}
where $\mathcal{F} \in {\rm QCoh}(\mathcal{X})$ and $F \in {\rm QCoh}(X)$. The functor $F_{\mathcal{E}}$ is exact since $\pi_*$ is exact and $\mathcal{E}^{\vee}$ is a locally free sheaf. If $\pi:\mathcal{X} \rightarrow X$ is flat, then the functor $G_{\mathcal{E}}$ is also exact.

The compositions of the above two functors
\begin{align*}
& G_{\mathcal{E}} \circ F_{\mathcal{E}}: {\rm QCoh}(\mathcal{X}) \rightarrow {\rm QCoh}(\mathcal{X}), \\
& F_{\mathcal{E}} \circ G_{\mathcal{E}}: {\rm QCoh}(X) \rightarrow {\rm QCoh}(X),
\end{align*}
can be written in the following way
\begin{align*}
& G_{\mathcal{E}} \circ F_{\mathcal{E}}(\mathcal{F})=\pi^* \pi_* \mathcal{H}om_{\mathcal{O}_{\mathcal{X}}}(\mathcal{E},\mathcal{F})\otimes \mathcal{E},\\
& F_{\mathcal{E}} \circ G_{\mathcal{E}}(F)=\pi_* \mathcal{H}om_{\mathcal{O}_{\mathcal{X}}}(\mathcal{E},\pi^*F \otimes \mathcal{E}).
\end{align*}
We define the following morphisms
\begin{align*}
& \theta_{\mathcal{E}}(\mathcal{F}): G_{\mathcal{E}} \circ F_{\mathcal{E}}(\mathcal{F}) \rightarrow \mathcal{F},\\
& \varphi_{\mathcal{E}}(F): F \rightarrow F_{\mathcal{E}} \circ G_{\mathcal{E}}(F).
\end{align*}
The morphism $\theta_{\mathcal{E}}(\mathcal{F})$ is exactly the adjunction morphism left adjoint to the identity morphism
\begin{align*}
\pi_*(\mathcal{F} \otimes \mathcal{E}^{\vee}) \xrightarrow{\rm id} \pi_*(\mathcal{F} \otimes \mathcal{E}^{\vee}).
\end{align*}

If $\mathcal{F}$ is a quasi-coherent sheaf on $\mathcal{X}$, then the following composition is the identity
\begin{align*}
F_{\mathcal{E}}(\mathcal{F}) \xrightarrow{\varphi_{\mathcal{E}}(F_{\mathcal{E}}(\mathcal{F}))} F_{\mathcal{E}} \circ G_{\mathcal{E}} \circ F_{\mathcal{E}}(\mathcal{F}) \xrightarrow{F_{\mathcal{E}}(\theta_{\mathcal{E}}(\mathcal{F}))} F_{\mathcal{E}}(\mathcal{F}).
\end{align*}
Similarly, let $F$ be a coherent sheaf on $X$, then the following composition is also the identity
\begin{align*}
G_{\mathcal{E}}(F) \xrightarrow{G_{\mathcal{E}} (\varphi_{\mathcal{E}}(F)) } G_{\mathcal{E}} \circ F_{\mathcal{E}} \circ G_{\mathcal{E}}(F) \xrightarrow{\theta_{\mathcal{E}} (G_{\mathcal{E}}(F)) } G_{\mathcal{E}}(F).
\end{align*}
These properties are proved in \cite[Lemma 2.9]{Nir}.

\subsection{Generating Sheaves}
With the same setup as in \S 2.2, a locally free sheaf $\mathcal{E}$ is a \emph{generator} for $\mathcal{F} \in {\rm QCoh}(\mathcal{X})$, if the morphism
\begin{align*}
\theta_{\mathcal{E}}(\mathcal{F}): \pi^* \pi_* \mathcal{H}om_{\mathcal{O}_{\mathcal{X}}}(\mathcal{E},\mathcal{F})\otimes \mathcal{E} \rightarrow \mathcal{F}
\end{align*}
is surjective. A locally free sheaf $\mathcal{E}$ is a \emph{generating sheaf} of $\mathcal{X}$, if it is a generator for every quasi-coherent sheaf on $\mathcal{X}$.

If $\mathcal{X}$ is a tame separated Deligne-Mumford stack of the form $\mathcal{X}=[Y/G]$, where $Y$ is a scheme and $G$ is a finite group acting on $Y$, then there exists a generating sheaf $\mathcal{E}$ of $\mathcal{X}$ \cite[Proposition 5.2]{OlSt}. Olsson and Starr also proved a more general result about the existence of a generating sheaf of a tame, separated and finitely presented Deligne-Mumford stack $\mathcal{X}$ over $S$, where the stack $\mathcal{X}=[Y/ \text{GL}_{n,S}]$ can be written as a global quotient \cite[Theorem 5.7]{OlSt}.

A locally free sheaf $\mathcal{E}$ on $\mathcal{X}$ is \emph{$\pi$-very ample}, if the representation of the stabilizer group at any geometric point of $\mathcal{X}$ contains every  irreducible representation. Nironi showed that a locally free sheaf $\mathcal{E}$ on a tame Deligne-Mumford stack $\mathcal{X}$ is a generating sheaf if and only if $\mathcal{E}$ is $\pi$-very ample \cite[Proposition 2.7]{Nir}. This property follows from the proof of \cite[Proposition 5.2]{OlSt}.

The existence of a generating sheaf $\mathcal{E}$ helps us to define a monomorphism of the quot-functors \cite[Lemma 6.1]{OlSt}
\begin{align*}
F_{\mathcal{E}}: \widetilde{{\rm Quot}}(\mathcal{G},\mathcal{X}) \rightarrow \widetilde{{\rm Quot}}(F_{\mathcal{E}}(\mathcal{G}),X),
\end{align*}
where $\mathcal{G}$ is a coherent sheaf over $\mathcal{X}$, and $\widetilde{{\rm Quot}}(\mathcal{G},\mathcal{X})$ is the quot-functor over $\mathcal{X}$ and $\widetilde{{\rm Quot}}(F_{\mathcal{E}}(\mathcal{G}),X)$ is the quot-functor over $X$. The quot-functors $\widetilde{{\rm Quot}}(\mathcal{G},\mathcal{X})$ and $\widetilde{{\rm Quot}}(F_{\mathcal{E}}(\mathcal{G}),X)$ are represented by algebraic spaces ${\rm Quot}(\mathcal{G},\mathcal{X})$ and ${\rm Quot}(F_{\mathcal{E}}(\mathcal{G}),X)$ respectively \cite[Theorem 1.1]{OlSt}. The monomorphism $F_{\mathcal{E}}$ can be improved to be a closed embedding of the corresponding algebraic spaces \cite[Theorem 4.4]{OlSt}
\begin{align*}
{\rm Quot}(\mathcal{G},\mathcal{X}) \hookrightarrow {\rm Quot}(F_{\mathcal{E}}(\mathcal{G}),X).
\end{align*}
This property allows us to study ${\rm Quot}(\mathcal{G},\mathcal{X})$ as a closed algebraic subspace of ${\rm Quot}(F_{\mathcal{E}}(\mathcal{G}),X)$. We will review these properties in \S 3.

\subsection{Projective Deligne-Mumford Stacks}
Projective Deligne-Mumford stacks are the main object we study in this paper. As its name, a projective Deligne-Mumford stack is not only a stack, but also inherits some good properties from projective schemes. We briefly review the definition of some properties of the projective Deligne-Mumford stacks. Readers can find those materials in \cite{Kre,Nir,OlSt}.

Let $k$ be an algebraically closed field and let $S$ be an algebraic space. With the same setup as in \S 2.1, a Deligne-Mumford stack $\mathcal{X}$ over $k$ is a \emph{projective stack} (resp. \emph{quasi-projective stack}), if $\mathcal{X}$ is a tame separated global quotient and its coarse moduli space $X$ is a projective (quasi-projective) scheme. Kresch proved that the following conditions are equivalent of a Deligne-Mumford stack $\mathcal{X}$ in characteristic zero \cite[Theorem 5.3, Lemma 5.4]{Kre}:
\begin{itemize}
\item $\mathcal{X}$ is projective (quasi-projective).
\item $\mathcal{X}$ has a projective (quasi-projective) moduli space $X$ and there exists a generating sheaf of $\mathcal{X}$.
\item $\mathcal{X}$ has a closed (locally closed) embedding in a smooth projective Deligne-Mumford stack.
\end{itemize}

Now let $\mathcal{X}$ be a tame separated Deligne-Mumford stack over $S$ which can be written as a global quotient. The stack $\mathcal{X}$ is a \emph{projective} ({\rm resp. }\emph{quasi-projective}) \emph{Deligne-Mumford stack over $S$}, if the morphism $\mathcal{X} \rightarrow S$ is the composition of $\pi: \mathcal{X} \rightarrow X$ and $\rho: X \rightarrow S$, where $\rho$ is a projective (resp. quasi-projective) morphism. In this case, we also say that $\mathcal{X}$ is a \emph{family of projective Deligne-Mumford stacks over $S$}. If $\mathcal{X}$ is a projective Deligne-Mumford stack over $S$, each fiber $\mathcal{X}_s$ over a geometric point $s$ of $S$ is a projective stack over $k$. Note that a projective Deligne-Mumford stack $\mathcal{X}$ over $S$ does not mean that the morphism $\mathcal{X} \rightarrow S$ is projective.

\subsection{Moduli Problem}
Let $S$ be an algebraic space, which is locally of finite type over an algebraically closed field $k$. A \emph{moduli problem} $\widetilde{F}$ is a functor
\begin{align*}
\widetilde{F}: ({\rm Sch}/S)^{\rm op} \rightarrow {\rm Set},
\end{align*}
where $({\rm Sch}/S)^{\rm op}$ is the opposite category of schemes over $S$ with respect to the fppf topology or the big \'etale topology. The functor $\widetilde{F}$ sends a scheme $T$ to a set of isomorphism classes of some objects. This is the classical definition of a moduli problem.

Nowadays, people prefer to consider a moduli problem as a category $\mathcal{M}$ fibered over the category $({\rm Sch}/S)$ in groupoids. In other words, a moduli problem over $({\rm Sch}/S)$ is a pair $(\mathcal{M},\hat{F})$, where $\mathcal{M}$ is a category fibered over $({\rm Sch}/S)$ and $\hat{F}: \mathcal{M} \rightarrow ({\rm Sch}/S)$ is a functor, such that $\hat{F}^{-1}(T)$ is a groupoid for every $T \in ({\rm Sch}/S)$.

There is a natural way to construct a functor $\widetilde{F}:({\rm Sch}/S)^{\rm op} \rightarrow {\rm Set}$ from the pair $(\mathcal{M},\hat{F})$. Given $T \in ({\rm Sch}/S)$, define $\widetilde{F}(T)=\hat{F}^{-1}(T)$. With respect to this construction, a moduli problem is a presheaf over $({\rm Sch}/S)$ with respect to the fppf or big \'etale topology. Furthermore, if $\hat{F}$ is a \emph{category fibered in groupoids} (CFG) satisfying the effective descent conditions, then the corresponding functor $\widetilde{F}$ is a sheaf. There are many good references about the above constructions and properties \cite{ArCorGri,CasaWise,Sta}.

We still use the classical definition, as contravariant functors, of moduli problems, and all moduli problems we consider in this paper are also sheaves. We will give some examples, which are closely related to the moduli problems we consider in this paper, at the end of this subsection.

Let $\widetilde{F}: ({\rm Sch}/S)^{\rm op} \rightarrow {\rm Set}$ be a moduli problem. The representability is a very important property of a moduli problem. Here are the definitions about the representability of $\widetilde{F}$ we consider in this paper.
\begin{enumerate}
\item If there is an algebraic space (or a scheme) $F$ such that $\widetilde{F}$ is isomorphic to ${\rm Hom}(\bullet,F)$, then we say that $F$ is a \emph{fine moduli space} for the moduli problem, or the moduli problem $\widetilde{F}$ \emph{is represented by} $F$.

\item A moduli problem $\widetilde{F}$ is \emph{co-represented} by an algebraic space (or a scheme) $F$ if there is a morphism $\alpha: \widetilde{F} \rightarrow {\rm Hom}(\bullet,F)$ such that for any $S$-scheme $T$ and any morphism $\alpha' : \widetilde{F} \rightarrow {\rm Hom}(\bullet,T)$, there is a unique morphism ${\rm Hom}(\bullet,F) \rightarrow {\rm Hom}(\bullet,T)$ such that the following diagram commutes:
\begin{center}
\begin{tikzcd}
\widetilde{F} \arrow[d, "\alpha'"] \arrow[r, "\alpha"]   & {\rm Hom}(\bullet, F)\arrow[ld,"\exists !"]   \\
{\rm Hom}(\bullet, T) &
\end{tikzcd}
\end{center}

\item A moduli problem $\widetilde{F}$ is \emph{universally co-represented} by $F$, if there is a morphism $\alpha: \widetilde{F} \rightarrow {\rm Hom}(\bullet,F)$ such that for any morphism $\beta: T \rightarrow F$, the fiber product ${\rm Hom}(\bullet,T) \times_{ {\rm Hom}(\bullet,F) } \widetilde{F}$ is co-represented by $T$.

\item We say that $F$ is a \emph{coarse moduli space} of the moduli problem $\widetilde{F}$, if there is a morphism $\alpha: \widetilde{F} \rightarrow {\rm Hom}(\bullet,F)$ such that
\begin{enumerate}
\item $F$ co-represents $\widetilde{F}$;
\item the map $\alpha_S: \widetilde{F}(S) \rightarrow {\rm Hom}(S,F)$ is a bijection.
\end{enumerate}
\end{enumerate}
\subsubsection{Moduli Problem of Coherent Sheaves on Deligne-Mumford Stacks}
Let $\mathcal{X}$ be a separated and locally finitely-presented Deligne-Mumford stack over $S$. We define
\begin{align*}
\widetilde{\mathcal{M}}(\mathcal{X}): (\text{Sch}/S)^{{\rm op}} \rightarrow \text{Set}
\end{align*}
the moduli problem of coherent sheaves as follows. For each $S$-scheme $T$, we define $\mathcal{X}_T$ as $\mathcal{X} \times_{S} T$ and $\mathcal{F}_T$ the pullback of $\mathcal{F}$ to $\mathcal{X}_T$. Define $\widetilde{\mathcal{M}}(\mathcal{X})(T)$ to be the set of coherent sheaves over $\mathcal{O}_{\mathcal{X}_T}$ such that
\begin{enumerate}
\item $\mathcal{F}_T$ is a coherent $\mathcal{O}_{\mathcal{X}_T}$-module;
\item $\mathcal{F}_T$ is flat over $T$;
\item the support of $\mathcal{F}_T$ is proper over $T$.
\end{enumerate}
This moduli problem has a natural structure as a stack, which means that the moduli problem is a sheaf over $(\text{Sch}/S)$. Moreover, Nironi proved that $\widetilde{\mathcal{M}}(\mathcal{X})$ can be covered by the quot-functors \cite[\S 2]{Nir}, and deduces that $\widetilde{\mathcal{M}}(\mathcal{X})$ is an algebraic space \cite[Proposition 2.26, Corollary 2.27]{Nir}

\subsubsection{Moduli Problem of $\mathcal{L}$-twisted Hitchin Pairs on Deligne-Mumford Stacks}
Let $\mathcal{X}$ be a separated and locally finitely-presented Deligne-Mumford stack over $S$. We fix a line bundle (invertible sheaf) $\mathcal{L}$ over $\mathcal{X}$, which is considered as the bundle for twisting. Let $\mathcal{F}$ be a coherent sheaf over $\mathcal{X}$. An \emph{$\mathcal{L}$-twisted Higgs field} $\Phi$ on the coherent sheaf $\mathcal{F}$ is a homomorphism
\begin{align*}
\Phi:\mathcal{F} \rightarrow \mathcal{F} \otimes \mathcal{L}.
\end{align*}
An \emph{$\mathcal{L}$-twisted Hitchin pair} over $\mathcal{X}$ is a pair $(\mathcal{F},\Phi)$, where $\mathcal{F}$ is a coherent sheaf over $\mathcal{X}$ and $\Phi$ is an $\mathcal{L}$-twisted Higgs field.

We define
\begin{align*}
\widetilde{\mathcal{M}}_H(\mathcal{X},\mathcal{L}): (\text{Sch/S})^{{\rm op}} \rightarrow \text{Set}
\end{align*}
the moduli problem of $\mathcal{L}$-twisted Hitchin pairs on $\mathcal{X}$ as follows. For each $S$-scheme $T$, we define the following set
\begin{align*}
\widetilde{\mathcal{M}}_H(\mathcal{X},\mathcal{L})(T)=\{ (\mathcal{F}_T,\Phi_T) \text{ } | \text{ }\mathcal{F}_T \in \widetilde{\mathcal{M}}(\mathcal{X})(T), \Phi_T: \mathcal{F}_T \rightarrow \mathcal{F}_T \otimes p_{\mathcal{X}}^* \mathcal{L} \} ,
\end{align*}
where $p_{\mathcal{X}}:\mathcal{X} \times_S T \rightarrow \mathcal{X}$ is the natural projection. This moduli problem is also an algebraic stack \cite{CasaWise} and is proved by the author that the moduli problem $\widetilde{\mathcal{M}}_H(\mathcal{X},\mathcal{L})$ is represented by an algebraic space \cite{Sun20194}. In fact, the twisted Hitchin pairs is a special case of $\Lambda$-modules, where $\Lambda$ is a sheaf of graded algebras (see \S 4.1).

\section{Moduli Space of Coherent Sheaves on Deligne-Mumford Stacks}
Quot-functors of coherent sheaves on a Deligne-Mumford stack $\mathcal{X}$ are proved to be represented by algebraic spaces \cite[Theorem 1.1]{OlSt}, which is called the \emph{quot-spaces} in this paper. Under some necessary conditions, a quot-space is a scheme, and we call it \emph{quot-scheme}. In this case, a quot-scheme of coherent sheaves over $\mathcal{X}$ can be embedded into a quot-scheme on the coarse moduli space $X$ of $\mathcal{X}$ \cite[\S 6]{OlSt}. This property enables us to construct the moduli space of $p$-semistable coherent sheaves on $\mathcal{X}$.

In \S 3.1, we review the result of the representability of quot-functors on Deligne-Mumford stacks. In the rest of this section, we work on a projective Deligne-Mumford stack with a given generating sheaf $\mathcal{E}$ on $\mathcal{X}$ and a given polarization $\mathcal{O}_X(1)$ on $X$. We give the definitions of saturations (\S 3.2), modified Hilbert polynomials, $p$-stability condition (\S 3.3), Harder-Narasimhan filtrations and Jordan-H\"{o}lder filtrations (\S 3.4), and we review some properties of boundedness (\S 3.5 and \S 3.6). Most of these properties can be found in \cite[\S 2-\S 4]{Nir} and \cite[\S 2]{BSP}. Otherwise we will give a proof (e.g. Corollary \ref{304}). In \S 3.7, we review the geometric invariant theory (GIT).

The moduli space of $p$-semistable coherent sheaves on $\mathcal{X}$ is constructed as the GIT quotient of a quot-scheme \cite[Theorem 6.22]{Nir}. In \S 3.8, we construct the moduli space of ($p$-semistable) sheaves on a projective Deligne-Mumford stack $\mathcal{X}$ (over a scheme $S$ of finite type) and give a slightly different proof.

\subsection{Quot-Functors}
Let $S$ be an algebraic space, which is locally of finite type over an algebraically closed field $k$, and let $\mathcal{X}$ be a separated and locally finitely-presented Deligne-Mumford stack over $S$. Denote by $(\text{Sch}/S)$ the category of $S$-schemes with respect to the big \'etale topology. We take a coherent sheaf $\mathcal{G}$, and define the moduli problem
\begin{align*}
\widetilde{{\rm Quot}}(\mathcal{G},\mathcal{X}): (\text{Sch}/S)^{{\rm op}} \rightarrow \text{Set}
\end{align*}
as follows. For each $S$-scheme $T$, we define $\widetilde{{\rm Quot}}(\mathcal{G},\mathcal{X})(T)$ to be the set of quotients $[\mathcal{G}_T \rightarrow \mathcal{F}_T]$ such that
\begin{enumerate}
\item $\mathcal{F}_T$ is a locally finitely-presneted quasi-coherent $\mathcal{O}_{\mathcal{X}_T}$-module;
\item $\mathcal{F}_T$ is flat over $T$;
\item the support of $\mathcal{F}_T$ is proper over $T$.
\end{enumerate}
The moduli problem $\widetilde{{\rm Quot}}(\mathcal{G},\mathcal{X})$ is called the \emph{quot-functor}. The quot-functor $\widetilde{{\rm Quot}}(\mathcal{G},\mathcal{X})$ has a natural stack structure.

Artin proved that the quot-functor $\widetilde{{\rm Quot}}(\mathcal{G},\mathcal{X})$ is represented by a separated and locally finitely-presented algebraic space over $S$ when $\mathcal{X}$ is an algebraic space \cite{Art}. Olsson and Starr generalized this result to Deligne-Mumford stacks.
\begin{thm}[Theorem 1.1 in \cite{OlSt}]\label{301}
With respect to the above notation, the quot-functor $\widetilde{{\rm Quot}}(\mathcal{G},\mathcal{X})$ is represented by an algebraic space which is separated and locally finitely presented over $S$.
\end{thm}
Denote by ${\rm Quot}(\mathcal{G},\mathcal{X})$ the algebraic space representing the quot-functor $\widetilde{{\rm Quot}}(\mathcal{G},\mathcal{X})$. The algebraic space ${\rm Quot}(\mathcal{G},\mathcal{X})$ is called the \emph{quot-space}. If there is no ambiguity, we would like to omit $\mathcal{X}$, and use the notations $\widetilde{{\rm Quot}}(\mathcal{G})$ for the quot-functor and ${\rm Quot}(\mathcal{G})$ for the quot-space.

Now let $\mathcal{X}$ be a projective Deligne-Mumford stack. We choose a polarization $\mathcal{O}_X(1)$ on $X$ and a generating sheaf $\mathcal{E}$ on $\mathcal{X}$. The functor $F_{\mathcal{E}}$ induces a natural transformation
\begin{align*}
F_{\mathcal{E}}:\widetilde{{\rm Quot}}(\mathcal{G},\mathcal{X}) \rightarrow \widetilde{{\rm Quot}}(F_{\mathcal{E}}(\mathcal{G}),X),
\end{align*}
which is proved to be an monomorphism.
\begin{lem}[Lemma 6.1 in \cite{OlSt}]\label{302}
For each $T \in ({\rm Sch}/S)$, the map
\begin{align*}
F_{\mathcal{E}}(T):\widetilde{{\rm Quot}}(\mathcal{G},\mathcal{X})(T) \rightarrow \widetilde{{\rm Quot}}(F_{\mathcal{E}}(\mathcal{G}),X)(T)
\end{align*}
is injective.
\end{lem}

It is well-known that the quot-functor $\widetilde{{\rm Quot}}(F_{\mathcal{E}}(\mathcal{G}),X)$ is representable by disjoint union of schemes, which are parameterized by integer polynomials (Hilbert polynomials). Therefore, we have the following theorem.

\begin{thm}[Theorem 1.5 and Theorem 4.4 in \cite{OlSt}]\label{3021}
Let $S$ be an affine scheme, and let $\mathcal{X}$ be a projective (resp. quasi-projective) Deligne-Mumford stack. The connected components of ${\rm Quot}(\mathcal{G},\mathcal{X})$ are projective (resp. quasi-projective) $S$-schemes, which are parameterized by integer polynomials.
\end{thm}

Nironi showed that the proof of the above theorem also works for noetherian schemes $S$ of finite type (see \cite[Theorem 2.17]{Nir}).

\subsection{Pure Sheaves and Saturations}
In this subsection, we discuss pure sheaves over a projective Deligne-Mumford stack, of which the definition is similar to that of pure sheaves over a scheme \cite[\S 1.1]{HuLe} and the materials can be found in \cite[\S 3]{Nir}. We also define the saturation of a subsheaf $\mathcal{F}' \subseteq \mathcal{F}$ and prove that the saturation is preserved under the functor $F_\mathcal{E}$.

Let $\mathcal{X}$ be a projective Deligne-Mumford stack with moduli space $X$ over an algebraically closed field $k$. We fix a polarization $\mathcal{O}_X(1)$ on $X$ and a generating sheaf $\mathcal{E}$ on $\mathcal{X}$. Let $\mathcal{F}$ be a coherent sheaf on $\mathcal{X}$. The \emph{support} of $\mathcal{F}$ is defined to be the closed substack associated to the ideal
\begin{align*}
0 \rightarrow \mathcal{I} \rightarrow \mathcal{O}_{\mathcal{X}} \rightarrow \mathcal{E}nd_{\mathcal{O}_{\mathcal{X}}}(\mathcal{F}).
\end{align*}
The dimension of the support $\mathcal{I}$ is defined as the \emph{dimension} of the coherent sheaf $\mathcal{F}$.

Let
\begin{align*}
Y \times_{\mathcal{X}} Y \overset{s}{\underset{t}{\rightrightarrows}} Y \xrightarrow{\pi'} \mathcal{X}
\end{align*}
be an \'etale presentation of $\mathcal{X}$. By the flatness of the maps $s,t,\pi'$, we have
\begin{align*}
\dim(\mathcal{F})=\dim(\pi'^*\mathcal{F}).
\end{align*}

A coherent sheaf $\mathcal{F}$ is \emph{pure of dimension $d$}, if for every nontrivial subsheaf $\mathcal{F}'$, the support of $\mathcal{F}'$ is of dimension $d$.

Similar to the case of schemes, for any coherent sheaf $\mathcal{F}$ over $\mathcal{X}$, we have the filtration
\begin{align*}
0 \subseteq  T_0(\mathcal{F}) \subseteq \dots \subseteq T_d(\mathcal{F})=\mathcal{F},
\end{align*}
where $T_i(\mathcal{F})/T_{i-1}(\mathcal{F})$ is pure of dimension $i$ or zero. This filtration is called the \emph{torsion filtration} of $\mathcal{F}$, and denote by
\begin{align*}
\mathcal{F}_{\rm tor}:=T_{d-1}(\mathcal{F})
\end{align*}
the \emph{torsion part} of $\mathcal{F}$.

Given a projective Deligne-Mumford stack $\mathcal{X}$, the pureness of a coherent sheaf $\mathcal{F}$ over $\mathcal{X}$ is preserved by the functor $F_{\mathcal{E}}$ as explained in the following corollary.
\begin{cor}[Corollary 3.7 in \cite{Nir}, Proposition 2.22 in \cite{BSP}]\label{303}
Let $\mathcal{X}$ be a projective Deligne-Mumford stack, and let $\mathcal{F}$ be a coherent sheaf over $\mathcal{X}$. The torsion filtration of $\mathcal{F}$
\begin{align*}
0 \subseteq  T_0(\mathcal{F}) \subseteq \dots \subseteq T_d(\mathcal{F})=\mathcal{F},
\end{align*}
is sent to the torsion filtration of $F_{\mathcal{E}}(\mathcal{F})$ under the functor $F_{\mathcal{E}}$. In other words, $F_{\mathcal{E}} (T_i(\mathcal{F}))=T_i(F_{\mathcal{E}}(\mathcal{F}))$.
\end{cor}

Given a coherent sheaf $\mathcal{F}$ over $\mathcal{X}$, let $\mathcal{F}'$ be a subsheaf of $\mathcal{F}$. The \emph{saturation} of $\mathcal{F}'$ is the minimal subsheaf $\mathcal{F}'_{\rm sat}$ containing $\mathcal{F}'$ such that $\mathcal{F}/\mathcal{F}'_{\rm sat}$ is pure of dimension $\dim(\mathcal{F})$ or zero. From the definition, the saturation of a given subsheaf $\mathcal{F}' \subseteq \mathcal{F}$ is exactly the kernel of the surjection
\begin{align*}
\mathcal{F} \rightarrow \mathcal{F}/\mathcal{F}' \rightarrow (\mathcal{F}/\mathcal{F}')/ T_{d-1}(\mathcal{F}/\mathcal{F}').
\end{align*}
If $\mathcal{F}'=\mathcal{F}'_{\rm sat}$, then we say that $\mathcal{F}'$ is \emph{saturated}. The saturation of a subsheaf is also preserved by the functor $F_{\mathcal{E}}$ as proved in the following corollary.

\begin{cor}\label{304}
Let $\mathcal{F}$ be a sheaf over $\mathcal{X}$ and let $\mathcal{F}'$ be a subsheaf of $\mathcal{F}$. The saturation $\mathcal{F}'_{\rm sat}$ of $\mathcal{F}'$ is preserved by $F_{\mathcal{E}}$, i.e.
\begin{align*}
F_{\mathcal{E}}(\mathcal{F}'_{\rm sat}) = (F_{\mathcal{E}}(\mathcal{F}'))_{\rm sat}.
\end{align*}
\end{cor}

\begin{proof}
There are two natural projections
\begin{align*}
j: \mathcal{F} \rightarrow \mathcal{F}/\mathcal{F}', \quad p: \mathcal{F}/\mathcal{F}' \rightarrow (\mathcal{F}/\mathcal{F}')/ T_{d-1}(\mathcal{F}/\mathcal{F}').
\end{align*}
Let $i: \mathcal{F} \rightarrow (\mathcal{F}/\mathcal{F}')/ T_{d-1}(\mathcal{F}/\mathcal{F}')$ be the composition of the above two projections. By the definition of the saturation, we have the following short exact sequence
\begin{align*}
0 \rightarrow \mathcal{F}'_{\rm sat} \rightarrow \mathcal{F} \xrightarrow{i} (\mathcal{F}/\mathcal{F}')/ T_{d-1}(\mathcal{F}/\mathcal{F}') \rightarrow 0.
\end{align*}
Thus we have the following diagram
\begin{center}
\begin{tikzcd}
  & 0 \arrow[d] & 0 \arrow[d] & & \\
0 \arrow[r] & \mathcal{F}' \arrow[d] \arrow[r,"\cong"]   & \mathcal{F}' \arrow[d] \arrow[r] & 0 \arrow[d]&   \\
0 \arrow[r]& \mathcal{F}'_{\rm sat} \arrow[d] \arrow[r] & \mathcal{F} \arrow[d,"j"] \arrow[r,"i"] & (\mathcal{F}/\mathcal{F}')/ T_{d-1}(\mathcal{F}/\mathcal{F}') \arrow[d,"\cong"] \arrow[r] & 0 \\
0 \arrow[r]& T_{d-1}(\mathcal{F}/\mathcal{F}') \arrow[d] \arrow[r] & \mathcal{F}/\mathcal{F}' \arrow[d] \arrow[r,"p"] & (\mathcal{F}/\mathcal{F}')/ T_{d-1}(\mathcal{F}/\mathcal{F}') \arrow[d] \arrow[r] & 0 \\
& 0 & 0 & 0 &
\end{tikzcd}
\end{center}
where the rows and columns are exact sequences.

The above diagram gives the following short exact sequence
\begin{align*}
0 \rightarrow \mathcal{F}' \rightarrow \mathcal{F}'_{\rm sat} \rightarrow T_{d-1}(\mathcal{F}/\mathcal{F}') \rightarrow 0.
\end{align*}
We apply the functor $F_{\mathcal{E}}$ to the above sequence and we have
\begin{center}
\begin{tikzcd}
0 \arrow[r]& F_{\mathcal{E}}(\mathcal{F}') \arrow[d,"\cong"] \arrow[r] & F_{\mathcal{E}}(\mathcal{F}'_{\rm sat}) \arrow[d,dashrightarrow] \arrow[r] & F_{\mathcal{E}}(T_{d-1}(\mathcal{F}/\mathcal{F}')) \arrow[d,"\cong"] \arrow[r] & 0 \\
0 \arrow[r] & F_{\mathcal{E}}(\mathcal{F}') \arrow[r] & (F_{\mathcal{E}}(\mathcal{F}'))_{\rm sat} \arrow[r] & T_{d-1}(F_{\mathcal{E}}(\mathcal{F}/\mathcal{F}')) \arrow[r] & 0
\end{tikzcd}
\end{center}
The first row is the short exact sequence by applying the functor $F_{\mathcal{E}}$ to the previous one and the second row is the short exact sequence for the sheaf $F_{\mathcal{E}}(\mathcal{F}')$. The two objects in the first column are the same. In the third column, we have $F_{\mathcal{E}}(T_{d-1}(\mathcal{F}/\mathcal{F}')) \cong T_{d-1}(F_{\mathcal{E}}(\mathcal{F}/\mathcal{F}'))$ by Corollary \ref{303}. Thus we have
\begin{align*}
F_{\mathcal{E}}(\mathcal{F}'_{\rm sat}) = (F_{\mathcal{E}}(\mathcal{F}'))_{\rm sat}.
\end{align*}
\end{proof}

\subsection{Modified Hilbert Polynomials and $p$-Stability Conditions}
With respect to the same conditions and notations in \S 3.2, let $\mathcal{F}$ be a coherent sheaf on $\mathcal{X}$. The \emph{modified Hilbert polynomial} $P_{\mathcal{E}}(\mathcal{F},m)$ is defined as
\begin{align*}
P_{\mathcal{E}}(\mathcal{F},m)=\chi(\mathcal{X},\mathcal{F} \otimes \mathcal{E}^{\vee} \otimes \pi^* \mathcal{O}_X(m)), \quad m \gg 0.
\end{align*}
Since the functor $\pi_*: {\rm QCoh}(\mathcal{X}) \rightarrow {\rm QCoh}(X)$ is exact, the modified Hilbert polynomial can be written as the classical Hilbert polynomial for the coherent sheaf $F_{\mathcal{E}}(\mathcal{F})(m)$ over $X$,
\begin{align*}
P_{\mathcal{E}}(\mathcal{F},m)=\chi(X,F_{\mathcal{E}}(\mathcal{F})(m)),\quad  m \gg 0.
\end{align*}
Based on this property, the modified Hilbert polynomial is additive on short exact sequences.  Also, if $\mathcal{F}$ is pure of dimension $d$, the function $P_{\mathcal{E}}(\mathcal{F},m)$ is a polynomial (with respect to the variable $m$) and we can write it in the following way
\begin{align*}
P_{\mathcal{E}}(\mathcal{F},m)=\sum_{i=0}^d \alpha_{\mathcal{E},i}(\mathcal{F})\frac{m^i}{i!}.
\end{align*}
We use the notation $P_{\mathcal{E}}(\mathcal{F})$ for the modified Hilbert polynomial of $\mathcal{F}$. Given an integer polynomial $P$, if we claim that $P$ is the modified Hilbert polynomial of $\mathcal{F}$, then it means that $P=P_{\mathcal{E}}(\mathcal{F})$.

The \emph{reduced modified Hilbert polynomial} $p_{\mathcal{E}}(\mathcal{F})$ is a monic  polynomial with rational coefficients defined as
\begin{align*}
p_{\mathcal{E}}(\mathcal{F})=\frac{P_{\mathcal{E}}(\mathcal{F})}{\alpha_{\mathcal{E},d}(\mathcal{F})}.
\end{align*}
The \emph{slope} of $\mathcal{F}$ of dimension $d$ is
\begin{align*}
\hat{\mu}_{\mathcal{E}}(\mathcal{F})=\frac{\alpha_{\mathcal{E},d-1}(\mathcal{F})}{\alpha_{\mathcal{E},d}(\mathcal{F})}.
\end{align*}
The \emph{rank} (or \emph{multiplicity}) of $\mathcal{F}$ is
\begin{align*}
r(\mathcal{F}) = \frac{\alpha_{\mathcal{E},d}(\mathcal{F})}{\alpha_{\mathcal{E},d}(\mathcal{O}_{\mathcal{X}})}.
\end{align*}

A pure coherent sheaf $\mathcal{F}$ is \emph{$p_{\mathcal{E}}$-semistable} (resp. \emph{$p_{\mathcal{E}}$-stable}), if for every proper subsheaf $\mathcal{F}' \subseteq \mathcal{F}$ we have
\begin{align*}
p_{\mathcal{E}}(\mathcal{F}') \leq p_{\mathcal{E}}(\mathcal{F})\quad (\text{resp. } p_{\mathcal{E}}(\mathcal{F}') < p_{\mathcal{E}}(\mathcal{F})).
\end{align*}
A $p_{\mathcal{E}}$-semistable sheaf $\mathcal{F}$ is called \emph{polystable}, if $\mathcal{F}$ is the direct sum of $p_{\mathcal{E}}$-stable sheaves. If there is no ambiguity, we use the terminologies \emph{$p$-semistable sheaf}, \emph{$p$-stable sheaf} and \emph{$p$-polystable sheaf} for simplicity.

Now fixing an integer polynomial $P$ (as the modified Hilbert polynomial), we define a new moduli problem
\begin{align*}
\widetilde{{\rm Quot}}(\mathcal{G},P) : (\text{Sch}/S)^{{\rm op}} \rightarrow \text{Set},
\end{align*}
which can be considered as a sub-functor of the quot-functor $\widetilde{{\rm Quot}}(\mathcal{G})$ with respect to the given polynomial $P$. Given an $S$-scheme $T$, we define $\widetilde{{\rm Quot}}(\mathcal{G},P)(T)$ to be the set of quotients $[(\mathcal{G}_T \rightarrow \mathcal{F}_T)]$ such that
\begin{enumerate}
\item $(\mathcal{G}_T \rightarrow \mathcal{F}_T) \in \widetilde{{\rm Quot}}(\mathcal{G})(T)$;
\item for each point $t \in T$, the modified Hilbert polynomial of $(\mathcal{F}_T)_t$ is $P$.
\end{enumerate}
The functor $\widetilde{{\rm Quot}}(\mathcal{G},P)$ is also represented by an algebraic space by Theorem \ref{301}. Denote this algebraic space by ${\rm Quot}(\mathcal{G},P)$. If $S$ is an affine scheme (or a noetherian scheme of finite type), the algebraic space ${\rm Quot}(\mathcal{G},P)$ is a projective $S$-scheme, which is a connected component of ${\rm Quot}(\mathcal{G})$ (see Theorem \ref{3021}). Furthermore, any connected component of ${\rm Quot}(\mathcal{G})$ is of the form ${\rm Quot}(\mathcal{G},P)$ \cite[\S 4]{OlSt}.

\subsection{Harder-Narasimhan Filtrations and Jordan-H\"{o}lder Filtrations}
The construction of the Harder-Narasimhan filtration and the Jordan-H\"{o}lder filtration of a coherent sheaf over a scheme is well-known \cite[\S 1.3 and \S 1.5]{HuLe}. Since the functor $F_{\mathcal{E}}$ is exact and the modified Hilbert polynomial $P_{\mathcal{E}}$ is additive for short exact sequences, the construction of these two filtrations can be generalized to coherent sheaves over a projective Deligne-Mumford stack $\mathcal{X}$. We give the definition and review some results about these two filtrations in this subsection. Details can be found in \cite[\S 1.3 and \S 1.5]{HuLe} and \cite[\S 3.4]{Nir}.

\subsubsection*{\textbf{Harder-Narasimhan Filtrations}}
Let $\mathcal{F}$ be a pure sheaf of dimension $d$ on $\mathcal{X}$. A \emph{Harder-Narasimhan Filtration} of $\mathcal{F}$ is an increasing filtration
\begin{align*}
0 = {\rm HN}_0(\mathcal{F}) \subseteq {\rm HN}_1(\mathcal{F}) \subseteq \dots \subseteq {\rm HN}_l(\mathcal{F})= \mathcal{F},
\end{align*}
satisfying the following two conditions:
\begin{enumerate}
\item The factors $gr_i^{\rm HN}(\mathcal{F})={\rm HN}_{i}(\mathcal{F}) / {\rm HN}_{i-1}(\mathcal{F})$ are $p$-semistable sheaves of dimension $d$ for $1 \leq i \leq l$.
\item Denote by $p_i(m)$ the reduced modified Hilbert polynomial $p_{\mathcal{E}}(gr_i^{\rm HN}(\mathcal{F}),m)$ such that
    \begin{align*}
        p_{\rm max}(\mathcal{F}):=p_1 > \dots > p_l=: p_{\rm min}(\mathcal{F}).
    \end{align*}
\end{enumerate}

For every pure coherent sheaf $\mathcal{F}$ on $\mathcal{X}$, there is a unique Harder-Narasimhan filtration of $\mathcal{F}$ \cite[Theorem 3.22]{Nir}. The proof of the Harder-Narasimhan filtration on $\mathcal{X}$ is similar to the case over a scheme \cite[\S 1.3]{HuLe}. The key to construct the Harder-Narasimhan filtration of $\mathcal{F}$ is to prove the existence and uniqueness of the \emph{destabilizing subsheaf} of $\mathcal{F}$. We only give the definition of the destabilizing subsheaf, and we refer the reader to \cite{Nir} for the proofs.

Let $\mathcal{F}$ be a purely $d$-dimensional coherent sheaf on $\mathcal{X}$. There is a subsheaf $\mathcal{F}_{\rm de} \subseteq \mathcal{F}$ such that
\begin{enumerate}
\item for all subsheaves $\mathcal{F}' \subseteq \mathcal{F}$ we have $p(\mathcal{F}_{\rm de}) \geq p(\mathcal{F}')$;
\item if $p_{\mathcal{E}}(\mathcal{F}_{\rm de}) = p_{\mathcal{E}}(\mathcal{F}')$, we have $\mathcal{F}' \subseteq \mathcal{F}_{\rm de}$.
\end{enumerate}
The subsheaf $\mathcal{F}_{\rm de}$ is called the \emph{destabilizing subsheaf} of $\mathcal{F}$. Note that $\mathcal{F}_{\rm de}$ is $p$-semistable, saturated and uniquely determined.

\subsubsection*{\textbf{Jordan-H\"{o}lder Filtrations}}
Let $\mathcal{F}$ be a $p$-semistable sheaf on $\mathcal{X}$ with reduced modified Hilbert polynomial $p_{\mathcal{E}}(\mathcal{F})$. A \emph{Jordan-H\"{o}lder Filtration} of $\mathcal{F}$ is an increasing filtration
\begin{align*}
0 = {\rm JH}_0(\mathcal{F}) \subseteq {\rm JH}_1(\mathcal{F}) \subseteq \dots \subseteq {\rm JH}_l(\mathcal{F})= \mathcal{F}
\end{align*}
such that the factor $gr_i^{\rm JH}(\mathcal{F})={\rm JH}_{i}(\mathcal{F}) / {\rm JH}_{i-1}(\mathcal{F})$ is stable with reduced modified Hilbert polynomial $p_{\mathcal{E}}(\mathcal{F},m)$ for $1 \leq i \leq l$.

For every semistable sheaf $\mathcal{F}$ on $\mathcal{X}$, there is a Jordan-H\"{o}lder filtration of $\mathcal{F}$ and the graded sheaf $gr^{\rm JH}(\mathcal{F}):=\bigoplus_i gr_i^{\rm JH}(\mathcal{F})$ does not depend on the choice of the Jordan-H\"{o}lder filtration \cite[Theorem 3.23]{Nir}.

Two $p$-semistable sheaves $\mathcal{F}_1$ and $\mathcal{F}_2$ with the same reduced modified Hilbert polynomial are \emph{$S$-equivalent} if the graded sheaves of the Jordan-H\"older filtrations are isomorphic, i.e.
\begin{align*}
gr^{\rm JH}(\mathcal{F}_1) \cong gr^{\rm JH}(\mathcal{F}_2).
\end{align*}

\begin{rem}\label{305}
Unlike the pureness and the saturation, the Harder-Narasimhan filtration and the Jordan-H\"{o}lder filtration are not preserved by the functor $F_{\mathcal{E}}$. More precisely, let
\begin{align*}
0 \subseteq {\rm HN}_0(\mathcal{F}) \subseteq {\rm HN}_1(\mathcal{F}) \subseteq \dots \subseteq {\rm HN}_l(\mathcal{F})= \mathcal{F}
\end{align*}
be the Harder-Narasimhan filtration of the sheaf $\mathcal{F}$. The filtration
\begin{align*}
0 \subseteq F_{\mathcal{E}}({\rm HN}_0(\mathcal{F})) \subseteq F_{\mathcal{E}}({\rm HN}_1(\mathcal{F})) \subseteq \dots \subseteq F_{\mathcal{E}}({\rm HN}_l(\mathcal{F}))= F_{\mathcal{E}}(\mathcal{F})
\end{align*}
is not the Harder-Narasimhan filtration of $F_{\mathcal{E}}(\mathcal{F})$ in general. The same situation holds for the Jordan-H\"{o}lder filtration. Therefore the functor $F_{\mathcal{E}}$ does not preserve the $p$-semistability (resp. $p$-stability). In other words, if $\mathcal{F}$ is $p$-semistable (resp. $p$-stable), the coherent sheaf $F_{\mathcal{E}}(\mathcal{F})$ may not be $p$-semistable (resp. $p$-stable). A careful discussion is in \cite[Remark 3.24]{Nir}.
\end{rem}

\subsection{Boundedness of Coherent Sheaves \Romannum{1}}
In this subsection, we first review the definition and some properties of the boundedness of coherent sheaves over a scheme. Then we extend these properties to coherent sheaves over a projective Deligne-Mumford stack. We use the notation $X$ for a scheme over $S$ in the first part of this subsection. In the second part of this subsection, $X$ will be the coarse moduli space of the given projective Deligne-Mumford stack $\mathcal{X}$. Some materials can be found in \cite[\S 1.7]{HuLe} and \cite[\S 4]{Nir}.

\subsubsection*{\textbf{Boundedness over Schemes}}
Let $F$ be a coherent sheaf over a scheme $X$. The sheaf $F$ is \emph{$m$-regular} if we have $H^i(X,F(m-i))=0$. Denote by \emph{${\rm reg}(F)$} the least integer $m$ such that $F$ is $m$-regular. The number ${\rm reg}(F)$ is called the \emph{regularity of $F$}, or more precisely, the \emph{Mumford-Castelnuovo regularity of $F$}. If $F$ is a $m$-regular coherent sheaf on $X$, for $n \geq m$, we have
\begin{itemize}
\item $F$ is $n$-regular;
\item $H^0(F(n)) \otimes H^0(\mathcal{O}_X(1)) \rightarrow H^0(F(n+1))$ is surjective;
\item $F(n)$ is generated by global sections.
\end{itemize}
Let $F_1$ and $F_2$ be two coherent sheaves over $X$. If $F_i$ is $m_i$-regular, then there is an lower bound of the regularity for the tensor product $F_1 \otimes F_2$.

\begin{lem}[Corollary 3.4 in \cite{Cavi}]\label{306}
Let $M$ be an $m$-regular finitely generated graded $R$-module and let $N$ be an $n$-regular finitely generated graded $R$-module such that $\dim {\rm Tor}_1^R(M,N) \leq 1$. Then $M \otimes N$ is $(m+n)$-regular.
\end{lem}

This result can be easily generalized to the coherent sheaves. We omit the proof for the following lemma.
\begin{lem}\label{307}
Let $F_1$ be an $m$-regular coherent sheaf over $X$ and let $F_2$ be an $n$-regular coherent sheaf over $X$ such that $\dim {\rm Tor}_1^R(F_1,F_2) \leq 1$. Then $F_1 \otimes F_2$ is $(m+n)$-regular.
\end{lem}

Now we will give the definition of the boundedness. A family $\widetilde{\mathfrak{F}}$ of isomorphism classes of coherent sheaves over $X$ is \emph{bounded}, if there is an $S$-scheme $T$ of finite type and a coherent $\mathcal{O}_{X_T}$-sheaf $\mathfrak{F}$ such that
\begin{align*}
\widetilde{\mathfrak{F}} \subseteq \{\mathfrak{F}_t \text{ }|\text{ }t \text{ a closed point in }T\}.
\end{align*}
There are several properties equivalent to the property of the boundedness of the family $\widetilde{\mathfrak{F}}$. We list them as follows:
\begin{itemize}
\item The family $\widetilde{\mathfrak{F}}$ is bounded.
\item The set of Hilbert polynomials $\{P(F)| F \in \widetilde{\mathfrak{F}}\}$ is finite and there is a non-negative integer $m$ such that ${\rm reg}(F) \leq m$ for every $F \in \widetilde{\mathfrak{F}}$. In other words, for any $F \in \widetilde{\mathfrak{F}}$, the coherent sheaf $F$ is $m$-regular.
\item The set of Hilbert polynomials $\{P(F)| F \in \widetilde{\mathfrak{F}}\}$ is finite and there is a coherent sheaf $G$ such that all $F \in \widetilde{\mathfrak{F}}$ admit surjective morphisms $G \rightarrow F$.
\item The coherent sheaves in $\widetilde{\mathfrak{F}}$ have the same Hilbert polynomial $P$. There are constants $C_i$, $i=0, \dots , d=\deg(P)$ such that for every $F \in \widetilde{\mathfrak{F}}$ there is an $F$-regular sequence of hyperplane sections $H_1,\dots,H_d$ such that $h^0(F|_{\cap_{j \leq i} H_j})\leq C_i$. This property is known as the \emph{Kleiman Criterion} \cite[Theorem 1.7.8]{HuLe}.
\end{itemize}
By the Kleiman Criterion, we have the following theorem.
\begin{thm}[Theorem 3.3.7 in \cite{HuLe}]\label{308}
Let $f: X \rightarrow S$ be a projective morphism of schemes of finite type over $k$. Let $\mathcal{O}_X(1)$ be an $f$-ample line bundle. We fix a polynomial of degree $d$ and a rational number $\mu_0$. Then the family of purely $d$-dimensional sheaves with Hilbert polynomial $P$ on the fibers of $f$ such that the maximal slope $\hat{\mu}_{\rm max} \leq \mu_0$ is bounded. In particular, the family of semistable sheaves with Hilbert polynomial $P$ is bounded.
\end{thm}

\subsubsection*{\textbf{Boundedness over Stacks}}
Let $\mathcal{X}$ be a projective Deligne-Mumford stack over $S$ and denote by $X$ its coarse moduli space. Let $\mathcal{E}$ be a generating sheaf of $\mathcal{X}$. Let $\mathcal{F}$ be a coherent sheaf over $\mathcal{X}$. The regularity of $\mathcal{F}$ (with respect to $\mathcal{E}$) is defined to be the regularity of $F_{\mathcal{E}}(\mathcal{F})$, i.e.
\begin{align*}
    {\rm reg}_{\mathcal{E}}(\mathcal{F}):={\rm reg}(F_{\mathcal{E}}(\mathcal{F})).
\end{align*}
If there is no ambiguity, a coherent sheaf $\mathcal{F}$ is \emph{$m$-regular} if $F_{\mathcal{E}}(\mathcal{F})$ is $m$-regular.

With respect to this definition, we have the following lemma about the regularity of coherent sheaves over $\mathcal{X}$.
\begin{lem}\label{309}
Let $\mathcal{F}$ be a coherent sheaf over $\mathcal{X}$, and let $\mathcal{E}$ be a generating sheaf. Suppose that $\pi_*\mathcal{F}$ is $m_0$-regular. Then, there exists an integer $m$ such that $\mathcal{F}$ is $m$-regular.
\end{lem}

\begin{proof}
Since $\mathcal{E}$ is a generating sheaf over $\mathcal{X}$, $\pi_*\mathcal{E}^{\vee}$ is a locally free sheaf. There exists an integer $m_{\mathcal{E}}$ such that $\pi_*\mathcal{E}^{\vee}$ is $m_{\mathcal{E}}$-regular. By assumption, we know that $\pi_*\mathcal{F}$ is $m_0$-regular. Take $m=m_0+m_{\mathcal{E}}$. By Lemma \ref{307}, we know that $\pi_*\mathcal{F} \otimes \pi_* \mathcal{E}^{\vee}$ is $m$-regular. By the exactness of the functor $\pi_*$, $F_{\mathcal{E}}(\mathcal{F})$ is $m$-regular.
\end{proof}

A family $\widetilde{\mathfrak{F}}$ of coherent sheaves on $\mathcal{X}$ is \emph{bounded}, if there is an $S$-scheme $T$ of finite type and a coherent sheaf $\mathfrak{F}$ on $\mathcal{X} \times_S T$ such that
\begin{align*}
\widetilde{\mathfrak{F}} \subseteq \{\mathfrak{F}_t \text{ }|\text{ }t \text{ a closed point in }T\}.
\end{align*}
There are some properties equivalent to the property of the boundedness in the version of stacks \cite[\S 4]{Nir}. We list some of them as follows:
\begin{itemize}
\item The family $\widetilde{\mathfrak{F}}$ is bounded.
\item The set of modified Hilbert polynomials $P_{\mathcal{E}}(\mathcal{F})$ for $\mathcal{F} \in \widetilde{\mathfrak{F}}$ is finite and there is a non-negative integer $m$ such that $\mathcal{F}$ is $m$-regular for every $F \in \widetilde{\mathfrak{F}}$.
\item The set of modified Hilbert polynomials $\{P_{\mathcal{E}}(\mathcal{F})| \mathcal{F} \in \widetilde{\mathfrak{F}}\}$ is finite and there is a coherent sheaf $\mathcal{G}$ such that all $\mathcal{F} \in \widetilde{\mathfrak{F}}$ admit surjective morphisms $\mathcal{G} \rightarrow \mathcal{F}$.
\item The family $F_{\mathcal{E}}(\widetilde{\mathfrak{F}})$ is bounded.
\end{itemize}
The above equivalent conditions tell us that if we want to prove the family $\widetilde{\mathfrak{F}}$ of $p$-semistable coherent sheaves over $\mathcal{X}$ with the same modified Hilbert polynomial $P$ is bounded, it is equivalent to prove that the corresponding family $F_{\mathcal{E}}(\widetilde{\mathfrak{F}})$ over $X$ is bounded. By Theorem \ref{308}, if we can prove that the maximal slope $\hat{\mu}_{\rm max}(F_{\mathcal{E}}(\mathcal{F}))$ of the family $F_{\mathcal{E}}(\widetilde{\mathfrak{F}})$ of coherent sheaves is bounded, then the family $\widetilde{\mathfrak{F}}$ is bounded.

Let $\mathcal{F}$ be a $p$-semistable sheaf on $\mathcal{X}$. We choose an integer $\widetilde{m}$ such that $\pi_* \mathcal{E}nd_{\mathcal{O}_{\mathcal{X}}}(\mathcal{E})(\widetilde{m})$ is generated by global sections. Nironi proved the following inequality \cite[Proposition 4.24]{Nir}
\begin{align*}
\hat{\mu}_{\rm max}(F_{\mathcal{E}}(\mathcal{F})) \leq \hat{\mu}_{\mathcal{E}}(\mathcal{F})+\widetilde{m} \deg(\mathcal{O}_X(1)).
\end{align*}
This inequality together with Theorem \ref{308} gives us the following result.
\begin{thm}[Theorem 4.27 in \cite{Nir}]\label{310}
Let $f: \mathcal{X} \rightarrow S$ be a family of projective Deligne-Mumford stacks with a family of moduli spaces $X \rightarrow S$ over an algebraically closed field $k$. Let $\mathcal{E}$ be a generating sheaf of $\mathcal{X}$, and let $\mathcal{O}_{X}(1)$ be an $f$-ample line bundle. We fix an integral polynomial $P$ of degree $d$ and a rational number $\mu_0$. Then the family $\widetilde{\mathfrak{F}}$ of purely $d$-dimensional sheaves with modified Hilbert polynomial $P$ on the fibers of $f$ such that the maximal slope $\hat{\mu}_{\rm max}( \mathcal{F}) \leq \mu_0$ is bounded. In particular, the family of $p$-semistable purely $d$-dimensional sheaves with modified Hilbert polynomial $P$ is bounded.
\end{thm}

\subsection{Boundedness of Coherent Sheaves \Romannum{2}}
In this subsection, we review the result of the upper bound for the number of global sections of $p$-semistable sheaves on $\mathcal{X}$.

\begin{lem}[Corollary 4.30 in \cite{Nir}]\label{311}
For any pure $p$-semistable sheaf $\mathcal{F}$ of dimension $d$ on a projective stack $\mathcal{X}$, we have
\begin{align*}
h^0(\mathcal{X},\mathcal{F} \otimes \mathcal{E}^{\vee})  \leq
\begin{cases}
r{\hat{\mu}_{\mathcal{E}}(\mathcal{F}) + \widetilde{m} \deg(\mathcal{O}_X(1)) + r^2+f(r)+\frac{d-1}{2}  \choose d}, & \text{ if } \hat{\mu}_{\rm max}(F_{\mathcal{E}}(\mathcal{F})) \geq \frac{d+1}{2} - r^2\\
0, & \text{ if } \hat{\mu}_{\rm max}(F_{\mathcal{E}}(\mathcal{F})) < \frac{d+1}{2} - r^2
\end{cases}
\end{align*}
where $r$ is the rank of $F_{\mathcal{E}}(\mathcal{F})$, $\widetilde{m}$ is an integer such that $\pi_* \mathcal{E}nd_{\mathcal{O}_{\mathcal{X}}}(\mathcal{E})(\widetilde{m})$ is generated by global sections and $f(r)=-1+ \sum\limits_{i=1}^r \frac{1}{i}$.
\end{lem}

Note that the condition $ \hat{\mu}_{\rm max}(F_{\mathcal{E}}(\mathcal{F})) < \frac{d+1}{2} - r^2$ is equivalent to
\begin{align*}
\hat{\mu}_{\mathcal{E}}(\mathcal{F}) + \widetilde{m} \deg(\mathcal{O}_X(1)) + r^2+f(r)+\frac{d-1}{2} < d
\end{align*}
by applying the inequality $\hat{\mu}_{\rm max}(F_{\mathcal{E}}(\mathcal{F})) \leq \hat{\mu}_{\mathcal{E}}(\mathcal{F})+\widetilde{m} \deg(\mathcal{O}_X(1))$ (see \S 3.5).

\begin{cor}\label{312}
There is an integer $B$, which depends on the following data
\begin{itemize}
\item $r$: rank of $F_{\mathcal{E}}(\mathcal{F})$;
\item $\widetilde{m}$: the integer such that $\pi_* \mathcal{E}nd_{\mathcal{O}_{\mathcal{X}}}(\mathcal{E})(\widetilde{m})$ is generated by global sections;
\item $\deg(\mathcal{O}_X(1))$: the degree of $\mathcal{O}_X(1)$;
\item $d$: the dimension of the pure sheaf $\mathcal{F}$,
\end{itemize}
such that
\begin{align*}
h^0(\mathcal{X},\mathcal{F} \otimes \mathcal{E}^{\vee} \otimes \pi^* \mathcal{O}_X(m))& \leq
\begin{cases}
r \frac{(\hat{\mu}_{\mathcal{E}}(\mathcal{F}) + B+m)^d}{d!} , & \text{ if } \hat{\mu}_{\rm max}(F_{\mathcal{E}}(\mathcal{F})) \geq \frac{d+1}{2} - r^2-m\\
0, & \text{ if } \hat{\mu}_{\rm max}(F_{\mathcal{E}}(\mathcal{F})) < \frac{d+1}{2} - r^2-m
\end{cases}
\end{align*}
\end{cor}

\begin{proof}
First, we consider the case $m=0$. Lemma \ref{311} tells us that
\begin{align*}
h^0(\mathcal{X},\mathcal{F} \otimes \mathcal{E}^{\vee})  \leq
\begin{cases}
r{\hat{\mu}_{\mathcal{E}}(\mathcal{F}) + \widetilde{m} \deg(\mathcal{O}_X(1)) + r^2+f(r)+\frac{d-1}{2}  \choose d}, & \text{ if } \hat{\mu}_{\rm max}(F_{\mathcal{E}}(\mathcal{F})) \geq \frac{d+1}{2} - r^2\\
0, & \text{ if } \hat{\mu}_{\rm max}(F_{\mathcal{E}}(\mathcal{F})) < \frac{d+1}{2} - r^2
\end{cases}
\end{align*}
Let $B$ be the integer $\widetilde{m} \deg(\mathcal{O}_X(1)) + r^2+f(r)+\frac{d-1}{2}$. We have
\begin{align*}
r{\hat{\mu}_{\mathcal{E}}(\mathcal{F}) + \widetilde{m} \deg(\mathcal{O}_X(1)) + r^2+f(r)+\frac{d-1}{2}  \choose d}&=r\frac{(\hat{\mu}_{\mathcal{E}}(\mathcal{F})+B)!}{d!((\hat{\mu}_{\mathcal{E}}(\mathcal{F})+B)-d)!}\\
\leq r\frac{ (\hat{\mu}_{\mathcal{E}}(\mathcal{F})+B)^d }{ d! }.
\end{align*}
Thus the inequality holds when $m=0$.

Now we will prove this corollary by induction on $d$. It is easy to check that the statement holds when $d=0$. Suppose that the inequality holds for $d \leq k-1$, and we will prove that the inequality holds for $d=k$. By the exactness of the functor $F_{\mathcal{E}}$, we know that
\begin{align*}
h^0(\mathcal{X},\mathcal{F} \otimes \mathcal{E}^{\vee} \otimes \pi^* \mathcal{O}_X(m))& =h^0(X,F_{\mathcal{E}}(\mathcal{F})(m)).
\end{align*}
Let $Y$ be a generic hyperplane on $X$. We have the exact sequence
\begin{align*}
0 \rightarrow F_{\mathcal{E}}(\mathcal{F})(m-1) \rightarrow F_{\mathcal{E}}(\mathcal{F})(m) \rightarrow F_{\mathcal{E}}(\mathcal{F})|_Y (m) \rightarrow 0.
\end{align*}

If $m< \frac{d+1}{2} - r^2 - \hat{\mu}_{\rm max}(F_{\mathcal{E}}(\mathcal{F}))$, we get
\begin{align*}
h^0(X,F_{\mathcal{E}}(\mathcal{F})|_Y (m))=0.
\end{align*}
Thus,
\begin{align*}
h^0(X, F_{\mathcal{E}}(\mathcal{F})(m))- h^0(X, F_{\mathcal{E}}(\mathcal{F})(m-1))=0.
\end{align*}
By induction, we have
\begin{align*}
h^0(X, F_{\mathcal{E}}(\mathcal{F})(m))=h^0(X, F_{\mathcal{E}}(\mathcal{F})(m-1))=0.
\end{align*}

If $m \geq \frac{d+1}{2} - r^2 - \hat{\mu}_{\rm max}(F_{\mathcal{E}}(\mathcal{F}))$, we have
\begin{align*}
h^0(X, F_{\mathcal{E}}(\mathcal{F})(m))-h^0(X, F_{\mathcal{E}}(\mathcal{F})(m-1)) \leq \frac{(\hat{\mu}_{\mathcal{E}}(\mathcal{F})+B'+m)^{d-1}}{(d-1)!}.
\end{align*}
Consider $h^0(X, F_{\mathcal{E}}(\mathcal{F})(m))$ as a function $f(m)$. We have
\begin{align*}
f(m)-f(m-1) \leq \frac{(\hat{\mu}_{\mathcal{E}}(\mathcal{F})+B'+m)^{d-1}}{(d-1)!}.
\end{align*}
Therefore, there exists a constant $C$ depending on $d$ such that
\begin{align*}
f(m) \leq  \frac{(\hat{\mu}_{\mathcal{E}}(\mathcal{F})+B'+C+m)^d}{d!}.
\end{align*}
Take $B=B'+C$. This finishes the proof of this corollary.
\end{proof}

\begin{rem}\label{313}
Let $\widetilde{\mathfrak{F}}^{ss}(P)$ be the set (or family) of $p$-semistable sheaves of pure dimension $d$ on $\mathcal{X}$ with the modified Hilbert polynomial $P$. Note that the slope $\hat{\mu}_{\mathcal{E}}(\mathcal{F})$ is uniquely determined by the modified Hilbert polynomial $P$, where $\mathcal{F} \in \widetilde{\mathfrak{F}}^{ss}(P)$. Thus, there is an upper bound for the set
\begin{align*}
\{h^0(\mathcal{X},\mathcal{F}\otimes \mathcal{E}^{\vee} \otimes \pi^* \mathcal{O}_{X}(m)) \text{ } | \text{ }\mathcal{F} \in \widetilde{\mathfrak{F}}^{ss}(P)\}.
\end{align*}
\end{rem}

\subsection{Geometric Invariant Theory}
In this subsection, we make a brief review about the geometric invariant theory (GIT), which will be used to construct the moduli space of $p$-semistable sheaves and the moduli space of $p$-semistable $\Lambda$-modules over a projective Deligne-Mumford stack $\mathcal{X}$. There are many very good references about this topic \cite{HuLe,MumFogKir}.

Let $G$ be a reductive algebraic group over an algebraically closed field $k$ acting on a projective $k$-scheme $X$. Given an action $\sigma: G \times X \rightarrow X$, a pair $(Y,\phi:X \rightarrow Y)$ is called a \emph{geometric quotient} of $X$ with respect to the $G$-action $\sigma$, if it satisfies the following conditions:
\begin{enumerate}
\item $\phi \circ \sigma=\phi \circ p_2$, where $p_2: G \times X \rightarrow X$ is the natural projection;
\item $\phi$ is surjective and submersive;
\item the image of $\Psi=(\sigma,p_2): G \times X \rightarrow X \times X$ is $X \times_Y X$;
\item $\mathcal{O}_Y \cong (\phi_*(\mathcal{O}_X))^G$.
\end{enumerate}
We say that a geometric quotient $(Y,\phi)$ is \emph{universal} if for any morphism $Y' \rightarrow Y$, the pair $(Y',\phi')$ is a geometric quotient of $X \times_Y Y'$ with respect to $G$, where $\phi'$ is induced by $\phi$. A geometric quotient $(Y,\phi)$ is \emph{good} if $\sigma$ is closed and $\Psi$ is separated.

Let $\mathscr{L}$ be a $G$-linearized ample line bundle on $X$. A point $x \in X$ is \emph{semistable} with respect to a given $G$-linearized ample line bundle $\mathscr{L}$ if there is an integer $n$ and an $G$-invariant global section $s \in H^0(X,\mathscr{L}^{\otimes n})^G$ such that $s(x) \neq 0$. A point $x$ is \emph{stable} (with respect to $\mathscr{L}$) if it is semistable, the stabilizer $G_x$ is finite and the $G$-orbit of $x$ is closed in the open set of all semistable points in $X$. Denote by $X^{ss}$ (resp. $X^s$) the set of all semistable points (resp. stable points).

\begin{thm}[Theorem 1.10 in \cite{MumFogKir}]\label{314}
Let $X$ be a projective scheme and let $G$ be a reductive group. Let $\mathscr{L}$ be a $G$-linearized ample line bundle on $X$. Then there is a projective scheme $Y$ and a morphism $\pi: X^{ss} \rightarrow Y$ such that $\pi$ is a universal good geometric quotient for the $G$-action. There is an open subset $Y^s \subseteq Y$ such that $X^s = \pi^{-1}(Y^s)$ and $\pi: X^s \rightarrow Y^s$ is a universal geometric quotient. Finally, there is a positive integer $m$ and a very ample line bundle $\mathscr{M}$ on $Y$ such that $\mathscr{L}^{\otimes m}|_{X^{ss}} \cong \pi^{-1}(\mathscr{M})$.
\end{thm}

At the end of this section, we want to review \emph{Luna's \'Etale Slicing Theorem}. We refer the reader to \cite{HuLe,MumFogKir} for more details. Recently, people proved a stacky version of \emph{Luna's \'etale slicing theorem}, which generalizes the theorem to algebraic stacks (see \cite{AlHRy}). In this paper, we only need the classical version for schemes.

\begin{thm}[Luna's \'Etale Slicing Theorem]\label{315}
Let $G$ be a reductive group acting on a finite type scheme $X$. Let $X \rightarrow X/G$ be the universal good geometric quotient. Let $x \in X$ be a point with a closed $G$-orbit. Then there exists a $G_x$-invariant locally closed subscheme $C \subseteq X$ passing through $x$, where $G_x$ is the stabilizer of $x$, such that the multiplication $C \times G \rightarrow X$ induces a $G$-equivariant \'etale morphism $C \times_{G_x} G \rightarrow X$.
\end{thm}

\subsection{Moduli Space of Coherent Sheaves}
In this subsection, we study the moduli problem of $p$-semistable sheaves on a projective Deligne-Mumford stack $\mathcal{X}$, and construct the moduli space $\mathcal{M}^{ss}(\mathcal{E},\mathcal{O}_X(1),P)$ of this moduli problem. The existence of the moduli space has been proved by Nironi (see \cite[\S 5 and \S 6]{Nir}), but we construct the moduli space in a slightly different way and we also explore some properties of smooth points in this moduli space $\mathcal{M}^{ss}(\mathcal{E},\mathcal{O}_X(1),P)$.

Let $S$ be a noetherian scheme of finite type, or an affine scheme. Let $\mathcal{X}$ be a projective (or quasi-projective) Deligne-Mumford stack with coarse moduli space $\pi: \mathcal{X} \rightarrow X$ over $S$. We choose a polarization $\mathcal{O}_X(1)$ on $X$ and a generating sheaf $\mathcal{E}$ on $\mathcal{X}$. Let $P$ be an integer polynomial (as modified Hilbert polynomial), and $d$ is the degree of $P$, which is a positive integer (as pure dimension).

We consider the moduli problem
\begin{align*}
\widetilde{\mathcal{M}}^{ss}(\mathcal{E},\mathcal{O}_X(1),P): (\text{Sch}/S)^{{\rm op}} \rightarrow \text{Set}.
\end{align*}
Given an $S$-scheme $T$, we define $\widetilde{\mathcal{M}}^{ss}(\mathcal{E},\mathcal{O}_X(1),P)(T)$ to be the set of equivalent classes of sheaves $\mathcal{F}_T$ on $\mathcal{X} \otimes_S T$ such that
\begin{itemize}
\item $\mathcal{F}_T$ is a $T$-flat family of $p$-semistable sheaf;
\item the modified Hilbert polynomial of each fiber of $\mathcal{F}_T$ is $P$;
\item $\mathcal{F}_T$ is equivalent to another family of sheaves $\mathcal{F}'_T$, if $\mathcal{F}_T \cong \mathcal{F}'_T \otimes p^* L$ for some $L \in {\rm Pic}(T)$.
\end{itemize}
We use the notation $``\sim"$ for the equivalence relation, i.e. $\mathcal{F}_T \sim \mathcal{F}'_T$.

The moduli problem $\widetilde{\mathcal{M}}^{ss}(\mathcal{E},\mathcal{O}_X(1),P)$ is defined for the $p$-semistable sheaves. Similarly, we can define a moduli problem $\widetilde{\mathcal{M}}^{s}(\mathcal{E},\mathcal{O}_X(1),P)$ for the $p$-stable sheaves. In this section, we will show that these two moduli problems are co-represented by  projective (or quasi-projective) $S$-schemes.

We first consider the quot-functor $\widetilde{{\rm Quot}}(\mathcal{G},P)$. The quot-functor $\widetilde{{\rm Quot}}(\mathcal{G},P)$ is represented by a projective $S$-scheme ${\rm Quot}(\mathcal{G},P)$ (see Theorem \ref{3021} and \cite[Theorem 2.17]{Nir}), which parametrizes quotients $[\mathcal{G} \rightarrow \mathcal{F}]$ with modified Hilbert polynomial $P$ (see \S 3.3). Moreover, by Theorem \ref{310} there is a positive integer $m$ such that for any element $[\mathcal{G} \rightarrow \mathcal{F}] \in {\rm Quot}(\mathcal{G},P)$, $\mathcal{F}$ is $m$-regular. Therefore there is a natural embedding
\begin{align*}
    \psi_m: {\rm Quot}(\mathcal{G},P)\hookrightarrow  {\rm Quot}(F_{\mathcal{E}}(\mathcal{G}),P) \hookrightarrow {\rm Grass}(H^0(X/S, F_{\mathcal{E}}(\mathcal{G})(m)),P(m) ),
\end{align*}
which is a closed embedding as $m$ increasing. Let $\mathscr{L}$ be the canonical invertible sheaf on the Grassmannian ${\rm Grass}(H^0(X/S,F_{\mathcal{E}}(\mathcal{G})(m)),P(m))$. Denote by $\mathscr{L}_m$, which is a very ample invertible sheaf on ${\rm Quot}(\mathcal{G},P)$, the pullback of $\mathscr{L}$ by the embedding $\psi_m$. Over a point $[\mathcal{G} \rightarrow \mathcal{F}]$, the fiber of the line bundle $\mathscr{L}_m$ is $\wedge^{P(m)} {\rm H}^0(X/S, F_{\mathcal{E}}(\mathcal{F})(m))$.

Now we go back to the family of $p$-semistable sheaves. By Theorem \ref{310}, we know that the family $\widetilde{\mathfrak{F}}^{ss}(P)$ of purely $d$-dimensional $p$-semistable coherent sheaves with modified Hilbert polynomial $P$ is bounded. Thus we can find an integer $m$ such that for any $\mathcal{F} \in \widetilde{\mathfrak{F}}^{ss}(P)$, $\mathcal{F}$ is $m$-regular. Moreover, by Remark \ref{313}, we can choose a positive integer $N$ large enough such that for any $\mathcal{F} \in \widetilde{\mathfrak{F}}^{ss}(P)$, we have
\begin{align*}
P(N) \geq P_{\mathcal{E}}(\mathcal{F},m)=h^0(X/S,F_{\mathcal{E}}(\mathcal{F})(m)).
\end{align*}
Let $V$ be the linear space $S^{\oplus P(N)}$, and note that
\begin{align*}
V  \cong H^0(X/S,F_{\mathcal{E}}(\mathcal{F})(N)).
\end{align*}

Denote by $\mathcal{G}$ be the coherent sheaf $\mathcal{E} \otimes \pi^* \mathcal{O}_X(-N)$. The above discussion tells us that any coherent sheaf $\mathcal{F} \in \widetilde{\mathfrak{F}}^{ss}(P)$ corresponds to a surjection $[V \otimes \mathcal{G} \rightarrow \mathcal{F}]$ together with an isomorphism
\begin{align*}
V \cong H^0(X/S,F_{\mathcal{E}}(\mathcal{F})(N)).
\end{align*}
With respect to the above discussion, we consider the quot-scheme ${\rm Quot}(V \otimes \mathcal{G},P)$. By definition, this quot-scheme parametrizes quotients $[\rho: V \otimes \mathcal{G} \rightarrow \mathcal{F}]$. The quotient induces a morphism
\begin{align*}
\alpha: V \rightarrow H^0(X/S, F_{\mathcal{E}}(\mathcal{F})(N)).
\end{align*}

Now we define a subset $Q^{ss} \subseteq {\rm Quot}(V \otimes \mathcal{G},P)$. The set $Q^{ss}$ includes all quotients $[\rho: V \otimes \mathcal{G} \rightarrow \mathcal{F}]$ such that
\begin{itemize}
\item $\mathcal{F}$ has pure dimension $d$ and is a $p$-semistable sheaf with modified Hilbert polynomial $P$;
\item the morphism $\alpha: V  \xrightarrow{\cong} H^0(X/S, F_{\mathcal{E}}(\mathcal{F})(N))$ induced by $\rho$ is an isomorphism.
\end{itemize}
Both the above two conditions are open condition. Therefore, $Q^{ss}$ is an open subset of ${\rm Quot}(V \otimes \mathcal{G},P)$. Also, $Q^{ss}$ parameterizes the family $\widetilde{\mathfrak{F}}^{ss}(P)$ of sheaves. Similarly, we can construct $Q^{s} \subseteq {\rm Quot}(V \otimes \mathcal{G},P)$ including all $p$-stable sheaves.

Now we come to the part of GIT. There is a natural ${\rm GL}(V)$-action on ${\rm Quot}(V\otimes \mathcal{G},P)$. For each quotient $[\rho: V  \otimes \mathcal{G} \rightarrow \mathcal{F}] \in {\rm Quot}(V\otimes \mathcal{G},P)$, there is a natural injective homomorphism
\begin{align*}
{\rm Aut}(\mathcal{F}) \rightarrow {\rm GL}(V)
\end{align*}
such that the image is the stabilizer subgroup ${\rm GL}(V)_{[\rho]}$ of ${\rm GL}(V)$ at the quotient $[\rho]$. It is easy to check that the center $Z({\rm GL}(V))$ of ${\rm GL}(V)$ is contained in ${\rm GL}(V)_{[\rho]}$ of any point $[\rho] \in {\rm Quot}(V \otimes \mathcal{G},P)$. Therefore, we consider the action of ${\rm SL}(V)$ and ${\rm PGL}(V)$ instead of ${\rm GL}(V)$ (see \cite[Lemma 4.3.2]{HuLe} and \cite[Remark 5.5]{Nir}).

As we discussed above for the quot-scheme, there is natural embedding
\begin{align*}
    \psi_N: {\rm Quot}(V\otimes \mathcal{G},P)\hookrightarrow {\rm Grass}(H^0(X, F_{\mathcal{E}}\left(V \otimes \mathcal{G}\right)(N) ),P(N) ),
\end{align*}
where $N$ is a large enough positive integer. We use the same notation $\mathscr{L}_N$ for the line bundle over the scheme ${\rm Quot}(V\otimes \mathcal{G},P)$. The group action ${\rm SL}(V)$ on ${\rm Quot}(V\otimes \mathcal{G},P)$ induces an action on the line bundle $\mathscr{L}_N$, and it is easy to check that $\mathscr{L}_N$ is ${\rm SL}(V)$-linearized. Therefore, we can discuss the semistable and stable points in ${\rm Quot}(V\otimes \mathcal{G},P)$ in the sense of GIT.

We will prove that a point $[\rho : V\otimes \mathcal{G} \rightarrow \mathcal{F}] \in {\rm Quot}(V\otimes \mathcal{G},P)$ is semistable (resp. stable) if and only if $\mathcal{F}$ is $p$-semistable (resp. $p$-stable) (see Theorem \ref{316}). We will use Lemma \ref{3141}, \ref{3142} and \ref{3143} to prove the statement.

\begin{lem}\label{3141}
Let $V$ and $W$ be vector spaces, and let ${\rm Grass}(V \otimes W, d)$ be the Grassmannian of $d$-dimension subspace in $V \otimes W$. A point $[V \otimes W \rightarrow U] \in {\rm Grass}(V \otimes W, d)$ is semistable with respect to the action of ${\rm SL}(V)$ and invertible sheaf $\mathscr{L}$ if and only if, for all nonzero proper subspaces $H \subseteq V$, we have ${\rm Im}(H \otimes W) \neq 0$ and
\begin{align*}
\frac{\dim(H)}{\dim({\rm Im}(H \otimes W)) } \leq \frac{\dim(V)}{\dim(U)}.
\end{align*}
\end{lem}

\begin{proof}
This is a special case of \cite[Proposition 4.3]{MumFogKir}.
\end{proof}

\begin{lem}\label{3142}
Let $\mathcal{F}$ be a $p$-semistable sheaf of dimension $d$ with modified Hilbert polynomial $P$. For all sufficiently large integers $N$, if $\mathcal{F}' \subseteq \mathcal{F}$ is a subsheaf, then we have
\begin{align*}
\frac{h^0(F_{\mathcal{E}}(\mathcal{F}'))(N)}{r(\mathcal{F}')} \leq \frac{P(N)}{r(\mathcal{F})}.
\end{align*}
\end{lem}

\begin{proof}
Let $\mathcal{F}' \subseteq \mathcal{F}$ be a subsheaf. We consider the Harder-Narasimhan filtration of ${\rm HN}_{\bullet}(\mathcal{F}')$, and denote by ${\rm HN}_{i}(\mathcal{F}')$ the terms in the filtration. Let $\mathcal{H}_i=gr_i^{HN}(\mathcal{F}')={\rm HN}_{i}(\mathcal{F}') / {\rm HN}_{i-1}(\mathcal{F}')$. Since $F_{\mathcal{E}}$ is an exact functor, we have
\begin{align*}
h^0(F_{\mathcal{E}}(\mathcal{F})(N)  ) \leq \sum\limits_i h^0( F_{\mathcal{E}}(\mathcal{H}_i)(N) ).
\end{align*}
Also, we have $\hat{\mu}_{\mathcal{E}}(\mathcal{H}_i) \leq \hat{\mu}_{\mathcal{E}}(\mathcal{F})$ and $\sum\limits_i r(\mathcal{H}_i)= r(\mathcal{F}')$.

By Corollary \ref{312}, we have
\begin{align*}
h^0( F_{\mathcal{E}}(\mathcal{H}_i )(N) ) \leq r(\mathcal{H}_i) \frac{( \hat{\mu}_{\mathcal{E}}(\mathcal{H}_i)+B+N )^d  }{d!}.
\end{align*}
For simplicity, we do not consider the other case with dimension zero. Thus,
\begin{align*}
h^0(F_{\mathcal{E}}(\mathcal{F})(N)) & \leq \sum\limits_i  r(\mathcal{H}_i) \frac{( \hat{\mu}_{\mathcal{E}}(\mathcal{H}_i)+B+N )^d  }{d!} \\
& \leq (r(\mathcal{F}' )-1 ) \frac{(\hat{\mu}_{\mathcal{E}}(\mathcal{F})  N+B)}{d!} + \frac{(\nu +N +B)^d}{d!},
\end{align*}
where $\nu={\rm min}(\hat{\mu}_{\mathcal{E}}(\mathcal{H}_i))$. For any $A$, we can always find a $C \geq A$ such that if $\nu \leq \mu_{\mathcal{E}}(\mathcal{F})-C$, we have
\begin{align*}
(r(\mathcal{F}' )-1 ) \frac{(\hat{\mu}_{\mathcal{E}}(\mathcal{F})  N+B)}{d!} + \frac{(\nu +N +B)^d}{d!} \leq r(\mathcal{F}' )\frac{(N-A )^d}{d!}, \quad N \geq C.
\end{align*}
On the other hand, we can choose $A$ such that
\begin{align*}
\frac{(N-A)^d}{d!} \leq \frac{p_{\mathcal{E}}(\mathcal{F},N) }{r(\mathcal{F})}, \quad N \geq A.
\end{align*}
Therefore, if $\nu \leq \mu_{\mathcal{E}}(\mathcal{F})-C$, we have
\begin{align*}
\frac{h^0(F_{\mathcal{E}}(\mathcal{F}')   ) }{r(\mathcal{F}') } <\frac{h^0(F_{\mathcal{E}}(\mathcal{F})   ) }{r(\mathcal{F}) }.
\end{align*}

If $\nu \geq \mu_{\mathcal{E}}(\mathcal{F})-C$, it is equivalent to consider the saturation $\mathcal{F}'_{\rm sat}$ of such a subsheaf $\mathcal{F}' \subseteq \mathcal{F}$. The set of saturations $\mathcal{F}'_{\rm sat}$ of all such subsheaves $\mathcal{F}'$ is bounded. Therefore, there are only finitely many modified Hilbert polynomials. By the finiteness of the modified Hilbert polynomials and the $p$-semistability of $\mathcal{F}$, we can take a large enough $N$ such that
\begin{align*}
p_{\mathcal{E}}(\mathcal{F}'_{\rm sat}(N)) \leq p_{\mathcal{E}}(\mathcal{F}(N)).
\end{align*}
Thus, we have the desired inequality
\begin{align*}
\frac{h^0(F_{\mathcal{E}}(\mathcal{F}'_{\rm sat}))(N)}{r(\mathcal{F}')} \leq \frac{P(N)}{r(\mathcal{F})}
\end{align*}
in this case.
\end{proof}

In the classical case, the number $C$ is determined by the multiplicity $r$ and the degree $d$ of the Hilbert polynomial (see Theorem \cite[Theorem 4.4.1]{HuLe}), which is an application of the Le Potier-Simpson Estimate.

\begin{lem}[Lemma 6.10 in \cite{Nir}]\label{3143}
If $\mathcal{F}$ is a coherent sheaf on $\mathcal{X}$ that can be deformed to a pure sheaf of the same dimension $d$, then there is a pure sheaf $\mathcal{H}$ of dimension $d$ on $\mathcal{X}$ and a map $\mathcal{F} \rightarrow \mathcal{H}$ such that the kernel is $T_{d-1}(\mathcal{F})$ and $P_{\mathcal{E}}(\mathcal{F})=P_{\mathcal{E}}(\mathcal{H})$.
\end{lem}

The classical version of this lemma can be found in \cite[Proposition 4.4.2]{HuLe} and \cite[Lemma 1.17]{Simp2}. Nironi used a similar approach to prove this lemma.

\begin{thm}\label{316}
Let $P$ be an integral polynomial with degree $d$. There exists a large integer $N$ such that a point $[V \otimes \mathcal{G} \rightarrow \mathcal{F}] \in {\rm Quot}(V \otimes \mathcal{G},P)$ is semistable (resp. stable) with respect to the action of ${\rm SL(V)}$ and the line bundle $\mathscr{L}_N$, if and only if $\mathcal{F}$ is a $p$-semistable (resp. $p$-stable) sheaf of pure dimension $d$ and the map $V \rightarrow {\rm H}^0(X,F_{\mathcal{E}}(\mathcal{F})(N))$ is an isomorphism.
\end{thm}

\begin{proof}
We only give the proof of the semistable case, and the stable case can be proved similarly. The proof includes three parts
\begin{enumerate}
\item we first prove that $p$-semistability implies semistability;
\item next, we prove the other direction under the assumption that $\mathcal{F}$ is pure;
\item finally, we prove the general case based on Lemma \ref{3143}.
\end{enumerate}

Let $[\rho: V \otimes \mathcal{G} \rightarrow \mathcal{F}]$ be a point in ${\rm Quot}(V \otimes \mathcal{G},P)$. Suppose that $\mathcal{F}$ is a $p$-semistable sheaf such that $V \cong {\rm H}^0(X,F_{\mathcal{E}}(\mathcal{F}))$ is an isomorphism. We will prove that $[\rho]$ is semistable.

We first take a positive integer $M$ large enough such that ${\rm Quot}(V \otimes \mathcal{G},P)$ can be embedded into ${\rm Grass}(V \otimes W ,P(M))$, where $W=H^0(X/S,  F_{\mathcal{E}}( \mathcal{G})(M) )$. Let $H$ be a non-trivial proper subspace of $V$ such that the image of $H \otimes \mathcal{G}$ is non-empty. Let $\mathcal{F}'$ be the image of $H \otimes \mathcal{G}$. We have
\begin{align*}
0 \rightarrow \mathcal{F}'' \rightarrow H \otimes \mathcal{G} \rightarrow \mathcal{F}' \rightarrow 0,
\end{align*}
where $\mathcal{F}''$ is the kernel of the quotient $H \otimes \mathcal{G} \rightarrow \mathcal{F}'$. Since $M$ is a large enough integer, we can assume that
\begin{align*}
h^0(X/S, F_{\mathcal{E}}(\mathcal{F}')(M) )=p_{\mathcal{E}}(\mathcal{F}',M), \quad h^1(X/S,F_{\mathcal{E}}(\mathcal{F}'')(M))=0.
\end{align*}
Based on the above property, we have a surjective morphism
\begin{align*}
H \otimes W \rightarrow H^0(X/S, F_{\mathcal{E}}(\mathcal{F}'(M)) ) \rightarrow 0.
\end{align*}
Since $\mathcal{F}$ is $p$-semistable, by Lemma \ref{3142}, we have
\begin{align*}
\frac{h^0(F_{\mathcal{E}}(\mathcal{F}')(N))}{r(\mathcal{F}')} \leq \frac{h^0(F_{\mathcal{E}}(\mathcal{F})(N))}{r(\mathcal{F})}.
\end{align*}
Note that the polynomials $P_{\mathcal{E}}(\mathcal{F},m)$ all have first term $r(\mathcal{F}) m ^d$. Thus, if the integer $M$ is large enough, we have
\begin{align*}
\frac{h^0(\mathcal{F}'(N))}{ P_{\mathcal{E}}(\mathcal{F}',M)  } \leq \frac{h^0(\mathcal{F}(N))}{ P_{\mathcal{E}}(\mathcal{F},M)   },
\end{align*}
and then,
\begin{align*}
\frac{\dim(H)}{\dim({\rm Im}(H \otimes W))}=\frac{h^0(F_{\mathcal{E}}(\mathcal{F}')(N))}{ P_{\mathcal{E}}(\mathcal{F}',M)  } \leq \frac{P(N)}{ P(M) }=\frac{\dim(V)}{\dim(U)}.
\end{align*}
This inequality holds for any non-trivial proper subspace $H$ of $V$. By Lemma \ref{3141}, the point $[V \otimes \mathcal{G} \rightarrow \mathcal{F}]$ is semistable. This finishes the proof for one direction.

Now let $[\rho: V \otimes \mathcal{G} \rightarrow \mathcal{F}]$ be semistable in the sense of GIT. We will prove that $\mathcal{F}$ is $p$-semistable and the map $V \rightarrow {\rm H}^0(X,F_{\mathcal{E}}(\mathcal{F})(N))$ is an isomorphism.

We first suppose that $\mathcal{F}$ is pure. Let $\mathcal{F}'$ be a subsheaf of $\mathcal{F}$. By taking the pullback of the following diagram
\begin{center}
\begin{tikzcd}
    V' \arrow[d, dashed, hook] \arrow[r, dashed] & \mathcal{F}' \otimes \mathcal{G}^{\vee} \arrow[d, hook]\\
    V \arrow[r,"\rho"] & \mathcal{F} \otimes \mathcal{G}^{\vee}
\end{tikzcd}
\end{center}
we find a subspace $V' \subseteq V$ such that the quotient $[V' \otimes \mathcal{G} \rightarrow \mathcal{F}']$ is induced by $[\rho]$. Furthermore, $\mathcal{F}$ and $\mathcal{F}'$ have the same regularity and $V' \cong H^0(F_{\mathcal{E}}(\mathcal{F}')(N))$. With the same notation as in the first part of the proof, by taking $N$ and $M$ large enough, we have
\begin{align*}
\frac{h^0(F_{\mathcal{E}}(\mathcal{F}')(N))}{ P_{\mathcal{E}}(\mathcal{F}',M)  }=\frac{\dim(H)}{\dim({\rm Im}(H \otimes W))} \leq \frac{\dim(V)}{\dim(U)}= \frac{P(N)}{ P(M) }.
\end{align*}
This inequality gives us the following
\begin{align*}
\frac{ P_{\mathcal{E}}(\mathcal{F}',N) }{r(\mathcal{F}')} \leq \frac{ P_{\mathcal{E}}(\mathcal{F},N) }{r(\mathcal{F})}, \quad N \gg 0.
\end{align*}
Therefore, $\mathcal{F}$ is $p$-semistable. Taking $N$ large enough, the induced map
\begin{align*}
V \rightarrow {\rm H}^0(X,F_{\mathcal{E}}(\mathcal{F})(N))
\end{align*}
is surjective. By counting the dimension, this map is an isomorphism. This finishes the proof when $\mathcal{F}$ is pure.

To complete the proof of this theorem, we will show that any semistable (GIT) point $[\rho]$, the sheaf $\mathcal{F}$ is pure. By Lemma \ref{3143}, there exists a pure sheaf $\mathcal{H}$ and a morphism $\varrho: \mathcal{F} \rightarrow \mathcal{H}$ such that
\begin{itemize}
\item the kernel of $\varrho$ is $T_{d-1}(\mathcal{F})$, i.e. the map $\varrho$ is generically injective;
\item $P_{\mathcal{E}}(\mathcal{F})=P_{\mathcal{E}}(\mathcal{H})$.
\end{itemize}
The map $\varrho$ induces an injective map
\begin{align*}
V \xrightarrow{\cong} H^0(X/S, F_{\mathcal{E}}(\mathcal{F})(N)  )\rightarrow H^0(X/S, F_{\mathcal{E}}(\mathcal{H})(N)  ).
\end{align*}
Let $\mathcal{H}''$ be any quotient of $\mathcal{H}$, and denote by $\mathcal{F}'$ the kernel of the composition $\mathcal{F} \rightarrow \mathcal{H} \rightarrow \mathcal{H}''$. We have the following exact sequence
\begin{align*}
0 \rightarrow \mathcal{F}' \rightarrow \mathcal{F} \rightarrow \mathcal{H} \rightarrow \mathcal{H}'' \rightarrow 0.
\end{align*}
This implies
\begin{align*}
h^0(F_{\mathcal{E}}(\mathcal{H}'')(N)) & \geq h^0(F_{\mathcal{E}}(\mathcal{F})(N)) - h^0(F_{\mathcal{E}}(\mathcal{F}')(N))\\
& \geq (r(\mathcal{F})-r(\mathcal{F}')) p_{\mathcal{E}}(\mathcal{F},N) = r(\mathcal{H}'')  p_{\mathcal{E}}(\mathcal{F},N).
\end{align*}
Therefore, $\mathcal{H}$ is $p$-semistable. Furthermore, $V \cong h^0(F_{\mathcal{E}}(\mathcal{H})(N))$. Note that $\varrho$ induces an injection $V \rightarrow H^0(X/S, F_{\mathcal{E}}(\mathcal{H})(N)  )$. By counting the dimension, it is an isomorphism. This isomorphism means that the map $V \otimes \mathcal{G} \rightarrow \mathcal{H}$ factors through $\mathcal{F}$, i.e. the morphism $\varrho: \mathcal{F} \rightarrow \mathcal{H}$ is surjective. Since $P_{\mathcal{E}}(\mathcal{F})=P_{\mathcal{E}}(\mathcal{H})$, we have $\mathcal{F} \cong \mathcal{H}$. This means that $\mathcal{F}$ is pure.

\end{proof}

This theorem tells us that although $Q^{ss}$ is defined as the set of $p$-semistable coherent sheaves, it is exactly the set of semistable points in ${\rm Quot}(V \otimes \mathcal{G},P)$ under the action of ${\rm SL}(V)$. Similarly, $Q^{s}$ is the set of stable points. Before we move on to the GIT quotient of $Q^{ss}$ (and $Q^{s}$) and construct the moduli space, we first discuss the orbit of a semistable point $[\rho] \in Q^{ss}$ under the action of ${\rm SL}(V)$. The following two lemmas \ref{3181} and \ref{3182} are stated in \cite[Theorem 6.20]{Nir} without a proof, and we give the proofs here.
\begin{lem}\label{3181}
Let $[V \otimes \mathcal{G} \rightarrow \mathcal{F}_i]$, $i=1,2$ be two points in $Q^{ss}$. The closures of the corresponding orbits in $Q^{ss}$ intersect if and only if $gr^{\rm JH}(\mathcal{F}_1) \cong gr^{\rm JH}(\mathcal{F}_2)$.
\end{lem}

\begin{proof}
Let $[\rho: V \otimes \mathcal{G} \rightarrow \mathcal{F}] \in Q^{ss}$ be a point. Let
\begin{align*}
0 = {\rm JH}_0(\mathcal{F}) \subseteq {\rm JH}_1(\mathcal{F}) \subseteq \dots \subseteq {\rm JH}_l(\mathcal{F})= \mathcal{F}
\end{align*}
be the Jordan-H\"older filtration of $\mathcal{F}$ (the same notation as in \S 3.4). To prove the lemma, it is enough to show that we can construct a quotient $[\bar{\rho}: V \otimes \mathcal{G} \rightarrow gr^{\rm JH}(\mathcal{F})]$  such that $[\bar{\rho}]$ is included in the closure of the orbit of $[\rho]$.

Since $N$ is a large enough integer, we can assume that $F_{\mathcal{E}}({\rm JH}_i(\mathcal{F}))(N)$ is globally generated, and let $V_{\leq i}$ be the subspace of $V$ such that the quotient $[V_{\leq i} \otimes \mathcal{G} \rightarrow {\rm JH}_i(\mathcal{F})]$ is induced by $[\rho]$ and $V_{\leq i} \cong H^0(X/S, F_{\mathcal{E}}({\rm JH}_i(\mathcal{F}))(N))$. Let $V_i: = V_{\leq i}/ V_{\leq i-1}$. We have the induced surjections $V_i \otimes \mathcal{G} \rightarrow gr_i^{\rm JH}(\mathcal{F})$. Summing up these induced surjections, we get a point $[\bar{\rho}: V \otimes \mathcal{G} \rightarrow gr^{\rm JH}(\mathcal{F})]$.

To show that $[\bar{\rho}]$ is in the closure of the orbit of $[\rho]$, it suffices to find an one-parameter subgroup $\lambda$ such that $\lim\limits_{t \rightarrow 0} \lambda(t) \cdot [\rho]=[\bar{\rho}]$. The construction of such an one-parameter subgroup $\lambda$ is the same as \cite[Lemma 4.4.3]{HuLe}. Therefore, the point $[\bar{\rho}]$ is included in the closure of the orbit of $[\rho]$ in $Q^{ss}$.
\end{proof}

\begin{lem}\label{3182}
Let $[\rho: V \otimes \mathcal{G} \rightarrow \mathcal{F}]$ be a point in $Q^{ss}$. The orbit of $[\rho]$ is closed in $Q^{ss}$ if and only if $\mathcal{F}$ is $p$-polystable.
\end{lem}

\begin{proof}
Suppose that $\mathcal{F}=\bigoplus\limits_i \mathcal{F}_i^{n_i}$ is $p$-polystable. Let $[\rho': V \otimes \mathcal{G} \rightarrow \mathcal{F}']$ be a point in the closure of the orbit of $[\rho]$. To prove that the orbit of $[\rho]$ is closed, it suffices to show that $\mathcal{F}' \cong \mathcal{F}$.

Since $[\rho']$ is in the closure of the orbit of $[\rho]$, there is a smooth curve $T$ (over $S$) parametrizing a flat family $\mathcal{F}_T$ of sheaves on $\mathcal{X}$ such that
\begin{align*}
(\mathcal{F}_T)_0 \cong \mathcal{F}', \quad (\mathcal{F}_T)_{T\backslash \{0\}} \cong \mathcal{O}_{T \backslash \{0\}} \otimes \mathcal{F}.
\end{align*}
By the flatness of the family $\mathcal{F}_T$ and the stacky version of semicontinuity (see \cite[Theorem 1.8]{Nir}), the function
\begin{align*}
T \rightarrow \mathbb{N}, \quad t \rightarrow {\rm hom}(\mathcal{F}_i, (\mathcal{F}_T)_t)
\end{align*}
is semicontinous for each $i$, where ${\rm hom}(\mathcal{F}_i, (\mathcal{F}_T)_t)=\dim({\rm Hom}(\mathcal{F}_i, (\mathcal{F}_T)_t) )$. If $t \neq 0$, ${\rm hom}(\mathcal{F}_i, (\mathcal{F}_T)_t)=n_i$. The semicontinuity implies that $n'_i= {\rm hom}(\mathcal{F}_i, (\mathcal{F}_T)_0) \geq n_i$. Then, we have $\sum\limits_i \mathcal{F}_i^{n'_i} \subseteq \mathcal{F}'$. Since $\mathcal{F}'$ is $p$-semistable and $\mathcal{F}'$, $\mathcal{F}$ are in the same flat family $\mathcal{F}_T$, then the only possible case is $n'_i=n_i$ and $\mathcal{F}' \cong \mathcal{F}$.

\end{proof}

Let
\begin{align*}
\mathcal{M}^{ss}(\mathcal{E},\mathcal{O}_X(1),P):=Q^{ss}/{\rm SL}(V)
\end{align*}
be the good geometric quotient with respect to the group action ${\rm SL}(V)$ and line bundle $\mathscr{L}_N$ given by Theorem \ref{314}.

\begin{thm}\label{317}
With respect to the situation above, we have the following results.
\begin{enumerate}
\item The moduli space $\mathcal{M}^{ss}(\mathcal{E},\mathcal{O}_X(1),P)$ is a projective $S$-scheme.
\item There exists a natural morphism
\begin{align*}
\widetilde{\mathcal{M}}^{ss}(\mathcal{E},\mathcal{O}_X(1),P) \rightarrow \mathcal{M}^{ss}(\mathcal{E},\mathcal{O}_X(1),P)
\end{align*}
such that $\mathcal{M}^{ss}(\mathcal{E},\mathcal{O}_X(1),P)$ universally co-represents $\widetilde{\mathcal{M}}^{ss}(\mathcal{E},\mathcal{O}_X(1),P)$. The points in the moduli space $\mathcal{M}^{ss}(\mathcal{E},\mathcal{O}_X(1),P)$ represent the $S$-equivalent classes of $p$-semistable sheaves.
\item $\mathcal{M}^{s}(\mathcal{E},\mathcal{O}_X(1),P)$ is a coarse moduli space of $\widetilde{\mathcal{M}}^{s}(\mathcal{E},\mathcal{O}_X(1),P)$.
\item If $x \in \mathcal{M}^{s}(\mathcal{E},\mathcal{O}_X(1),P)$ is a point such that $Q^s$ is smooth at the inverse image of $x$, then $\mathcal{M}^{s}(\mathcal{E},\mathcal{O}_X(1),P)$ is smooth at $x$.
\end{enumerate}
\end{thm}

\begin{proof}
The first two statements follow from Theorem \ref{301}, Theorem \ref{314}, Theorem \ref{316} and Lemma \ref{3181}. Nironi also stated $(1)$ and $(2)$ in \cite[Theorem 6.22]{Nir} and pointed out that $\mathcal{M}^{ss}(\mathcal{E},\mathcal{O}_X(1),P)$ is not a coarse moduli space.

Now we will prove the other two statements. By Theorem \ref{314}, there is an open subset $\mathcal{M}^{s}$ of $\mathcal{M}^{ss}(\mathcal{E},\mathcal{O}_X(1),P)$ such that its preimage via the map $Q^{ss}/{\rm SL}(V)$ is the set of stable points. By Theorem \ref{316}, a point in $Q^{ss}$ is stable if and only if it is $p$-stable. Therefore the open set $\mathcal{M}^{s}$ is exactly $\mathcal{M}^{s}(\mathcal{E},\mathcal{O}_X(1),P)$. To prove that the set $\mathcal{M}^{s}(\mathcal{E},\mathcal{O}_X(1),P)$ is a coarse moduli space, we only have to check that there is a bijection
\begin{align*}
\widetilde{\mathcal{M}}^{s}(\mathcal{E},\mathcal{O}_X(1),P)(S) \rightarrow {\rm Hom}(S,\mathcal{M}^{s}(\mathcal{E},\mathcal{O}_X(1),P)).
\end{align*}
Clearly, two $p$-stable sheaves $\mathcal{F}_1$, $\mathcal{F}_2$ are $S$-equivalent if and only if $\mathcal{F}_1 \cong \mathcal{F}_2$. Therefore the bijection is directly induced from morphism $\widetilde{\mathcal{M}}^{ss}(\mathcal{E},\mathcal{O}_X(1),P)(S) \rightarrow {\rm Hom}(S,\mathcal{M}^{ss}(\mathcal{E},\mathcal{O}_X(1),P))$. This finishes the proof of $(3)$.

The proof of the last statement is similar to the classical case \cite[\S 4]{HuLe}, and we will use Luna's \'Etale Slicing Theorem (Theorem \ref{315}) to prove this statement. Let $[\rho]$ be a point in $\mathcal{M}^{s}(\mathcal{E},\mathcal{O}_X(1),P)$. It is easy to check the stabilizer of $[\rho]$ in the group ${\rm GL}(V)$ is exactly $Z({\rm GL}(V))$, the center of ${\rm GL}(V)$. By Luna's \'Etale Slicing Theorem, $Q^s$ is a principal ${\rm PGL}(V)$-bundle over $\mathcal{M}^{s}(\mathcal{E},\mathcal{O}_X(1),P)$. Moreover, there is a locally closed subset $C \subseteq Q^{ss}$ of $[\rho]$ such that the multiplication map $C \times {\rm PGL}(V) \rightarrow Q^s$ induces an \'etale morphism $C/ {\rm PGL}(V)_{[\rho]} \rightarrow \mathcal{M}^{s}(\mathcal{E},\mathcal{O}_X(1),P)$. Therefore, by the property of the \'etale morphism, if the inverse image of $[\rho]$ in $Q^s$ is smooth, then $x \in \mathcal{M}^{s}(\mathcal{E},\mathcal{O}_X(1),P)$ is also a smooth point.
\end{proof}

\section{$\Lambda$-Modules on Projective Deligne-Mumford Stacks}
In this section, we give the definition of $\Lambda$-modules on a Deligne-Mumford stack and prove some properties of them. The setup of \S 4.2 and \S 4.3 is the same as \S 3.2 - \S 3.6. We are working on a projective Deligne-Mumford stack $\mathcal{X}$ over an algebraically closed field $k$. In these two subsections, we construct the Harder-Narasimhan filtration and Jordan-H\"older filtration of $\Lambda$-modules (see \S 4.2), and prove the second boundedness (see Proposition \ref{407}).

\subsection{Graded Algebras and $\Lambda$-Modules over Projective Deligne-Mumford Stacks}
A \emph{graded ring} $R$ is a ring together with a direct sum decomposition $R = R_0 \oplus R_1 \oplus \dots$ such that $R_i R_j \subseteq R_{i+j}$ for $i,j \geq 0$. A \emph{graded $R$-module} $M$ is an $R$-module with a direct sum decomposition $M=\bigoplus\limits_{-\infty}^{\infty}M_i$ such that $R_i M_j \subseteq M_{i+j}$ for all $i,j$. A \emph{graded $R$-algebra} $M$ is a graded $R$-module such that $M_i M_j \subseteq M_{i+j}$ for all $i,j$. With respect to the above definitions of graded structures, a sheaf of graded algebras over a stack can be defined in a similar way as in \cite[\S 2]{Simp2}.

Let $S$ be an algebraic space, which is locally of finite type over an algebraically closed field $k$. Let $\mathcal{X}$ be a separated and locally finitely-presented Deligne-Mumford stack over $S$. A \emph{sheaf of graded algebras} over $\mathcal{X}$ is a sheaf of $\mathcal{O}_{\mathcal{X}}$-algebras $\Lambda$ with a filtration
\begin{align*}
\Lambda_0 \subseteq \Lambda_1 \subseteq \dots \Lambda_n \subseteq \dots
\end{align*}
satisfying the following conditions.
\begin{enumerate}
    \item $\Lambda$ has both left and right $\mathcal{O}_{\mathcal{X}}$-module structures.
    \item $\Lambda=\lim\limits_{\longleftarrow} \Lambda_i$ and $\Lambda_i \cdot \Lambda_j \subseteq \Lambda_{i+j}$.
    \item There is a natural morphism $\mathcal{O}_{\mathcal{X}} \rightarrow \Lambda$, of which the image is $\Lambda_0$.
    \item The graded sheaf ${\rm Gr}_i(\Lambda)=\Lambda_i / \Lambda_{i-1}$ is a $\mathcal{O}_{\mathcal{X}}$-module for $i \geq 1$.
    \item The left and right $\mathcal{O}_{\mathcal{X}}$-module structures on ${\rm Gr}_i(\Lambda)$ are equal. In other words, there is an isomorphism such that ${\rm Gr}_i(\Lambda)_l \cong {\rm Gr}_i(\Lambda)_r$.
    \item ${\rm Gr}(\Lambda):=\bigoplus\limits_{i=0}^{\infty}{\rm Gr}_i(\Lambda)$ is generated by ${\rm Gr}_1(\Lambda)$. More precisely, the morphism of sheaves
        \begin{align*}
        \overbrace{{\rm Gr}_1(\Lambda) \otimes_{\mathcal{O}_{\mathcal{X}}} \dots \otimes_{\mathcal{O}_{\mathcal{X}}} {\rm Gr}_1(\Lambda)}^{i} \rightarrow {\rm Gr}_i(\Lambda)
        \end{align*}
        is surjective.
\end{enumerate}

Let $U \rightarrow \mathcal{X}$ be a local chart of $\mathcal{X}$, and we assume that $U= \Spec(A)$ is an affine scheme and the morphism is \'etale. In this local chart, $\Lambda(U)$ is a graded algebra over $A$. 

A \emph{$\Lambda$-sheaf $\mathcal{F}$} is a sheaf on $\mathcal{X}$ together with a left $\Lambda$-action. A \emph{coherent $\Lambda$-sheaf} (resp. \emph{quasi-coherent $\Lambda$-sheaf}) $\mathcal{F}$ is a $\Lambda$-sheaf such that $\mathcal{F}$ is coherent (resp. quasi-coherent) with respect to the $\mathcal{O}_{\mathcal{X}}$-structure. In this paper, a $\Lambda$-sheaf is also called a \emph{$\Lambda$-module}. Similarly, a sheaf is called a \emph{$\mathcal{O}_{\mathcal{X}}$-module}.

There are several ways to understand ``an action of $\Lambda$". Usually an action of $\Lambda$ on $\mathcal{F}$ means that we have a morphism
\begin{align*}
\Lambda \rightarrow \mathcal{E}nd(\mathcal{F}).
\end{align*}
Equivalently, this morphism can be interpreted as
\begin{align*}
\Lambda \otimes \mathcal{F} \rightarrow \mathcal{F}.
\end{align*}
Sometimes we use the notation $(\mathcal{F},\Phi)$ for a $\Lambda$-module, where $\mathcal{F}$ is a coherent sheaf and $\Phi: \Lambda \rightarrow \mathcal{E}nd(\mathcal{F})$ is the action of $\Lambda$ on $\mathcal{F}$.

By condition $(6)$, the graded sheaf ${\rm Gr}(\Lambda)$ is generated by ${\rm Gr}_1(\Lambda)$, which is a coherent sheaf. Therefore the sheaf $\Lambda$ is also generated by $\Lambda_1$. Now given an action of $\Lambda$ on $\mathcal{F}$, it gives a unique action of ${\rm Gr}_1(\Lambda)$ on $\mathcal{F}$, and we have an injective map
\begin{align*}
{\rm Hom}(\Lambda,\mathcal{E}nd(\mathcal{F})) \rightarrow {\rm Hom}({\rm Gr}_1(\Lambda),\mathcal{E}nd(\mathcal{F})).
\end{align*}
If ${\rm Gr}_1(\Lambda)$ is locally free, a morphism ${\rm Gr}_1(\Lambda) \rightarrow \mathcal{E}nd(\mathcal{F})$ induces a morphism
\begin{align*}
\mathcal{F} \rightarrow \mathcal{F} \otimes {\rm Gr}_1(\Lambda)^*.
\end{align*}
Note that a morphism ${\rm Gr}_1(\Lambda) \rightarrow \mathcal{E}nd(\mathcal{F})$ may not induce a well-defined morphism $\Lambda \rightarrow \mathcal{E}nd(\mathcal{F})$.

Let $\mathcal{F}$ be a coherent $\Lambda$-module on $\mathcal{X}$. We say that $\mathcal{F}'$ is a \emph{$\Lambda$-submodule} of $\mathcal{F}$, if
\begin{itemize}
\item $\mathcal{F}'$ is a subsheaf of $\mathcal{F}$ as the $\mathcal{O}_{\mathcal{X}}$-module;
\item $\mathcal{F}'$ is preserved under the action of $\Lambda$, i.e. $\Lambda \otimes_{\mathcal{E}nd(\mathcal{F}) }\mathcal{F}' \subseteq \mathcal{F}'$.
\end{itemize}
The set of $\Lambda$-subsheaves of $\mathcal{F}$ can be obtained by tensoring every subsheaf of $\mathcal{F}$ by $\Lambda$, i.e.
\begin{align*}
{\rm SubSf}_{\Lambda}(\mathcal{F})= \Lambda \otimes_{\mathcal{E}nd(\mathcal{F})} {\rm SubSf}(\mathcal{F}),
\end{align*}
where ${\rm SubSf}(\mathcal{F})$ is the set of subsheaves of $\mathcal{F}$, i.e.
\begin{align*}
{\rm SubSf}(\mathcal{F})=\{\mathcal{F}'| \mathcal{F}' \subseteq \mathcal{F}\}.
\end{align*}

Here are some properties of a $\Lambda$-module $\mathcal{F}$, which are easy to check.
\begin{itemize}
    \item The torsion part $\mathcal{F}_{\rm tor}$ of $\mathcal{F}$ is preserved by $\Lambda$. Thus $\mathcal{F}_{\rm tor}$ is a $\Lambda$-submodule.
    \item Let $\mathcal{F}'$ be a $\Lambda$-submodule of $\mathcal{F}$. Then $\mathcal{F}/\mathcal{F}'$ is a $\Lambda$-module.
    \item Let $\mathcal{F}'$ be a $\Lambda$-submodule of $\mathcal{F}$. The saturation $\mathcal{F}'_{\rm sat}$ of $\mathcal{F}'$ is a $\Lambda$-module.
\end{itemize}
Now we give some examples of sheaves of graded algebras.
\subsubsection*{\textbf{Sheaf of Differential Operators}}
Let $D_{\mathcal{X}}$ be the sheaf of differential operators over $\mathcal{X}$. Clearly, $D_{\mathcal{X}}$ has a natural graded structure, of which the filtration $(D_{\mathcal{X}})_i$ is the sheaf of differential operators with order $\leq i$. A derivation $d$ on $D_{\mathcal{X}}$ is a map $d:D_{\mathcal{X}} \rightarrow D_{\mathcal{X}}$ such that
\begin{enumerate}
\item $d(ab)=(da)b+(-1)^{\bar{a}}a(db)$, where $\bar{a}$ is the order of $a$,
\item $d((D_{\mathcal{X}})_i) \subseteq (D_{\mathcal{X}})_{i-1}$,
\item $d^2=0$.
\end{enumerate}
A basic example of a derivation is the Lie bracket $d_{\frac{\partial}{\partial x}}(\bullet):=[\frac{\partial}{\partial x},\bullet]$.

This definition can be extended to any sheaf of graded algebras $\Lambda$. A derivation $d$ of $\Lambda$ is a map $d: \Lambda \rightarrow \Lambda$ such that
\begin{enumerate}
\item $d(ab)=(da)b+(-1)^{\bar{a}}a(db)$.
\item $d(\Lambda_i) \subseteq \Lambda_{i-1}$.
\item $d^2=0$.
\end{enumerate}

Let $v \in {\rm Gr}_1(\Lambda)$. There is a natural derivation $d_v$ defined by the commutator $d_v(a):=[v,a]=va-av$. Now we consider the class of $v$ in ${\rm Gr}_1(\Lambda)$. We use the same notation $v$ for the corresponding class in ${\rm Gr}_1(\Lambda)$. There is a unique morphism $\sigma: {\rm Gr}_1(\Lambda) \rightarrow \mathcal{H}om(\Omega^1_{\mathcal{X}}, \mathcal{O}_{\mathcal{X}})$. The morphism $\sigma$ is called the \emph{symbol} of $\Lambda$.

Denote by $\Theta_{\Lambda}$ be the set of derivations of $\Lambda$. Note that $\Theta_{\Lambda}$ has a natural structure of sheaves. Let $\mathcal{F}$ be a coherent sheaf. A \emph{connection $\nabla$} on $\mathcal{F}$ is a $\mathcal{O}_{\mathcal{X}}$-morphism $\nabla: \Theta_{\Lambda} \rightarrow \mathcal{E}nd(\mathcal{F})$ satisfying the following conditions.
\begin{enumerate}
    \item $\nabla_{f \theta}(s)= f \nabla_{\theta}(s)$.
    \item $\nabla_{\theta}(fs)=\theta(f)s+ f \nabla_{\theta}(s)$.
    \item $\nabla_{[\theta_1,\theta_2]}(s)= [\nabla_{\theta_1},\nabla_{\theta_2}](s)$.
\end{enumerate}
Note that $d_v$ is also a derivation for $v \in {\rm Gr}_1(\Lambda)$. Thus $\nabla_{d_v}$ is a homomorphism of $\mathcal{F}$, which induces an action of $v$ on $\mathcal{F}$. Thus a connection $\nabla$ gives us a well-defined $\Lambda$-action on $\mathcal{F}$, i.e. a $\Lambda$-sheaf $\mathcal{F}$.

\subsubsection*{\textbf{$\mathcal{L}$-Twisted Higgs Bundle}}
Now we consider the example, the $\mathcal{L}$-twisted Higgs bundle. Let $\mathcal{L}$ be a locally free sheaf over $\mathcal{X}$. The sheaf of graded algebras corresponding to $\mathcal{L}$ is defined as $\Lambda_{\mathcal{L}}:={\rm Sym}^{\bullet}(\mathcal{L}^*)$. Note that ${\rm Gr}_1(\Lambda_{\mathcal{L}})=\mathcal{L}^*$. In this case, a morphism $\Lambda_{\mathcal{L}} \rightarrow \mathcal{E}nd(\mathcal{F})$ is uniquely determined by ${\rm Gr}_1(\Lambda_{\mathcal{L}}) \rightarrow \mathcal{E}nd(\mathcal{F})$. If we start with a morphism ${\rm Gr}_1(\Lambda_{\mathcal{L}}) \rightarrow \mathcal{E}nd(\mathcal{F})$, this morphism will give us a well-defined map
\begin{align*}
{\rm Gr}_1(\Lambda_{\mathcal{L}}) \otimes \mathcal{F} \rightarrow \mathcal{F},
\end{align*}
and then,
\begin{align*}
\mathcal{F} \rightarrow \mathcal{F} \otimes ({\rm Gr}_1(\Lambda_{\mathcal{L}}))^* \Rightarrow \mathcal{F} \rightarrow \mathcal{F} \otimes \mathcal{L}.
\end{align*}
The induced map $\mathcal{F} \rightarrow \mathcal{F} \otimes \mathcal{L}$ is exactly an \emph{$\mathcal{L}$-twisted Higgs field}.

\subsection{$p$-Semistability of $\Lambda$-modules}
In this subsection, we study the $p$-semistability of $\Lambda$-modules and prove some properies of Harder-Narasimhan filtrations and Jordan-H\"older filtrations for $\Lambda$-modules.

Let $\mathcal{X}$ be a projective Deligne-Mumford stack and with moduli space $X$ over $S$. We fix a polarization $\mathcal{O}_X(1)$ on $X$ and a generating sheaf $\mathcal{E}$ on $\mathcal{X}$. With respect to the polarizations $\mathcal{O}_X(1)$ and $\mathcal{E}$, we have the reduced modified Hilbert polynomial $p_{\mathcal{E}}(\mathcal{F})$ of a coherent sheaf $\mathcal{F}$ over $\mathcal{X}$. The setup and notations are the same as those in \S 3.2-\S 3.6.

Let $\Lambda$ be a sheaf of graded algebras on $\mathcal{X}$. A $\Lambda$-module $\mathcal{F}$ is \emph{pure of dimension $d$}, if the underlying structure as an $\mathcal{O}_{\mathcal{X}}$-module is pure of dimension $d$.

A $\Lambda$-module $\mathcal{F}$ is \emph{$p$-semistable} (resp. \emph{$p$-stable}), if $\mathcal{F}$ is a pure coherent sheaf and for any $\Lambda$-subsheaf $\mathcal{F}' \subseteq \mathcal{F}$ with $0 < {\rm rk}(\mathcal{F}') < {\rm rk}(\mathcal{F})$, we have
\begin{align*}
p_{\mathcal{E}}(\mathcal{F}',m) \leq p_{\mathcal{E}}(\mathcal{F},m),\quad m\gg 0, \quad (\text{resp. }<).
\end{align*}

\subsubsection*{\textbf{$\Lambda$-Harder-Narasimhan Filtrations}}
Let $\mathcal{F}$ be a purely $d$-dimensional $\Lambda$-sheaf on $\mathcal{X}$. A \emph{destabilizing $\Lambda$-subsheaf} $\mathcal{F}^{\Lambda}_{\rm de}$ is a $\Lambda$-subsheaf of $\mathcal{F}$ such that
\begin{enumerate}
\item For all $\Lambda$-subsheaves $\mathcal{F}' \subseteq \mathcal{F}$ we have $p_{\mathcal{E}}(\mathcal{F}^{\Lambda}_{\rm de}) \geq p_{\mathcal{E}}(\mathcal{F}')$.
\item If $p_{\mathcal{E}}(\mathcal{F}^{\Lambda}_{\rm de}) = p_{\mathcal{E}}(\mathcal{F}')$, we have $\mathcal{F}' \subseteq \mathcal{F}^{\Lambda}_{\rm de}$.
\end{enumerate}

\begin{lem}\label{401}
For any purely $d$-dimensional $\Lambda$-module $\mathcal{F}$ on $\mathcal{X}$, there is a unique destabilizing $\Lambda$-subsheaf $\mathcal{F}_{\rm de}^{\Lambda}$.
\end{lem}

\begin{proof}
The proof of this lemma is similar to \cite[Lemma 1.3.5]{HuLe}. We only give the construction of the destabilizing $\Lambda$-subsheaf.

Consider the set of non-trivial $\Lambda$-subsheaves of the given $\Lambda$-sheaf $\mathcal{F}$
\begin{align*}
{\rm SubSf}_{\Lambda}(\mathcal{F})= \{\mathcal{F}'| \mathcal{F}' \text{ is a $\Lambda$-subsheaf of } \mathcal{F}\}.
\end{align*}
We can define a partial-order on the set ${\rm SubSf}_{\Lambda}(\mathcal{F})$ as follows. Let $\mathcal{F}_1$ and $\mathcal{F}_2$ be two $\Lambda$-subsheaves of $\mathcal{F}$. We say that $\mathcal{F}_1 \leq \mathcal{F}_2$, if $\mathcal{F}_1 \subseteq \mathcal{F}_2$ and $p_{\mathcal{E}}(\mathcal{F}_1) \leq p_{\mathcal{E}}(\mathcal{F}_2)$. We take a maximal subsheaf $\mathcal{F}' \subseteq \mathcal{F}$ with respect to the partial order relation such that the coefficient $\alpha_{\mathcal{E},d}(\mathcal{F}')$ is minimal among all such maximal subsheaves. This subsheaf $\mathcal{F}'$ is exactly the destabilizing $\Lambda$-subsheaf of $\mathcal{F}$.
\end{proof}

\begin{lem}\label{402}
Let $\mathcal{F}$ be a purely $d$-dimensional $\Lambda$-module on $\mathcal{X}$. We have \begin{align*}
\mathcal{F}_{\rm de}^{\Lambda}= (\Lambda \otimes_{ \mathcal{E}nd(\mathcal{F}) } \mathcal{F}_{\rm de})_{\rm sat},
\end{align*}
where $\mathcal{F}_{\rm de}$ is the destabilizing sheaf of $\mathcal{F}$.
\end{lem}

\begin{proof}
We will prove that the sheaf $(\Lambda \otimes_{\mathcal{E}nd(\mathcal{F})} \mathcal{F}_{\rm de})_{\rm sat}$ satisfies the conditions of a destabilizing $\Lambda$-sheaf, and then by the uniqueness of the destabilizing $\Lambda$-subsheaf, we will prove this lemma.

Clearly, the sheaf $\Lambda \otimes_{ \mathcal{E}nd(\mathcal{F}) } \mathcal{F}_{\rm de}$ is a $\Lambda$-subsheaf of $\mathcal{F}$. Thus its saturation $(\Lambda \otimes_{ \mathcal{E}nd(\mathcal{F})} \mathcal{F}_{\rm de})_{\rm sat}$ is also a $\Lambda$-subsheaf of $\mathcal{F}$. Given any $\Lambda$-subsheaf $\mathcal{F}'$ of $\mathcal{F}$, it is also a subsheaf of $\mathcal{F}$. Therefore, we obtain
\begin{align*}
p_{\mathcal{E}}(\mathcal{F}') \leq p_{\mathcal{E}}(\mathcal{F}_{\rm de}).
\end{align*}
This inequality gives
\begin{align*}
p_{\mathcal{E}}(\mathcal{F}')=p_{\mathcal{E}}(\Lambda \otimes_{\mathcal{E}nd(\mathcal{F})} \mathcal{F}') \leq p_{\mathcal{E}}(\Lambda \otimes_{\mathcal{E}nd(\mathcal{F})} \mathcal{F}_{\rm de}).
\end{align*}
We also have the following inequality about the saturation
\begin{align*}
p_{\mathcal{E}}(\Lambda \otimes_{\mathcal{E}nd(\mathcal{F})} \mathcal{F}_{\rm de}) \leq p_{\mathcal{E}}((\Lambda \otimes_{\mathcal{E}nd(\mathcal{F})} \mathcal{F}_{\rm de})_{\rm sat}).
\end{align*}
This implies that
\begin{align*}
p_{\mathcal{E}}(\mathcal{F}') \leq p_{\mathcal{E}}((\Lambda \otimes_{\mathcal{E}nd(\mathcal{F})} \mathcal{F}_{\rm de})_{\rm sat}).
\end{align*}
This finishes the proof of the first condition of the destabilizing $\Lambda$-subsheaf.

Now we assume that $p_{\mathcal{E}}(\mathcal{F}')=p_{\mathcal{E}}((\Lambda \otimes_{\mathcal{E}nd(\mathcal{F})} \mathcal{F}_{\rm de})_{\rm sat})$. As a subsheaf of $\mathcal{F}$, we have
\begin{align*}
p_{\mathcal{E}}(\mathcal{F}') \leq p_{\mathcal{E}}(\mathcal{F}_{\rm de}).
\end{align*}
If $p_{\mathcal{E}}(\mathcal{F'})=p_{\mathcal{E}}(\mathcal{F}_{\rm de})$, then $\mathcal{F}'$ is a subsheaf of $\mathcal{F}_{\rm de}$. Thus $\mathcal{F}'$ is a $\Lambda$-subsheaf of $(\Lambda \otimes_{\mathcal{E}nd(\mathcal{F})} \mathcal{F}_{\rm de})_{\rm sat}$. If $p_{\mathcal{E}}(\mathcal{F}')<p_{\mathcal{E}}(\mathcal{F}_{\rm de})$, then
\begin{align*}
p_{\mathcal{E}}(\mathcal{F}')= p_{\mathcal{E}}( \Lambda \otimes_{\mathcal{E}nd(\mathcal{F})} \mathcal{F}')<p_{\mathcal{E}}(\Lambda \otimes_{\mathcal{E}nd(\mathcal{F})} \mathcal{F}_{\rm de})\leq p_{\mathcal{E}} ((\Lambda \otimes_{\mathcal{E}nd(\mathcal{F})} \mathcal{F}_{\rm de})_{\rm sat}).
\end{align*}
This contradicts the assumption that $p_{\mathcal{E}}(\mathcal{F}')=p_{\mathcal{E}}((\Lambda \otimes_{\mathcal{E}nd(\mathcal{F})} \mathcal{F}_{\rm de})_{\rm sat})$.
\end{proof}

Now we give the definition of the $\Lambda$-Harder-Narasimhan filtration. Let $\mathcal{F}$ be a $\Lambda$-module of pure dimension $d$ on $\mathcal{X}$. A \emph{$\Lambda$-Harder-Narasimhan filtration} of $\mathcal{F}$ is an increasing filtration
\begin{align*}
0 = {\rm HN}_0(\mathcal{F}) \subseteq {\rm HN}_1(\mathcal{F}) \subseteq \dots \subseteq {\rm HN}_l(\mathcal{F})= \mathcal{F},
\end{align*}
such that
\begin{enumerate}
\item The subsheaves ${\rm HN}_i(\mathcal{F})$ are $\Lambda$-modules for $1 \leq i \leq l$.
\item The factors $gr_i^{\rm HN}(\mathcal{F})={\rm HN}_{i}(\mathcal{F}) / {\rm HN}_{i-1}(\mathcal{F})$ are $p$-semistable $\Lambda$-modules for $1 \leq i \leq l$ of dimension $d$.
\item Denote by $p_i$ the reduced modified Hilbert polynomial $p_{\mathcal{E}}(gr_i^{\rm HN}(\mathcal{F}))$ such that
    \begin{align*}
        p_{\rm max}(\mathcal{F}):=p_1 > \dots > p_l=: p_{\rm min}(\mathcal{F}).
    \end{align*}
\end{enumerate}

\begin{prop}\label{403}
Let $\mathcal{F}$ be a $\Lambda$-module of pure dimension $d$ on $\mathcal{X}$. There is a unique $\Lambda$-Harder-Narasimhan filtration of $\mathcal{F}$.
\end{prop}

\begin{proof}
The existence of the $\Lambda$-Harder-Narasimhan filtration is proved by induction. In the base step, we take ${\rm HN}_1(\mathcal{F}) = \mathcal{F}^{\Lambda}_{\rm de}$. Then we consider the quotient sheaf $\mathcal{F}/\mathcal{F}^{\Lambda}_{\rm de}$, which is also a $\Lambda$-sheaf. By induction, we can assume that there is a $\Lambda$-Harder-Narasimhan filtration of $\mathcal{F}/\mathcal{F}^{\Lambda}_{\rm de}$. Thus we have a $\Lambda$-Harder-Narasimhan filtration of $\mathcal{F}$ by taking the preimage of the $\Lambda$-Harder-Narasimhan filtration of $\mathcal{F}/\mathcal{F}^{\Lambda}_{\rm de}$.

The proof of the uniqueness of the $\Lambda$-Harder-Narasimhan filtration is the same as the proof of the classical case \cite[Theorem 1.3.4]{HuLe}.
\end{proof}

\subsubsection*{\textbf{$\Lambda$-Jordan-H\"older Filtrations}}
Let $\mathcal{F}$ be a $p$-semistable $\Lambda$-module on $\mathcal{X}$ with the reduced modified Hilbert polynomial $p_{\mathcal{E}}(\mathcal{F})$. A \emph{$\Lambda$-Jordan-H\"{o}lder Filtration} of $\mathcal{F}$ is an increasing filtration
\begin{align*}
0 \subseteq {\rm JH}_0(\mathcal{F}) \subseteq {\rm JH}_1(\mathcal{F}) \subseteq \dots \subseteq {\rm JH}_l(\mathcal{F})= \mathcal{F}
\end{align*}
such that the factors $gr_i^{\rm JH}(\mathcal{F})={\rm JH}_{i}(\mathcal{F}) / {\rm JH}_{i-1}(\mathcal{F})$ are $p$-stable with respect to the reduced modified Hilbert polynomial $p_{\mathcal{E}}(\mathcal{F})$ for $1 \leq i \leq l$.

\begin{prop}\label{404}
Let $\mathcal{F}$ be a $\Lambda$-semistable sheaf on $\mathcal{X}$. There is a $\Lambda$-Jordan-H\"{o}lder filtration of $\mathcal{F}$, and the graded sheaf $gr^{\rm JH}(\mathcal{F}):=\bigoplus_i gr_i^{\rm JH}(\mathcal{F})$ does not depend on the choice of the Jordan-H\"{o}lder filtration.
\end{prop}

\begin{proof}
The proof of the existence and uniqueness (up to isomorphism between the graded sheaves) is the same as \cite[Proposition 1.5.2]{HuLe}.
\end{proof}

Two $p$-semistable $\Lambda$-sheaves $\mathcal{F}_1$ and $\mathcal{F}_2$ with the same reduced modified Hilbert polynomial are called \emph{$S$-equivalent} if the graded sheaves of the $\Lambda$-Jordan-H\"older filtrations are isomorphic, i.e.
\begin{align*}
gr^{\rm JH}(\mathcal{F}_1) \cong gr^{\rm JH}(\mathcal{F}_2).
\end{align*}

\subsection{Boundedness of $\Lambda$-modules \Romannum{2}}
Let $\widetilde{\mathfrak{F}}^{ss}_{\Lambda}(P)$ be the set of $p$-semistable $\Lambda$-modules of pure dimension $d$ with the modified Hilbert polynomial $P$. In this subsection, we study the upper bound the following set
\begin{align*}
\{h^0(\mathcal{X},\mathcal{F}\otimes \mathcal{E}^{\vee} \otimes \pi^* \mathcal{O}_{\mathcal{X}}(m) ) \text{ } | \text{ } \mathcal{F} \in \widetilde{\mathfrak{F}}^{ss}_{\Lambda}(P)\}.
\end{align*}
We prove that there is an upper-bound for the above set (see Proposition \ref{407}).

\begin{lem}\label{405}
Let $\mathcal{F}$ be a $\Lambda$-module of pure dimension $d$ and rank $r$ over $\mathcal{X}$. Let $\mathcal{F}'$ be a subsheaf of $\mathcal{F}$, not necessarily preserved by $\Lambda$. Denote by $\mathcal{F}'_r$ the image of the morphism $\Lambda_r \otimes_{\mathcal{O}_{\mathcal{X}}} \mathcal{F}' \rightarrow \mathcal{F}$. Then $\mathcal{F}'_r$ is a $\Lambda$-subsheaf of $\mathcal{F}$.
\end{lem}

\begin{proof}
The proof of this lemma is the same as \cite[Lemma 3.2]{Simp2}.
\end{proof}

\begin{lem}\label{406}
Suppose that there is an integer $m_0$ such that $\pi_*{\rm Gr}_1(\Lambda)(m_0)$ is generated by global sections. Then, there is an integer $m_1$ such that ${\rm Gr}_1(\Lambda)$ is $m_1$-regular. For any $p$-semistable $\Lambda$-module $\mathcal{F}$ of pure dimension $d$ and rank $r$, and any subsheaf $\mathcal{F}' \subseteq \mathcal{F}$, we have
\begin{align*}
\hat{\mu}_{\mathcal{E}}(\mathcal{F}') \leq \hat{\mu}_{\mathcal{E}}(\mathcal{F})+br,
\end{align*}
where $b$ is an integer depending on $m_1$ (or $m_0$) and the generating sheaf $\mathcal{E}$.
\end{lem}

\begin{proof}
By assumption, we know  $\pi_*{\rm Gr}_1(\Lambda)$ is $m_0$-regular. Therefore, by Lemma \ref{309}, we can find an integer $m_1$ such that ${\rm Gr}_1(\Lambda)$ is $m_1$-regular. This finishes the proof of the first statement.

The proof of the second part of this lemma is similar to \cite[Lemma 3.3]{Simp2}. We include the proof here for completeness. Note that a $p$-semistable $\Lambda$-module may not be a $p$-semistable coherent sheaf. Denote by $\mathcal{G}$ the destabilizing sheaf of $\mathcal{F}$ (not the $\Lambda$-destabilizing sheaf). By the definition of the destabilizing sheaf, it suffices to find a positive integer $m$ and prove the inequality
\begin{align*}
\hat{\mu}_{\mathcal{E}}(\mathcal{G}) \leq \hat{\mu}_{\mathcal{E}}(\mathcal{F})+mr.
\end{align*}
Let $\mathcal{G}_i$ be the image of $\Lambda_i \otimes \mathcal{G}$ in $\mathcal{F}$ for $i \geq 0$. Note that $\mathcal{G}_0=\mathcal{G}$. By the definition of sheaves of graded algebras, we have the following surjections of coherent sheaves
\begin{align*}
\Lambda_1 \otimes \mathcal{G}_i \rightarrow \mathcal{G}_{i+1} \rightarrow 0
\end{align*}
for $0 \leq i \leq r$. The above surjections induce
\begin{align*}
{\rm Gr}_1(\Lambda) \otimes (\mathcal{G}_i / \mathcal{G}_{i-1}) \rightarrow \mathcal{G}_{i+1}/\mathcal{G}_i \rightarrow 0
\end{align*}
for $1 \leq i \leq r$. We know that the coherent sheaf $F_{\mathcal{E}}({\rm Gr}_1(\Lambda))(m_1)$ is generated by global sections. Thus we have the following surjective map
\begin{align*}
V \otimes \mathcal{E} \otimes \pi^* \mathcal{O}_X(-m_1) \rightarrow {\rm Gr}_1(\Lambda) \rightarrow 0,
\end{align*}
where
\begin{align*}
V=H^0(\mathcal{X}, {\rm Gr}_1(\Lambda) \otimes \mathcal{E}^{\vee} \otimes \pi^* \mathcal{O}_X(m_1))=H^0(X,F_{\mathcal{E}}({\rm Gr}_1(\Lambda))(m_1)).
\end{align*}
This surjection induces the following one
\begin{align*}
V \otimes \mathcal{E} \otimes (\mathcal{G}_i / \mathcal{G}_{i-1}) \otimes \pi^* \mathcal{O}_X(-m_1) \rightarrow \mathcal{G}_{i+1}/\mathcal{G}_i \rightarrow 0.
\end{align*}
We take the quotient $\mathcal{O}_{\mathcal{X}}$-module $\mathcal{Q}_i$ of $\mathcal{G}_i$ with smallest reduced modified Hilbert polynomial of any quotient $\mathcal{O}_{\mathcal{X}}$-module, $0 \leq i \leq r$. By definition of destabilizing sheaf, $\mathcal{Q}_0=\mathcal{G}_0=\mathcal{G}$.

We discuss the quotient $\mathcal{Q}_{i+1}$ in the following two cases.
\begin{enumerate}
\item If $\mathcal{Q}_{i+1}$ has a nontrivial subsheaf which is a quotient of $\mathcal{G}_{i}$, by the smallest property of $\mathcal{Q}_i$, we have
\begin{align*}
p_{\mathcal{E}}(\mathcal{Q}_{i}) \leq p_{\mathcal{E}}(\mathcal{Q}_{i+1}).
\end{align*}
This inequality induces that $\hat{\mu}_{\mathcal{E}}(\mathcal{Q}_{i}) \leq \hat{\mu}_{\mathcal{E}}(\mathcal{Q}_{i+1})$.
\item Otherwise, $\mathcal{Q}_{i+1}$ is a quotient of $\mathcal{G}_{i+1} / \mathcal{G}_{i}$. By the surjection
\begin{align*}
V \otimes \mathcal{E} \otimes (\mathcal{G}_i / \mathcal{G}_{i-1}) \otimes \pi^* \mathcal{O}_X(-m_1) \rightarrow \mathcal{G}_{i+1}/\mathcal{G}_i \rightarrow 0,
\end{align*}
we obtain that $\mathcal{Q}_{i+1}$ is a quotient of $V \otimes \mathcal{E} \otimes (\mathcal{G}_i / \mathcal{G}_{i-1})\otimes \pi^* \mathcal{O}_X(-m_1)$, we have
    \begin{align*}
    p_{\mathcal{E}}(V \otimes \mathcal{E} \otimes \mathcal{Q}_{i}\otimes \pi^* \mathcal{O}_X(-m_1) ) \leq p_{\mathcal{E}}(\mathcal{Q}_{i+1}).
    \end{align*}
    Equivalently,
    \begin{align*}
    p_{\mathcal{E}}(V \otimes \mathcal{E} \otimes \mathcal{Q}_{i},n-m_1) \leq p_{\mathcal{E}}(\mathcal{Q}_{i+1},n).
    \end{align*}
    Therefore,
    \begin{align*}
    \hat{\mu}_{\mathcal{E}}(\mathcal{E} \otimes \mathcal{Q}_{i})-m_1 \leq \hat{\mu}_{\mathcal{E}}(\mathcal{Q}_{i+1}).
    \end{align*}
    For the generating sheaf $\mathcal{E}$, we can find a surjection
    \begin{align*}
    \mathcal{O}_{\mathcal{X}}(-m_{\mathcal{E}})^{\oplus N} \twoheadrightarrow \mathcal{E}.
    \end{align*}
    Therefore, we have
    \begin{align*}
    \hat{\mu}_{\mathcal{E}}(\mathcal{Q}_{i}) -m_1 -m_{\mathcal{E}}=
    \hat{\mu}_{\mathcal{E}}(\mathcal{O}_{\mathcal{X}}(-m_{\mathcal{E}})^{\oplus N} \otimes \mathcal{Q}_{i}) -m_1 \leq \hat{\mu}_{\mathcal{E}}(\mathcal{E} \otimes \mathcal{Q}_{i})-m_1 \leq \hat{\mu}_{\mathcal{E}}(\mathcal{Q}_{i+1})
    \end{align*}
\end{enumerate}
Take $b=m_1+m_{\mathcal{E}}$. In conclusion, we always have
\begin{align*}
\hat{\mu}_{\mathcal{E}}(\mathcal{Q}_{i})-b \leq \hat{\mu}_{\mathcal{E}}(\mathcal{Q}_{i+1}).
\end{align*}
Taking the sum over $i$, we have
\begin{align*}
\hat{\mu}_{\mathcal{E}}(\mathcal{Q}_{0}) \leq \hat{\mu}_{\mathcal{E}}(\mathcal{Q}_{r})+br.
\end{align*}
Now we consider the polynomial $p_{\mathcal{E}}(\mathcal{Q}_{r})$. By the definition of $\mathcal{Q}_{r}$, we have $p_{\mathcal{E}}(\mathcal{Q}_{r}) \leq p_{\mathcal{E}}(\mathcal{G}_{r})$. Moreover, we have $p_{\mathcal{E}}(\mathcal{G}_{r}) \leq p_{\mathcal{E}}((\mathcal{G}_{r})_{\rm sat})$, where $(\mathcal{G}_{r})_{\rm sat}$ is the saturation of $\mathcal{G}_{r}$. By Lemma \ref{405}, the sheaf $\mathcal{G}_{r}$ is a $\Lambda$-module. Thus, we have
\begin{align*}
p_{\mathcal{E}}(\mathcal{Q}_{r}) \leq p_{\mathcal{E}}(\mathcal{G}_{r}) \leq p_{\mathcal{E}}((\mathcal{G}_{r})_{\rm sat}) \leq p_{\mathcal{E}}(\mathcal{F}).
\end{align*}
The above inequalities of reduced modified Hilbert polynomials imply the following inequalities of slopes
\begin{align*}
\hat{\mu}_{\mathcal{E}}(\mathcal{Q}_{r}) \leq \hat{\mu}_{\mathcal{E}}(\mathcal{G}_{r}) \leq \hat{\mu}_{\mathcal{E}}((\mathcal{G}_{r})_{\rm sat}) \leq \hat{\mu}_{\mathcal{E}}(\mathcal{F}).
\end{align*}
Note that $\hat{\mu}_{\mathcal{E}}(\mathcal{Q}_{0})$ is exactly $\hat{\mu}_{\mathcal{E}}(\mathcal{G})$. We have
\begin{align*}
\hat{\mu}_{\mathcal{E}}(\mathcal{G})=\hat{\mu}_{\mathcal{E}}(\mathcal{Q}_{0}) \leq \hat{\mu}_{\mathcal{E}}(\mathcal{Q}_{r})+br \leq \hat{\mu}_{\mathcal{E}}(\mathcal{F})+br.
\end{align*}
\end{proof}

\begin{prop}\label{407}
Let $\mathcal{X}$ be a projective stack with polarizations $\mathcal{E}$, $\mathcal{O}_X(1)$. Let $\widetilde{\mathfrak{F}}^{ss}_{\Lambda}(P)$ be the set of $p$-semistable $\Lambda$-sheaves of pure dimension $d$ with the modified Hilbert polynomial $P$. There is an upper-bound for the set
\begin{align*}
\{h^0(\mathcal{X},\mathcal{F}\otimes \mathcal{E}^{\vee} \otimes \pi^* \mathcal{O}_{\mathcal{X}}(m) ) \text{ } | \text{ } \mathcal{F} \in \widetilde{\mathfrak{F}}^{ss}_{\Lambda}(P)\}.
\end{align*}
The upper-bound depends on the following data
\begin{itemize}
\item $P$: the polynomial;
\item $\widetilde{m}$: the integer such that $\pi_* \mathcal{E}nd_{\mathcal{O}_{\mathcal{X}}}(\mathcal{E})(\widetilde{m})$ is generated by global sections;
\item $\deg(\mathcal{O}_X(1))$: the degree of $\mathcal{O}_X(1)$;
\item $b$: the integer found in Lemma \ref{406}.
\end{itemize}
\end{prop}

Note that the dimension $d$ and the rank $r$ are determined by the given integer polynomial $P$.

\begin{proof}
Let $\mathcal{F}$ be an element in $\widetilde{\mathfrak{F}}^{ss}_{\Lambda}$. Although $\mathcal{F}$ is a $p$-semistable $\Lambda$-module, the coherent sheaf $\mathcal{F}$ may not be $p$-semistable as an $\mathcal{O}_{\mathcal{X}}$-module. Thus we take the Harder-Narasimhan filtration
\begin{align*}
0 = \mathcal{F}_0 \subseteq \mathcal{F}_1 \subseteq \cdots \subseteq \mathcal{F}_l =\mathcal{F}
\end{align*}
of $\mathcal{F}$. Denote by $gr_i^{\rm HN}(\mathcal{F})=\mathcal{F}_{i}/\mathcal{F}_{i-1}$ the quotient sheaf, where $1 \leq i \leq l$. Let $r_i$ be the rank of $F_{\mathcal{E}}(gr_i^{\rm HN}(\mathcal{F}))$. By Corollary \ref{312}, we know that there is an integer $B_i$ such that
\begin{align*}
h^0(\mathcal{X},\mathcal{F}_{i}/\mathcal{F}_{i-1} \otimes \mathcal{E}^{\vee} \otimes \pi^* \mathcal{O}_{\mathcal{X}}(m)) \leq r_i \frac{ (\hat{\mu}_{\mathcal{E}}(\mathcal{F}_{i}/\mathcal{F}_{i-1})+m+B_i)^d }{ d! }.
\end{align*}
By Lemma \ref{406}, we can find an integer $b_i$, which only depends on $\Lambda$, $\mathcal{E}$ and $r$, such that
\begin{align*}
\hat{\mu}_{\mathcal{E}}(\mathcal{F}_{i}/\mathcal{F}_{i-1}) \leq \hat{\mu}_{\mathcal{E}}(\mathcal{F})+b_i.
\end{align*}
Take $B={\rm sup}\{b_i+B_i, 1 \leq i \leq l\}$. We have the following inequality
\begin{align*}
h^0(\mathcal{X},\mathcal{F}_{i}/\mathcal{F}_{i-1} \otimes \mathcal{E}^{\vee}\otimes \pi^* \mathcal{O}_{\mathcal{X}}(m) ) \leq r_i \frac{ ( \hat{\mu}_{\mathcal{E}}(\mathcal{F})+m+B )^d }{d!}.
\end{align*}
Thus we have
\begin{align*}
h^0(\mathcal{X},\mathcal{F}\otimes \mathcal{E}^{\vee}\otimes \pi^* \mathcal{O}_{\mathcal{X}}(m) ) \leq \sum_{i=1}^{l} h^0(\mathcal{X},\mathcal{F}_{i}/\mathcal{F}_{i-1} \otimes \mathcal{E}^{\vee}\otimes \pi^* \mathcal{O}_{\mathcal{X}}(m)) \leq r \frac{ ( \hat{\mu}_{\mathcal{E}}(\mathcal{F})+m+B )^d }{d!},
\end{align*}
where $r$ is the rank of $F_{\mathcal{E}}(\mathcal{F})$. Note that $r$, $d$ and $\hat{\mu}_{\mathcal{E}}(\mathcal{F})$ are uniquely determined by the given polynomial $P$. This finishes the proof of this lemma.
\end{proof}

\section{$\Lambda$-Quot-Functors}
In this section, we define the $\Lambda$-quot-functor on Deligne-Mumford stacks and prove one of the main results in this paper that the $\Lambda$-quot-functor is represented by an algebraic space (Theorem \ref{501}). The method of proving this property is based on a theorem by Artin \cite[Theorem 5.3]{Art}. This theorem states that a functor is representable by an algebraic space if the functor satisfies a series of conditions. Therefore, proving the representability of the $\Lambda$-quot-functor is equivalent to check all conditions in the theorem by Artin. As an application of this result, we show that the $\Lambda$-quot-fucntor on a projective Deligne-Mumford stack is represented by a quasi-projective scheme (see Theorem \ref{511}).

In \S 5.1, we give the definitions, and state the main theorem (Theorem \ref{501}). In \S 5.2, we review the Artin's theorem (Theorem \ref{502}) and some necessary backgrounds about the conditions listed in the theorem. There are many good references about the background here. We refer the reader to \cite{Art,CasaWise,Hall} for more details.

The proof of the representability of the $\Lambda$-quot-functor is discussed in \S 5.3. We discuss the deformation and obstruction theory of $\Lambda$-quot-functors carefully in \S 5.3.4 and \S 5.3.5, and the deformation theory will help us to calculate the tangent space of the moduli space of $\Lambda$-modules in \S 6.

After we finish the proof of Theorem \ref{501}, we consider the case that $\mathcal{X}$ is a projective Deligne-Mumford stack on a scheme $S$ in \S 5.4. In this case, the $\Lambda$-quot-functor is represented by a quasi-projective scheme. Olsson and Starr considered these problems for quot-functors \cite[\S 6]{OlSt}, and we extend their approach to $\Lambda$-quot-functors. At the end of this section, we prove the boundedness of the family of $p$-semistable $\Lambda$-modules (Corollary \ref{514} in \S 5.5).

\subsection{Definitions and Results}
Let $S$ be an algebraic space, which is locally of finite type over an algebraically closed field $k$, and let $\mathcal{X}$ be a separated and locally finitely-presented Deligne-Mumford stack over $S$. Denote by $(\text{Sch}/S)$ the category of $S$-schemes with respect to the big \'etale topology or fppf topology. Let $\Lambda$ be a sheaf of graded algebras. We take a coherent $\mathcal{O}_{\mathcal{X}}$-module $\mathcal{G}$, which is not necessary to be a $\Lambda$-module, and define the functor
\begin{align*}
\widetilde{{\rm Quot}}_{\Lambda}(\mathcal{G},\mathcal{X}): (\text{Sch}/S)^{{\rm op}} \rightarrow \text{Set}
\end{align*}
as follows. For each $S$-scheme $T$, we define $\widetilde{{\rm Quot}}_{\Lambda}(\mathcal{G},\mathcal{X})(T)$ to be the set of $\mathcal{O}_{\mathcal{X}_T}$-module quotients $\mathcal{G}_T \rightarrow \mathcal{F}_T$ such that
\begin{enumerate}
\item $\mathcal{F}_T \in \widetilde{{\rm Quot}}(\mathcal{G})(T)$;
\item $\mathcal{F}_T$ is a $\Lambda_{T}$-module.
\end{enumerate}
The functor $\widetilde{{\rm Quot}}_{\Lambda}(\mathcal{G},\mathcal{X})$ is called the \emph{$\Lambda$-quot-functor}, and we will use the notation $\widetilde{{\rm Quot}}_{\Lambda}(\mathcal{G})$ for simplicity.

Recall that the moduli space of Higgs bundles over a smooth projective variety is an algebraic stack (see \cite[\S 7]{CasaWise}). With the same idea and proof, the $\Lambda$-quot-functor $\widetilde{{\rm Quot}}_{\Lambda}(\mathcal{G})$ has a natural stack structure. In other words, the $\Lambda$-quot-functor is a sheaf with respect to the big \'etale topology of $(\text{Sch}/S)$.

Given a coherent sheaf $\mathcal{F}$, $\mathcal{F}$ can be equipped with distinct $\Lambda$-structures, which are defined by the action of $\Lambda$ on $\mathcal{F}$. As an example, when $\Lambda$ acts trivially on $\mathcal{F}$, the action of $\Lambda$ on $\mathcal{F}$ is the same as $\mathcal{O}_{\mathcal{X}}$ on $\mathcal{F}$. In this case, the $\Lambda$-module $\mathcal{F}$ is exactly the same as its $\mathcal{O}_{\mathcal{X}}$-module structure. A $\Lambda$-structure on $\mathcal{F}$ is given by a morphism $\Lambda \rightarrow \mathcal{E}nd(\mathcal{F})$. The set of all morphisms
\begin{align*}
{\rm Hom}_{\mathcal{O}_{\mathcal{X}}}(\Lambda,\mathcal{E}nd(\mathcal{F})) \cong {\rm Hom}_{\mathcal{O}_{\mathcal{X}}}(\Lambda \otimes \mathcal{F},\mathcal{F})
\end{align*}
gives us all possible $\Lambda$-structures on $\mathcal{F}$. Thus a $\Lambda$-module $\mathcal{F}$ is a pair $(\mathcal{F},\Phi)$, where $\mathcal{F}$ is a coherent sheaf and $\Phi: \Lambda \otimes \mathcal{F} \rightarrow \mathcal{F}$ is a morphism. Based on the discussion above, $\widetilde{{\rm Quot}}_{\Lambda}(\mathcal{G})(T)$ is the set of pairs $(\mathcal{F}_T,\Phi_T)$ such that
\begin{enumerate}
\item $\mathcal{G}_T \rightarrow \mathcal{F}_T \in \widetilde{{\rm Quot}}(\mathcal{G})(T)$;
\item $\Phi_T: \Lambda_T \otimes \mathcal{F}_T \rightarrow \mathcal{F}_T$ is an $\mathcal{O}_{\mathcal{X}_T}$-morphism.
\end{enumerate}

By definition, we know that $\Lambda$ is a sheaf of graded algebras, which may not be a coherent sheaf. When constructing the moduli space and proving some properties, a coherent sheaf is a better option. Note that ${\rm Gr}_1(\Lambda)$ is a coherent sheaf and ${\rm Gr}_1(\Lambda)$ contains all generators of $\Lambda$. Therefore, a $\Lambda$-structure on $\mathcal{F}$ induces a ${\rm Gr}_1(\Lambda)$-structure on $\mathcal{F}$. We have the following injective map
\begin{align*}
{\rm Hom}_{\mathcal{O}_{\mathcal{X}}}(\Lambda,\mathcal{E}nd(\mathcal{F})) \hookrightarrow {\rm Hom}_{\mathcal{O}_{\mathcal{X}}}({\rm Gr}_1(\Lambda),\mathcal{E}nd(\mathcal{F}))
\end{align*}
as we discussed in \S 4.1. Thus we work on the morphism ${\rm Gr}_1(\Lambda) \rightarrow \mathcal{E}nd(\mathcal{F})$ in some special situations.

The goal of this section is to prove the following theorem.
\begin{thm}\label{501}
The $\Lambda$-quot-functor $\widetilde{{\rm Quot}}_{\Lambda}(\mathcal{G})$ is represented by a separated and locally finitely presented algebraic space.
\end{thm}

Denote by ${\rm Quot}_{\Lambda}(\mathcal{G})$ the algebraic space representing $\widetilde{{\rm Quot}}_{\Lambda}(\mathcal{G})$.

\subsection{A Theorem by Artin}
Before we prove Theorem \ref{501}, we review some properties of a moduli problem and a theorem by Artin \cite[Theorem 5.3]{Art}. We will use the Artin's theorem to prove Theorem \ref{501}. In this subsection, we always assume that $S$ is an algebraic space, which is locally of finite type over an algebraically closed field $k$, and a moduli problem
\begin{align*}
\widetilde{F}: (\text{Sch}/S)^{\rm op} \rightarrow \text{Set}
\end{align*}
is a presheaf over $(\text{Sch}/S)$ with respect to the big \'etale topology or fppf topology, which can be considered as a category fibered in groupoids.

\subsubsection{\textbf{Locally of Finite Presentation}}
Let $\widetilde{F} \rightarrow \widetilde{G}$ be a morphism of moduli problems, which is considered as a morphism between CFG. We say that the morphism is \emph{locally of finite presentation}, if for every filtered colimit of $\mathcal{O}_S$-algebras $A= \lim\limits_{\longrightarrow} A_i$ and every commutative diagram
\begin{center}
\begin{tikzcd}
\Spec A \arrow[d] \arrow[r]   & \widetilde{F} \arrow[d]   \\
\lim\limits_{\longrightarrow} \Spec A_i \arrow[r] \arrow[ru, dashrightarrow, "\exists !"]& \widetilde{G}
\end{tikzcd},
\end{center}
there exists a unique dashed arrow lifting the morphism $\lim\limits_{\longrightarrow} \Spec A_i \rightarrow \widetilde{G}$. A moduli problem $\widetilde{F}$ is \emph{locally of finite presentation} if the morphism $\widetilde{F} \rightarrow S$ is locally of finite presentation.

\subsubsection{\textbf{Integrability (Effectivity)}}
There are various definitions of \emph{integrability}, which is also called \emph{effectivity}. Most of the definitions of integrability are similar, and the difference comes the condition on the map
\begin{align*}
\widetilde{F}(\bar{A})\rightarrow \lim\limits_{\longleftarrow}\widetilde{F}(\bar{A}/\mathfrak{m}^{n+1}), \quad n \geq 1.
\end{align*}
The one we take in this paper comes from \cite{Art}. We refer the readers to \cite{CasaWise,Hall} for the other definitions of integrability.

Let $\widetilde{F}: (\text{Sch}/S)^{\rm op} \rightarrow \text{Set}$ be a moduli problem. Let $\bar{A}$ be a complete noetherian local $\mathcal{O}_S$-algebra and denote by $\mathfrak{m}$ the maximal ideal of $\bar{A}$. We prefer to use the notation $F(\bar{A})$ instead of $F(\Spec (\bar{A}))$. Given a positive integer $n$, as a contravariant functor (for $S$-schemes), we have a natural map $\widetilde{F}(\bar{A})\rightarrow \widetilde{F}(\bar{A}/\mathfrak{m}^{n+1})$. Thus we have a canonical map
\begin{align*}
\widetilde{F}(\bar{A})\rightarrow \lim\limits_{\longleftarrow}\widetilde{F}(\bar{A}/\mathfrak{m}^{n+1}).
\end{align*}
Let $(\mathcal{F}_{n})_{n \geq 1}$ be an element in $\lim\limits_{\longleftarrow}\widetilde{F}(\bar{A}/\mathfrak{m}^{n+1})$. If there is an element $\mathcal{F}' \in \widetilde{F}(\bar{A})$ such that $\mathcal{F}'$ induces $\mathcal{F}_1 \in \widetilde{F}(\bar{A}/\mathfrak{m}^2)$, then we say that the map $\widetilde{F}(\bar{A})\rightarrow \lim\limits_{\longleftarrow}\widetilde{F}(\bar{A}/\mathfrak{m}^{n+1})$ has a \emph{dense image}.

The moduli problem $\widetilde{F}$ is \emph{integrable} if for every complete noetherian local ring $\bar{A}$, the canonical map $\widetilde{F}(\bar{A})\rightarrow \lim\limits_{\longleftarrow}\widetilde{F}(\bar{A}/\mathfrak{m}^{n+1})$ is injective and the image is dense in $\lim\limits_{\longleftarrow}\widetilde{F}(\bar{A}/\mathfrak{m}^{n+1})$.

\subsubsection{\textbf{Homogeneity}}
An \emph{infinitesimal extension} of $S$-schemes is a closed embedding $T \hookrightarrow T'$ such that the ideal sheaf $I_{T/T'}$ is nilpotent. A well-known example is the ring of dual numbers $\mathbb{D}=\mathbb{Z}[\epsilon]/(\epsilon^2)$. There is a natural embedding $T \hookrightarrow T[\epsilon]:=T \times_{\Spec (\mathbb{Z})} \Spec (\mathbb{D})$, of which the ideal is nilpotent.

Let $T \hookrightarrow T'$ be an infinitesimal extension of $S$-schemes and let $f: T \rightarrow R$ be an affine $S$-morphism. Then there is a universal $S$-scheme $R'$ completing the following diagram \cite[\S 2.1]{Wise}
\begin{center}
\begin{tikzcd}
T \arrow[d, "f"] \arrow[r, hook]   & T' \arrow[d, dashrightarrow]   \\
R \arrow[r, dashrightarrow] & R'
\end{tikzcd}.
\end{center}
A CFG $\widetilde{F}$ is \emph{homogeneous} if for each diagram above, the natural morphism
\begin{align*}
\widetilde{F}(T) \rightarrow \widetilde{F}(T') \times_{\widetilde{F}(R')} \widetilde{F}(R)
\end{align*}
is an equivalence of categories. The \emph{homogeneity} is also called the \emph{Schlessinger's condition} \cite{Sch}.

Let $f: \widetilde{F} \rightarrow \widetilde{F}'$ be a morphism of moduli problems. We say that $f$ is \emph{homogeneous}, or that $\mathcal{X}$ is \emph{homogeneous} over $\mathcal{Y}$, if for any $S$-scheme $T$ and any morphism $T \rightarrow \widetilde{F}'$, the fiber product $\widetilde{F} \times_{\widetilde{F}'} T$ is homogeneous. This property is also called the \emph{relative homogeneity}.

Now we consider an example. Let $S$ be an algebraic space, which is locally of finite type over an algebraically closed field $k$ and let $\mathcal{X}$ be a separated and locally finitely-presented Deligne-Mumford stack over $S$. Let $\mathcal{E}$ and $\mathcal{F}$ be two quasi-coherent sheaves over $\mathcal{X}$. Denote by
\begin{align*}
\underline{\mathscr{H}om}(\mathcal{E},\mathcal{F}): (\text{Sch}/S)^{\rm op} \rightarrow {\rm Set}
\end{align*}
the moduli problem of homomorphisms between $\mathcal{E}$ and $\mathcal{F}$ such that
\begin{align*}
\underline{\mathscr{H}om}(\mathcal{E},\mathcal{F})(T):={\rm Hom}(\mathcal{E}_T,\mathcal{F}_T)
\end{align*}
for every $S$-scheme $T$. The moduli problem $\underline{\mathscr{H}om}(\mathcal{E},\mathcal{F})$ is representable by algebraic spaces locally of finite type \cite[Proposition 2.3]{Lie}, which also implies that $\underline{\mathscr{H}om}(\mathcal{E},\mathcal{F})$ satisfies the Schlessinger's condition.

\subsubsection{\textbf{Deformation Theory}}
An \emph{infinitesimal extension} of an $\mathcal{O}_{S}$-algebra $A$ is a surjective map of $\mathcal{O}_{S}$-algebras $A' \twoheadrightarrow A$ such that the kernel $M=\text{ker}(A' \rightarrow A)$ is a finitely generated nilpotent ideal.

Let $\widetilde{F}: (\text{Sch}/S)^{\rm op} \rightarrow \text{Set}$ be a functor. Let $A_0$ be a noetherian $\mathcal{O}_{S}$-domain. A \emph{deformation situation} is defined as a triple
\begin{align*}
(A' \rightarrow A \rightarrow A_0, M, \xi)
\end{align*}
where $A' \rightarrow A \rightarrow A_0$ is a diagram of infinitesimal extension, $M=\text{ker}(A' \rightarrow A)$ a finite $A_0$-module and $\xi \in \widetilde{F}(A_0)$. We have a natural map $\widetilde{F}(A) \rightarrow \widetilde{F}(A_0)$. Let $\xi$ be an element in $\widetilde{F}(A_0)$. Denote by $\widetilde{F}_{\xi}(A)$ the set of elements in $\widetilde{F}(A)$ whose image is $\xi \in \widetilde{F}(A_0)$.

The \emph{deformation theory} we consider in this paper is described in \cite[Definition 5.2]{Art}. A \emph{deformation theory} for $\widetilde{F}$ consists of the following data and conditions
\begin{enumerate}
\item A functor associates to every triple $(A_0,M,\xi)$ an $A_0$-module $D=D(A_0,M,\xi)$, and to every map of triples $(A_0,M,\xi) \rightarrow (B_0,N,\eta)$ an linear map $D(A_0,M,\xi)\rightarrow D(B_0,N,\eta)$.
\item For every deformation situation, there is an operation of the additive group of $D(A_0,M,\xi)$ on $\widetilde{F}_{\xi}(A')$ such that two elements are in the same orbit under the operation if and only if they have the same image in $\widetilde{F}_{\xi}(A)$, where $\widetilde{F}_{\xi}(A')$ is the subset of $\widetilde{F}(A')$ of elements whose image in $\widetilde{F}(A_0)$ is $\xi$.
\end{enumerate}

\subsubsection{\textbf{Artin's Theorem}}
\begin{thm}[Theorem 5.3 in \cite{Art}]\label{502}
Let $\widetilde{F}$ be a functor on $({\rm Sch/S})^{\rm op}$. Given a deformation theory for $\widetilde{F}$, then $\widetilde{F}$ is represented by a separated and locally of finite type algebraic space over $S$, if the following conditions hold:
\begin{enumerate}
\item $\widetilde{F}$ is a sheaf with respect to the fppf-topology (or big \'etale topology).
\item $\widetilde{F}$ is locally of finite presentation.
\item $\widetilde{F}$ is integrable.
\item $\widetilde{F}$ satisfies the following conditions of separation.
\begin{enumerate}
\item[{\rm (a)}] Let $A_0$ be a \emph{geometric discrete valuation ring}, which is a localization of a finite type $\mathcal{O}_S$-algebra with residue field of finite type over $\mathcal{O}_S$. Let $K,k$ be its fraction field and residue field respectively. If $\xi,\eta \in \widetilde{F}(A_0)$ induce the same element in $\widetilde{F}(K)$ and $\widetilde{F}(k)$, then $\xi=\eta$.
\item[{\rm (b)}] Let $A_0$ be an $\mathcal{O}_S$-integral domain of finite type. Let $\xi,\eta \in \widetilde{F}(A_0)$. Suppose that there is a dense set $\mathcal{S}$ in $\text{Spec}(A_0)$ such that $\xi=\eta$ in $\widetilde{F}(k(s))$ for all $s \in \mathcal{S}$. Then $\xi=\eta$ on a non-empty open subset of $\text{Spec}(A_0)$.
\end{enumerate}
\item The deformation theory satisfies the following conditions.
\begin{enumerate}
\item[{\rm (a)}] The module $D=D(A_0,M,\xi)$ commutes with localization in $A_0$ and is a finite module when $M$ is free of rank one.
\item[{\rm (b)}] The module operates freely on $F_\xi(A')$ when $M$ is of length one.
\item[{\rm (c)}] Let $A_0$ be an $\mathcal{O}_S$-integral domain of finite type. There is a non-empty open set $U$ of $\text{Spec}(A_0)$ such that for every closed point $s \in U$, we have $$D \otimes_{A_0} k(s)=D(k,M\otimes_{A_0}k(s), \xi_s)$$.
\end{enumerate}

\item Suppose that we have a deformation situation $(A' \rightarrow A \rightarrow A_0,M,\xi)$.
\begin{enumerate}
\item[{\rm (a)}] Let $A_0$ be of finite type and $M$ of length one. Let
\begin{center}
\begin{tikzcd}
B' \arrow[r] \arrow[d]  & B \arrow[r] \arrow[d] & A_0 \arrow[d] \\
A' \arrow[r] & A \arrow[r] & A_0
\end{tikzcd}
\end{center}
be a diagram of infinitesimal extensions of $A_0$ with $B'=A' \times_A B$. If $b \in \widetilde{F}(B)$ is an element lying over $\xi$ whose image $a \in \widetilde{F}(A)$ can be lifted to $\widetilde{F}(A')$, then $b$ can be lifted to $\widetilde{F}(B')$. This condition is the homogeneity.

\item[{\rm (b)}] $A_0$ is a geometric discrete valuation ring with fraction field $K$ and $M$ free of rank one. Denote by $A_K, A'_K$ the localizations of $A,A'$ respectively. If the image of $\xi$ in $\widetilde{F}(A_K)$ can be lifted to $\widetilde{F}(A'_K)$, then its image in $\widetilde{F}(A_0 \times_K A_K)$ can be lifted to $\widetilde{F}(A_0 \times_K A'_K)$.

\item[{\rm (c)}] With the same notations as $6(b)$, let $M$ be a free module of rank $n$ and $\xi \in \widetilde{F}(A)$. Suppose that for every one-dimensional quotient $M^*_K$ of $M_K$ the lifting of $\xi_K$ to $\widetilde{F}(A^*_K)$ is obstructed, where $A'_K \rightarrow A^*_K \rightarrow A_K$ is the extension determined by $M^*_K$. Then there is a non-empty open set $U$ of ${\rm Spec}(A_0)$ such that for every quotient $\epsilon: M  \rightarrow M^*$ of length one with support in $U$,  the lifting of $\xi$ to $\widetilde{F}(A^*)$ is obstructed, where $A' \rightarrow A^* \rightarrow A$ denotes the resulting extension.
\end{enumerate}

\end{enumerate}
\end{thm}

We want to remind the reader that Olsson and Starr applied Theorem \ref{502} to prove that the quot-functor $\widetilde{{\rm Quot}}(\mathcal{G})$ is represented by an algebraic space (see Theorem \ref{301} or \cite[Theorem 1.1]{OlSt}). Therefore, the quot-functor $\widetilde{{\rm Quot}}(\mathcal{G})$ satisfies all of the properties in Theorem \ref{502}. To prove that the $\Lambda$-quot-functor is represented by an algebraic space, we have to check that the functor $\widetilde{{\rm Quot}}_{\Lambda}(\mathcal{G})$ satisfies all of the conditions in the Artin's theorem.

\subsection{Proof of Theorem \ref{501}}
We use the notation $\mathcal{M}:=\widetilde{{\rm Quot}}_{\Lambda}(\mathcal{G})$ for the $\Lambda$-quot-functor and $\mathcal{M}':=\widetilde{{\rm Quot}}(\mathcal{G})$ for the quot-functor. It is easy to check that $\mathcal{M}$ is a sheaf with respect to the fppf (or big \'etale topology), and we omit the proof here. We refer the reader to \cite{CasaWise} for details.

The quot-functor $\mathcal{M}'$ is represented by an algebraic space, and $\mathcal{M}'$ satisfies all conditions in Theorem \ref{502} (see Proof of \cite[Theorem 1.1]{OlSt}). The proofs of \emph{locally of finite presentation} (\S 5.3.1), \emph{integrability} (\S 5.3.2) and \emph{separation} (\S 5.3.3) depend on the corresponding properties of $\mathcal{M}'$. In \S 5.3.4, we construct the deformation theory of $\mathcal{M}$ and prove that the deformation theory satisfies the conditions listed in Theorem \ref{502}(4). The obstruction theory is discussed in \S 5.3.5.

\subsubsection{\textbf{Locally of Finite Presentation}}
Let $A:=\lim\limits_{\longrightarrow}A_n$ be the colimit of $\mathcal{O}_S$-algebras $A_n$. There are natural maps $\mathcal{M}(A_n) \rightarrow \mathcal{M}(A)$ for $n \geq 1$. These maps induce the following one
\begin{align*}
\lim\limits_{\longrightarrow} \mathcal{M}(A_n) \rightarrow \mathcal{M}(A).
\end{align*}
To prove that $\mathcal{M}$ is of locally of finite presentation, we have to show that the above map is bijective. Let $(\mathcal{F}_n,\Phi_n)_{n \geq 1}$ be an element in $\lim\limits_{\longrightarrow} \mathcal{M}(A_n)$, where $\mathcal{F}_n \in \mathcal{M}'(A_n)$ and $\Phi_n : \Lambda \otimes \mathcal{F}_n \rightarrow \mathcal{F}_n$. Recall that $\mathcal{M}'$ is locally of finite presentation. Thus there exists a unique $\mathcal{F} \in \mathcal{M}'(A)$ corresponding to $(\mathcal{F}_n)_{n \geq 1}$. By \cite[(8.2.5)]{Groth}, the map
\begin{align*}
\lim\limits_{\longrightarrow} {\rm Hom}(\Lambda\otimes \mathcal{F}_n, \mathcal{F}_n) \rightarrow {\rm Hom}(\Lambda\otimes \mathcal{F}, \mathcal{F})
\end{align*}
is also bijective. Therefore we can find a unique element $(\mathcal{F},\Phi)$ corresponding to the given element $(\mathcal{F}_n,\Phi_n)_{n \geq 1} \in \lim\limits_{\longrightarrow} \mathcal{M}(A_n)$. This finishes the proof.

\subsubsection{\textbf{Integrability}}
Let $\bar{A}$ be a complete noetherian local $\mathcal{O}_S$-algebra and denote by $\mathfrak{m}$ the maximal ideal of $\bar{A}$. After changing the base, we consider $\mathcal{X}$ as a stack over $\bar{A}$. Let $\hat{\mathcal{X}}$ be the stack $\hat{\mathcal{X}}:=\lim\limits_{\longleftarrow}\mathcal{X}(\bar{A}/\mathfrak{m}^{n+1})$. There is a natural morphism $j: \hat{\mathcal{X}}\rightarrow \mathcal{X}$. Let $\mathcal{F}$ be a coherent sheaf over $\mathcal{X}$. Denote by $\hat{\mathcal{F}}$ the sheaf $j^*\mathcal{F}$. Before we prove the integrability, we first review the following lemma.

\begin{lem}[Lemma 2.2 in \cite{OlSt}]\label{503}
Let $\mathcal{F}_1$ and $\mathcal{F}_2$ be coherent sheaves over $\mathcal{X}$ with proper support over $\bar{A}$. For every integer $n$, the map
\begin{align*}
{\rm Ext}^n_{\mathcal{O}_{\mathcal{X}}}(\mathcal{F}_1,\mathcal{F}_2) \rightarrow {\rm Ext}^n_{\mathcal{O}_{ \hat{\mathcal{X}} }}(\hat{\mathcal{F}}_1,\hat{\mathcal{F}}_2)
\end{align*}
is an isomorphism.
\end{lem}

Now we will prove that $\mathcal{M}$ is integrable, i.e., the map
\begin{align*}
\mathcal{M}(\bar{A})\rightarrow \lim\limits_{\longleftarrow}\mathcal{M}(\bar{A}/\mathfrak{m}^{n+1})
\end{align*}
is injective and has a dense image.

Let $(\mathcal{F},\Phi) \in \mathcal{M}(\bar{A})$, where $\Phi:\Lambda \otimes \mathcal{F} \rightarrow \mathcal{F}$. We will find a unique element $(\mathcal{F}_n,\Phi_n)_{n \geq 1}$ in $\lim\limits_{\longleftarrow}\mathcal{M}(\bar{A}/\mathfrak{m}^{n+1})$ corresponding to the given pair $(\mathcal{F},\Phi)$. Recall that $\mathcal{M}'$ satisfies all of the conditions in Theorem \ref{502}. Thus the morphism
\begin{align*}
\mathcal{M}'(\bar{A})\rightarrow \lim\limits_{\longleftarrow}\mathcal{M}'(\bar{A}/\mathfrak{m}^{n+1})
\end{align*}
is an injection. We take an element $\mathcal{F} \in \mathcal{M}'(\bar{A})$, which corresponds to a unique element $(\mathcal{F}_n) \in \lim\limits_{\longleftarrow}\mathcal{M}'(\bar{A}/\mathfrak{m}^{n+1})$. By Lemma \ref{503}, we have
\begin{align*}
{\rm Hom}({\rm Gr}_1(\Lambda) \otimes \mathcal{F},\mathcal{F}) \rightarrow \lim\limits_{\longleftarrow}{\rm Hom}({\rm Gr}_1(\Lambda)\otimes \mathcal{F}_n,\mathcal{F}_n)
\end{align*}
is bijective. This bijection induces that the map
\begin{align*}
{\rm Hom}(\Lambda \otimes \mathcal{F},\mathcal{F}) \rightarrow \lim\limits_{\longleftarrow}{\rm Hom}(\Lambda\otimes \mathcal{F}_n,\mathcal{F}_n)
\end{align*}
is also a bijection. As a result, $\Phi$ corresponds to a unique map $(\Phi_n)_{n \geq 1} \in \lim\limits_{\longleftarrow}{\rm Hom}(\Lambda\otimes \mathcal{F}_n,\mathcal{F}_n)$. Therefore the natural map
\begin{align*}
\mathcal{M}(\bar{A})\rightarrow \lim\limits_{\longleftarrow}\mathcal{M}(\bar{A}/\mathfrak{m}^{n+1})
\end{align*}
is injective.

Now we will prove that the map $\mathcal{M}(\bar{A})\rightarrow \lim\limits_{\longleftarrow}\mathcal{M}(\bar{A}/\mathfrak{m}^{n+1})$ has a dense image. The proof is similar to that of the injectivity. We take an element
\begin{align*}
\left( (\mathcal{F}_n,\Phi_n) \right)_{n \geq 1} \in \lim\limits_{\longleftarrow} \mathcal{M}(\bar{A}/\mathfrak{m}^{n+1}),
\end{align*}
where $\mathcal{F}_n \in \mathcal{M}'(\bar{A}/\mathfrak{m}^{n+1})$ and $\Phi_n: \Lambda \otimes \mathcal{F}_n \rightarrow \mathcal{F}_n$. Since $\mathcal{M}'$ is integrable, we can find an element $\mathcal{F} \in \mathcal{M}'(\bar{A})$ such that $\mathcal{F}$ induces $\mathcal{F}_1 \in \mathcal{M}'(\bar{A}/\mathfrak{m}^{2})$. Let $(\mathcal{F}'_n)_{n \geq 1} \in \lim\limits_{\longleftarrow} \mathcal{M}'(\bar{A}/\mathfrak{m}^{n+1})$ be the element corresponding to $\mathcal{F}$, where $\mathcal{F}'_1=\mathcal{F}_1$. By Lemma \ref{503}, we know that the map
\begin{align*}
{\rm Hom}(\Lambda \otimes \mathcal{F},\mathcal{F}) \rightarrow \lim\limits_{\longleftarrow}{\rm Hom}(\Lambda\otimes \mathcal{F}'_n,\mathcal{F}'_n)
\end{align*}
is bijective. Let $(\Phi'_n)_{n \geq 1}$ be an element in $\lim\limits_{\longleftarrow}{\rm Hom}(\Lambda\otimes \mathcal{F}'_n,\mathcal{F}'_n)$ such that $\Phi'_1=\Phi_1$. The element $(\Phi'_n)_{n \geq 1}$ corresponds to a unique map $\Phi: \Lambda \otimes \mathcal{F}' \rightarrow \mathcal{F}'$. In conclusion, given an element $\left( (\mathcal{F}_n,\Phi_n) \right)_{n \geq 1}$, we find a pair $(\mathcal{F},\Phi)$ which induces $(\mathcal{F}_1,\Phi_1)$.

\subsubsection{\textbf{Separation}}
The quot-functor $\mathcal{M}'$ satisfies the separation condition ${\rm (4)}$. If $\mathcal{F}_1, \mathcal{F}_2 \in \mathcal{M}'(A_0)$ induce the same element in $\mathcal{M}'(K)$ and $\mathcal{M}'(k)$, then $\mathcal{F}_1=\mathcal{F}_2$. Let $\xi=(\mathcal{F}_1,\Phi_1), \eta=(\mathcal{F}_2,\Phi_2)$ be two elements in $\mathcal{M}$. If they induce the same element in $\mathcal{M}(K)$ and $\mathcal{M}(k)$, we have $\mathcal{F}_1=\mathcal{F}_2$. This also implies that ${\rm Hom}(\Lambda,\mathcal{E}nd(\mathcal{F}_1)) = {\rm Hom}(\Lambda,\mathcal{E}nd(\mathcal{F}_2))$. Therefore, $\Phi_1=\Phi_2$, and $\xi=\eta$. This finishes the proof for condition {\rm (4a)}.

The proof of condition {\rm (4b)} is the same, and we omit the proof here.

\subsubsection{\textbf{Deformation Theory}}
In this section, we calculate the $A_0$-module $\mathcal{M}_{\xi}(A_0[M])$ and prove that this module is the deformation theory for $\mathcal{M}$.

Let us consider a special case first. Let $A'=A_0[M]:=A_0 \oplus M$. Let $\mathcal{X}_{A_0}:= \mathcal{X} \times_S \Spec(A_0)$ and $\mathcal{X}_{A_0[M]}:= \mathcal{X} \times_S \Spec(A_0[M])$. Denote by $\pi: \mathcal{X}_{A_0[M]} \rightarrow \mathcal{X}_{A_0}$ the natural map. Let $(\mathcal{F},\Phi)$ be an element in $\mathcal{M}(A_0)$, where $\Phi: \Lambda_{\text{Spec}(A_0)} \otimes \mathcal{F}  \rightarrow \mathcal{F}$. Equivalently, $\Phi$ can be considered as a morphism $\Phi: \Lambda_{\text{Spec}(A_0)} \rightarrow \mathcal{E}nd(\mathcal{F})$. For simplicity, we use $\Lambda$ for the sheaf $\Lambda_{\Spec (A_0)}$ in this section. With respect to the above notation, the morphism $\Phi$ is
\begin{align*}
\Phi: \Lambda \otimes \mathcal{F}  \rightarrow \mathcal{F} \text{ or } \Phi: \Lambda \rightarrow \mathcal{E}nd(\mathcal{F}).
\end{align*}
Define $\mathcal{F}'=\mathcal{F} \times_{\text{Spec}(A_0)} \text{Spec} (A')$. Abusing the notation, we write $\mathcal{F}'$ as $\mathcal{F}\oplus\mathcal{F}[M]$. For a section $s$ of $\mathcal{E}nd(\mathcal{F})[M]$, the corresponding automorphism of $\mathcal{F}'$ is denoted by $1+s$. Moreover, if $v+w$ is a section of $\mathcal{E}nd(\mathcal{F}')$, where $v$ is a section of $\mathcal{E}nd(\mathcal{F})$ and $w$ is a section of $\mathcal{E}nd(\mathcal{F})[M]$, we have
\begin{align*}
\rho(1+s)(v+w)=v+w+\rho(s)(v),
\end{align*}
where $\rho$ is the natural action of $\mathcal{E}nd(\mathcal{F})$ on itself.

The \emph{deformation complex} $C_M^{\bullet}(\mathcal{F},\Phi)$ is defined as follows
\begin{align*}
C_M^{\bullet}(\mathcal{F},\Phi): C_M^0(\mathcal{F}) = \mathcal{E}nd(\mathcal{F})[M] \xrightarrow{e(\Phi)} C_M^1(\mathcal{F}) = \mathcal{H}om(\Lambda,\mathcal{E}nd(\mathcal{F})[M]),
\end{align*}
and the map $e(\Phi)$ is given by
\begin{align*}
\left( e(\Phi)(s) \right)(\lambda)=-\rho(s)(\Phi(\lambda)),
\end{align*}
where $s \in \mathcal{E}nd(\mathcal{F})[M]$ and $\lambda \in \Lambda$ are sections. If there is no ambiguity, we omit the notations $M$, $\mathcal{F}$, $\Phi$ in the complex $C_M^{\bullet}(\mathcal{F},\Phi)$ and use the following notation
\begin{align*}
C^{\bullet}: C^0 = \mathcal{E}nd(\mathcal{F})[M] \xrightarrow{e(\Phi)} C^1 = \mathcal{H}om(\Lambda,\mathcal{E}nd(\mathcal{F})[M])
\end{align*}
for the deformation complex.

Now we are ready to calculate $\mathcal{M}_{\xi}(A_0[M])$. The following proposition is a generalization of Theorem 2.3 in \cite{BisRam}.
\begin{prop}\label{504}
Let $\xi =(\mathcal{F},\Phi)$ be a $\Lambda$-module in $\mathcal{M}(A_0)$. The set $\mathcal{M}_{\xi}(A_0[M])$ is isomorphic to the hypercohomology group $\mathbb{H}^1(C^{\bullet})$, where $C^{\bullet}$ is the complex
\begin{align*}
C^{\bullet}:C^0 =\mathcal{E}nd(\mathcal{F})[M] \xrightarrow{e(\Phi)} C^1 = \mathcal{H}om(\Lambda,\mathcal{E}nd(\mathcal{F})[M]),
\end{align*}
where the map $e(\Phi)$ is defined as above.
\end{prop}

\begin{proof}
Let $\mathcal{U}=\{U_i=\text{Spec}(A_i)\}_{i \in I}$ be an \'etale covering of $\mathcal{X}$ by affine schemes, where $I$ is the index set. The covering $\mathcal{U}$ of $\mathcal{X}$ also gives an \'etale covering $\{U_i \times_S \text{Spec}(A_0)\}$ of $\mathcal{X}_{A_0}$. Define $U_i[M]=U_i \times_S \text{Spec}(A_0[M])$. Set
\begin{align*}
\mathcal{E}nd(\mathcal{F})[M]|_{U_i[M]}=C^0_i, \quad \mathcal{H}om(\Lambda,\mathcal{E}nd(\mathcal{F})[M]) |_{U_i[M]}=C^1_{i},
\end{align*}
where $C^0_i$ and $C^1_i$ are $A_0$-modules. Similarly, modules $C^0_{ij}$ (resp. $C^1_{ij}$) are restrictions of $C^0$ (resp. $C^1$) to $U_{ij}[M]=U_i[M] \bigcap U_j[M]$. We consider the following $\hat{C}$ech resolution of $C^{\bullet}$:
\begin{center}
\begin{tikzcd}
& 0 \arrow[d] & 0 \arrow[d] & \\
0 \arrow[r] & C^0 \arrow[r,"e(\Phi)"] \arrow[d,"d^0_0"] & C^1 \arrow[r] \arrow[d,"d^1_0"] & 0 \\
0 \arrow[r] & \sum C_i^0 \arrow[r,"e(\Phi)"] \arrow[d,"d^0_1"] & \sum C_i^1 \arrow[r] \arrow[d,"d^1_1"] & 0 \\
0 \arrow[r] & \sum C_{ij}^0 \arrow[r,"e(\Phi)"] \arrow[d,"d^0_2"] & \sum C_{ij}^1 \arrow[r] \arrow[d,"d^1_2"] & 0 \\
& \vdots  & \vdots  &
\end{tikzcd}
\end{center}
We calculate the first hypercohomology $\mathbb{H}^1(C^{\bullet})$ from the above diagram. Let $Z$ be the set of pairs $(s_{ij}, t_i)$, where $s_{ij} \in C^0_{ij}$ and $t_i \in C^1_i$ satisfying the following conditions:
\begin{enumerate}
\item $s_{ij}+s_{jk}=s_{ik}$ as elements of $C^0_{ijk}$.
\item $t_i-t_j=e(\Phi)(s_{ij})$ as elements of $C^1_{ij}$.
\end{enumerate}
Let $B$ be the subset of $Z$ consisting of elements $(s_i-s_j,e(\Phi)(s_i))$, where $s_i \in C^0_i$. By the definition of the hypercohomology, we have
\begin{align*}
\mathbb{H}^1(C^{\bullet}) = Z/B.
\end{align*}
We will prove that for each element in $\mathbb{H}^1(C^{\bullet})$, it corresponds to a unique $\Lambda$-module on $\mathcal{X}_{A_0[M]}$, of which the restriction to $\mathcal{X}_{A_0}$ is $(\mathcal{F},\Phi)$. In other words, there is a bijective map between $\mathbb{H}^1(C^{\bullet})$ and $\mathcal{M}_{\xi}(A_0[M])$.

We first prove that there is a natural map $\mathbb{H}^1(C^{\bullet})$ to $\mathcal{M}_{\xi}(A_0[M])$. Given an element $(s_{ij},t_i) \in Z$, we want to construct a $\Lambda$-module $(\mathcal{F}',\Phi')$ on $\mathcal{X}_{A_0[M]}$ such that
\begin{align*}
\mathcal{F}'|_{\mathcal{X}_{A_0}} \cong \mathcal{F}, \quad \Phi' |_{\mathcal{X}_{A_0}} \cong \Phi.
\end{align*}
We first give the construction of $\mathcal{F}'$. For each $U_i[M]$, there is a natural projection $\pi: U_i[M] \rightarrow U_i \times_S \text{Spec}(A_0)$. Take the sheaf $\mathcal{F}'_i=\pi^*(\mathcal{F}|_{U_i \times_S \text{Spec}(A_0)})$. By the first condition of $Z$, we can identify the restrictions of $\mathcal{F}'_i$ and $\mathcal{F}'_j$ to $U_{ij}[M]$ by the isomorphism $1+s_{ij}$ of $\mathcal{F}'_{ij}$. Therefore we get a well-defined quasi-coherent sheaf $\mathcal{F}'$ on $\mathcal{X}_{A_0[M]}$ such that the restriction of $\mathcal{F}'$ to $\mathcal{X}_{A_0}$ is $\mathcal{F}$.

Now we want to construct a morphism $\Phi' : \Lambda \rightarrow {\rm End}(\mathcal{F}')$. Note that
\begin{align*}
\mathcal{E}nd(\mathcal{F}') \cong \mathcal{E}nd(\mathcal{F}) \oplus \mathcal{E}nd(\mathcal{F}[M]).
\end{align*}
We also know that the morphism $\Phi'$ satisfies
\begin{align*}
\Phi'|_{\mathcal{X}_{A_0}}=\Phi.
\end{align*}
Thus on each affine set $U_i[M]$, we define the following morphism
\begin{align*}
\Phi_i+t_i : \Lambda \rightarrow \mathcal{E}nd(\mathcal{F}'_i),
\end{align*}
where $\Phi_i$ is the restriction of $\Phi$ to the open set $U_i \times_S \text{Spec}(A_0)$. By the second condition of the pair $(s_{ij}, t_i)$, i.e. $t_i-t_j=e(\Phi)(s_{ij})$, we have \begin{align*}
e(\Phi_i + t_i)(1+s_{ij})=\Phi_j+t_j.
\end{align*}
Therefore $\{\Phi_i+t_i\}_{i \in I}$ can be glued together to give a global homomorphism $\Phi':\Lambda \rightarrow  \mathcal{E}nd(\mathcal{F}')$. Given an element $(s_{ij},t_i)$ in $Z$, we construct a $\Lambda$-module $(\mathcal{F}',\Phi')$ in $\mathcal{M}_{\xi}(A_0[M])$.

Let $(s_{ij},t_i)$ be an element in $B$. In other words, $s_{ij}=s_i-s_j$ and $t_i=e(\Phi)(s_i)$. The identification of $\mathcal{F}'_i \cong \mathcal{F}'_j$ on $U_{ij}[M]$ is given by the isomorphism $$1+s_{ij}=1+(s_i-s_j).$$ Consider the following diagram
\begin{center}
\begin{tikzcd}
\mathcal{F}'_{ij} \arrow[d,"1+s_{ij}"] \arrow[r,"1+s_{i}"] & \mathcal{F}'_{ij} \arrow[d, "\text{Id}"]  \\
\mathcal{F}'_{ij} \arrow[r,"1+s_{j}"] & \mathcal{F}'_{ij}
\end{tikzcd}
\end{center}
The commutativity of the above diagram implies that $E'$ is trivial. Similarly, we have $$e(\Phi_i +t_i)(1+s_i)=\Phi_i.$$
Therefore the associated Hitchin pair $(\mathcal{F}',\Phi')$ is isomorphic to $(\pi^*\mathcal{F},\pi^* \Phi)$. In other words, for any element in $B$, the corresponding $\Lambda$-structure is trivial.

The above construction gives a well-defined map from $\mathbb{H}^1(C^{\bullet})$ to $\mathcal{M}_{\xi}(A_0[M])$.

Now we will construct the inverse map from $\mathcal{M}_{\xi}(A_0[M])$ to $\mathbb{H}^1(C^{\bullet})$. Let $(\mathcal{F}',\Phi') \in \mathcal{M}_{\xi}(A_0[M])$ be a $\Lambda$-module over $\mathcal{X}_{A_0[M]}$ such that
\begin{align*}
(\mathcal{F}'|_{\mathcal{X}_{A_0}},\Phi'|_{\mathcal{X}_{A_0}})=\xi=(\mathcal{F},\Phi).
\end{align*}
We still use the covering $\{U_i[M]\}_{i \in I}$ of $\mathcal{X}_{A_0[M]}$ to work on this problem locally. Clearly, $\mathcal{F}'_i=\mathcal{F}'|_{U_i[M]}$ is the pull-back of $\mathcal{F}'|_{\mathcal{X}_{A_0}}$. The coherent sheaf $\mathcal{F}'$ can be obtained by gluing $\mathcal{F}'_i$ together. Thus the automorphism $1+s_{ij}$ of $\mathcal{F}'_{ij}$ over the intersection $U_{ij}[M]$ should satisfy the condition
\begin{align*}
s_{ij}+s_{jk}=s_{ik}
\end{align*}
on $U_{ijk}[M]$, where $s_{ij}$ is an element in $\mathcal{E}nd(\mathcal{F})[M]|_{U_{ij}[M]}$. Now we consider the morphism
\begin{align*}
\Phi': \Lambda \rightarrow \mathcal{E}nd(\mathcal{F}') \cong \mathcal{E}nd(\mathcal{F}) \oplus \mathcal{E}nd(\mathcal{F})[M],
\end{align*}
which is given by $\Phi_i+t_i$ on the local chart $U_i \times_S \text{Spec}(A_0)$, where \begin{align*}
\Phi_i=\Phi|_{U_i \times_S \text{Spec}(A_0)} : \Lambda \rightarrow \mathcal{E}nd(\mathcal{F})|_{U_i \times_S \text{Spec}(A_0)}
\end{align*}
and
\begin{align*}
t_i : \Lambda \rightarrow \mathcal{E}nd(\mathcal{F})[M]|_{U_i \times_S \text{Spec}(A_0)}.
\end{align*}
By the compatibility condition of $\Phi_i+t_i$ on $U_{ij}[M]$, we have
\begin{align*}
e(\Phi_i+t_i)(1+s_{ij})=\Phi_j+t_j.
\end{align*}
This gives us
\begin{align*}
e(\Phi)(s_{ij})=t_i-t_j.
\end{align*}
Therefore, $(s_{ij},t_i) \in Z$.

The above discussion gives us a well-defined map from $\mathcal{M}_{\xi}(A_0[M])$ to $\mathbb{H}^1(C^{\bullet})$. It is easy to check that these two maps are inverse to each other. We finish the proof of this proposition.
\end{proof}

\begin{cor}\label{505}
The deformation complex
\begin{align*}
C^{\bullet}: C^0 = \mathcal{E}nd(\mathcal{F})[M] \xrightarrow{e(\Phi)} C^1 = {\rm Hom}(\Lambda, \mathcal{E}nd(\mathcal{F})[M])
\end{align*}
has the following long exact sequence
\begin{align*}
0 & \rightarrow \mathbb{H}^0(C^{\bullet}) \rightarrow H^0(\mathcal{X}_{A_0},C^0) \rightarrow H^0(\mathcal{X}_{A_0},C^1) \\
 & \rightarrow \mathbb{H}^1(C^{\bullet}) \rightarrow H^1(\mathcal{X}_{A_0},C^0) \rightarrow H^1(\mathcal{X}_{A_0},C^1) \rightarrow \mathbb{H}^2(C^{\bullet}) \rightarrow \cdots \quad .
\end{align*}
\end{cor}

\begin{proof}
This long exact sequence follows directly from the definition of hypercohomology (see \cite{BisRam}).
\end{proof}

\begin{cor}\label{506}
Let $0 \rightarrow M_1 \rightarrow M_2 \rightarrow M_3 \rightarrow 0$ be a short exact sequence for finitely generated $A_0$-modules. We have a long exact sequence for hypercohomology
\begin{align*}
\dots \rightarrow \mathbb{H}^i(C_{M_1}^{\bullet}) \rightarrow \mathbb{H}^i(C_{M_2}^{\bullet}) \rightarrow \mathbb{H}^i(C_{M_3}^{\bullet}) \rightarrow \mathbb{H}^{i+1}(C_{M_1}^{\bullet}) \rightarrow \cdots \quad .
\end{align*}
\end{cor}

Now we fix a $\Lambda$-module $\xi=(\mathcal{F},\Phi) \in \mathcal{M}_{\xi}(A_0[M])$ and consider the corresponding deformation complex
\begin{align*}
C^{\bullet}: C^0 = \mathcal{E}nd(\mathcal{F})[M]  \xrightarrow{e(\Phi)} C^1 = \text{Hom}(\Lambda,\mathcal{E}nd(\mathcal{F})[M]).
\end{align*}
We will check that the deformation theory of $\mathcal{M}$ with respect to the triple $(A_0,M,\xi)$ is given by the $A_0$-module $\mathbb{H}^1(C^{\bullet})$.

\begin{enumerate}
\item[{\rm (5a)}] It is well-known that the $A_0$-module $H^i(\mathcal{X}_{A_0},\bullet)$ commutes with localization in $A_0$, where $\bullet$ is a coherent sheaf. Thus the $A_0$-module $\mathbb{H}^1(C^{\bullet})$ also commutes with localization in $A_0$ by applying the \emph{Five Lemma} to the long exact sequence in Corollary \ref{505}. Now let $M$ be a free $A_0$-module of rank one. This case is exactly the \emph{infinitesimal deformation} and we use the notation $A_0[\varepsilon]:=A_0[M]$. By the \emph{finiteness theorem} of cohomology over Deligne-Mumford stacks \cite[\S 11.6]{Ol}, the modules $H^i(\mathcal{X}_{A_0},C^j)$ are finitely generated for $0 \leq i,j \leq 1$. Thus, $\mathbb{H}^1(C^{\bullet})\cong \mathcal{M}_{\xi}(A_0[\varepsilon])$ is also a finitely generated module by the long exact sequence in Corollary \ref{505}.

\item[{\rm (5b)}] We assume that $A=A_0$ and $A'=A_0[M]$. We will define an action $D=\mathcal{M}_{\xi}(A_0[M])$ on itself and show that this action is free. By Proposition \ref{504}, we know that
\begin{align*}
\mathcal{M}_{\xi}(A_0[M]) \cong \mathbb{H}^1(C^{\bullet}) = Z/B,
\end{align*}
where $Z$ is the set of pairs $(s_{ij}, t_i)$ such that $s_{ij} \in C^0_{ij}$ and $t_i \in C^1_i$ satisfy the following conditions
\begin{enumerate}
\item $s_{ij}+s_{jk}=s_{ik}$ as elements of $C^0_{ijk}$,
\item $t_i-t_j=e(\Phi)(s_{ij})$ as elements of $C^1_{ij}$.
\end{enumerate}
There is a natural action of $Z$ on itself
\begin{align*}
(s'_{ij},t'_i)(s_{ij},t_i):=(s'_{ij}+s_{ij},t'_i+t_i),
\end{align*}
where $(s'_{ij},t'_i)$, $(s_{ij},t_i) \in Z$. This action can be naturally extended to a well-defined action of $Z/B$ on itself, which is also a free action. Therefore we define a free action $D=\mathcal{M}_{\xi}(A_0[M])$ on itself.

\item[{\rm (5c)}] The condition of ${\rm (5c)}$ is a local property. We may assume that $\mathcal{X}$ is an algebraic space and $S$ is an affine scheme $\Spec(A_0)$. Before prove the condition ${\rm (5c)}$, we first review the following lemma.

\begin{lem}[Lemma 6.8, 6.9 in \cite{Art}]\label{507}
Let $\mathcal{X}$ be an algebraic space of finite type over an affine scheme $S=\Spec(A_0)$, where $A_0$ is an integral domain. Let $\mathcal{F}$, $\mathcal{G}$ be two coherent sheaves on $\mathcal{X}$, and we fix a non-negative integer $q$. Then there is a non-empty open set $U$ of $S$ such that for each $s \in U$, the canonical map is an isomorphism
\begin{align*}
{\rm Ext}_{X}^{q}(\mathcal{F},\mathcal{G})_s \xrightarrow{\cong} {\rm Ext}_{X_s}^q(\mathcal{F}_s,\mathcal{G}_s),
\end{align*}
and
\begin{align*}
H^q(\mathcal{X},\mathcal{F})_s \xrightarrow{\cong} H^q(\mathcal{X}_s,\mathcal{F}_s).
\end{align*}
\end{lem}

By the above lemma, we can find a non-empty open set $U$ of $\Spec(A_0)$ such that
\begin{align*}
H^i(\mathcal{X},C^j)_s \cong H^i(\mathcal{X}_s,C_s^j), \quad i \geq 0, j=1,2,
\end{align*}
for $s \in U$. Thus we have
\begin{align*}
\mathbb{H}^1(C^{\bullet})_s \cong \mathbb{H}^1(C^{\bullet}_s)
\end{align*}
by applying \emph{Five Lemma} to the long exact sequence in Corollary \ref{505}, where $C^{\bullet}_s$ is the restriction of the complex $C^{\bullet}$ to the point $s$. We finish the proof of the condition ${\rm (5c)}$.
\end{enumerate}

\subsubsection{\textbf{Obstruction Theory}}
Fix a deformation situation $(A'\rightarrow A \rightarrow A_0,M,\xi)$, where $M$ is a free $A_0$-module of rank $n$. For any quotient $\epsilon:M \rightarrow M^*$, let $A' \rightarrow A^*$ be the quotient of $A'$ defined by $M^*$.
\begin{center}
\begin{tikzcd}
M  \arrow[r] \arrow[d, "\epsilon"] & A' \arrow[r] \arrow[d] & A \arrow[d] \\
M^* \arrow[r] & A^* \arrow[r] & A
\end{tikzcd}
\end{center}
We can define the deformation situation $(A^* \rightarrow A \rightarrow A_0, M^*,\xi)$ (see \S 5.2.4). For any element $(\mathcal{F}^*,\Phi^*) \in \mathcal{M}_{\xi}(A^*)$, we want to lift it to a well-defined element in $\mathcal{M}_{\xi}(A')$. The obstruction for this lifting property comes from the vanishing of the second hypercohomology group $\mathbb{H}^2(C_{ {\rm ker}\epsilon }^{\bullet})$. By Corollary \ref{506}, we have a long exact sequence for the hypercohomology groups
\begin{align*}
\dots \rightarrow \mathbb{H}^1(C_{ {\rm ker}\epsilon}^{\bullet}) \rightarrow \mathbb{H}^1(C_{M}^{\bullet}) \rightarrow \mathbb{H}^1(C_{M^*}^{\bullet}) \rightarrow \mathbb{H}^2(C_{{\rm ker}\epsilon}^{\bullet}) \rightarrow \dots \quad .
\end{align*}
Such a lifting exists if and only if the morphism $\mathbb{H}^1(C_{M}^{\bullet}) \rightarrow \mathbb{H}^1(C_{M^*}^{\bullet})$ is surjective. Thus the vanishing of the second hypercohomology $\mathbb{H}^2(C_{{\rm ker}\epsilon}^{\bullet})$ is necessary and sufficient for the existence of such a lifting.
\begin{enumerate}
\item[{\rm (6a)}] The condition ${\rm (6a)}$ is exactly the \emph{homogeneity} of the functor $\mathcal{M}$. Note that there is a natural forgetful functor
    \begin{align*}
    \mathcal{M} \rightarrow \mathcal{M}', \quad (\mathcal{F}, \Phi) \rightarrow \mathcal{F}.
    \end{align*}
    The quot-functor $\mathcal{M}'$ is homogeneous \cite{OlSt}. If the forgetful functor is relatively homogeneous, we can prove that the functor $\mathcal{M}$ is homogeneous \cite[Lemma 10.18]{CasaWise}. We know that the fiber of the forgetful functor at a sheaf $\mathcal{F}$ is $\underline{\mathscr{H}om}(\Lambda \otimes \mathcal{F},\mathcal{F})$, which is homogeneous as we discussed in \S 5.2.3. Therefore the forgetful functor is relatively homogeneous, and the moduli problem $\mathcal{M}$ is homogeneous.

\item[{\rm (6b)}] There is a natural map
    \begin{align*}
    i:\mathcal{X}_{A'_K} \rightarrow \mathcal{X}_{A_0 \times_K A'_K},
    \end{align*}
    which is induced by the natural inclusion $A_0 \times_K A'_K \rightarrow A'_K$. The induced functor $i_*$ is left exact on the category of quasi-coherent sheaves, and denote by $i^*$ the left adjoint of the functor $i_*$. We have the following isomorphisms
    \begin{align*}
        \text{Ext}^q_{A'_K}(i^* \mathcal{F}_1,\mathcal{F}_2) \cong \text{Ext}^q_{A_0 \times_K A'_K}(\mathcal{F}_1,i_* \mathcal{F}_2), \quad q \geq 0,
    \end{align*}
    for quasi-coherent sheaves $\mathcal{F}_1$ over $\mathcal{X}_{A_0 \times_K A'_K}$ and $\mathcal{F}_2$ over $\mathcal{X}_{A'_K}$. Given a coherent sheaf $\mathcal{F}$ over $\mathcal{X}$, we consider the deformation complex
    \begin{align*}
        C^{\bullet}: C^0 = \mathcal{E}nd(\mathcal{F})[M]  \xrightarrow{e(\Phi)} C^1 = \text{Hom}(\Lambda,\mathcal{E}nd(\mathcal{F})[M]).
    \end{align*}
    We have
    \begin{align*}
    & H^q(\mathcal{X}_{A'_K},C^0_{A'_K}) \cong H^q(\mathcal{X}_{A_0 \times_K A'_K},C^0_{A_0 \times_K A'_K}),\\
    & H^q(\mathcal{X}_{A'_K},C^1_{A'_K}) \cong H^q(\mathcal{X}_{A_0 \times_K A'_K},C^1_{A_0 \times_K A'_K}).
    \end{align*}
    Therefore,
    \begin{align*}
    \mathbb{H}^2(C^{\bullet}_{A'_K}) \cong \mathbb{H}^2(C^{\bullet}_{A_0 \times_K A'_K})
    \end{align*}
    by applying \emph{Five Lemma} to the long exact sequence in Corollary \ref{505}. It follows that the obstructions of lifting elements to $\mathcal{M}(A_0 \times_K A'_K)$ and lifting elements to $\mathcal{M}(A'_K)$ are the same.

\item[{\rm (6c)}] The proof of the condition {\rm (6c)} is similar to that of the condition {\rm (5c)}. The difference is that we worked on the deformation $\mathbb{H}^1(C^{\bullet})$ in the condition {\rm (5c)}, while the condition {\rm (6c)} focuses on the obstruction $\mathbb{H}^2(C^{\bullet})$. We use the same notation as in the statement of {\rm (5c)}. Let $\xi \in \mathcal{M}(A)$. Suppost that for every one-dimensional quotient $M_K \rightarrow M^*_K$, there is a non-trivial obstruction to lift $\xi_K \in \mathcal{M}(A_K)$ to $\mathcal{M}(A^*_K)$, where $A^*_K$ is the extension defined by $M_K^*$. We want to prove that there exists an open subset $U \subseteq \Spec(A_0)$ such that $\xi$ cannot be lifted to $\mathcal{M}(A^*)$.

    Let $N$ be the kernel of $M \rightarrow M^*$, and we consider the deformation complex $C^{\bullet}_N$
    \begin{align*}
        C^{\bullet}_N: C_N^0 = \mathcal{E}nd(\mathcal{F})[N]  \xrightarrow{e(\Phi)} C_N^1 = \text{Hom}(\Lambda,\mathcal{E}nd(\mathcal{F})[N]).
    \end{align*}
    By Lemma \ref{507}, we can choose an open set $U$ of $S$ such that
    \begin{align*}
        H^q(\mathcal{X},C_N^0)_s \cong H^q(\mathcal{X}_s,(C_N^0)_s), \quad q \geq 0,
    \end{align*}
    for $s \in U$. This implies that
    \begin{align*}
    \mathbb{H}^2(C_{N}^{\bullet})_s \cong \mathbb{H}^2((C_N^{\bullet})_s)
    \end{align*}
    for $s \in U$. Thus for every quotient $M \rightarrow M^*$ of length one with support in $U$, the lifting of $\xi$ to $F(A^*)$ is obstructed. This finishes the proof of the condition {\rm (6c)}.
\end{enumerate}

\subsection{$\Lambda$-Quot-Functors on Projective Deligne-Mumford Stacks}
Let $S$ be an affine scheme or a noetherian scheme of finite type. Suppose that $\mathcal{X}$ is a projective Deligne-Mumford stack over $S$. Let $\mathcal{E}$ be a generating sheaf of $\mathcal{X}$. We fix a polynomial $P$, which is considered as a modified Hilbert polynomial, and a coherent sheaf $\mathcal{G}$ over $\mathcal{X}$. We define the quot-functor
\begin{align*}
\widetilde{\rm Quot}_{\Lambda}(\mathcal{G},P) : (\text{Sch}/S)^{\rm op} \rightarrow \text{Set}
\end{align*}
for $\Lambda$-modules with respect to the given polynomial $P$ as follows. Let $T \in (\text{Sch}/S)$ be an $S$-scheme. The set $\widetilde{\rm Quot}_{\Lambda}(\mathcal{G},P)(T)$ contains the coherent sheaves $\mathcal{F}_T$ such that
\begin{enumerate}
\item $\mathcal{F}_T \in \widetilde{\rm Quot}(\mathcal{G},P)(T)$ (see \S 3.3),
\item $\mathcal{F}_T$ is a $\Lambda_T$-module.
\item The modified Hilbert polynomial of $\mathcal{F}_T$ is $P$.
\end{enumerate}
We will prove that the functor $\widetilde{\rm Quot}_{\Lambda}(\mathcal{G},P)$ is represented by a quasi-projective $S$-scheme in this subsection (see Theorem \ref{511}). We first show that the functor $F_{\mathcal{E}}$ induces a closed immersion ${\rm Quot}_{\Lambda}(\mathcal{G},\mathcal{X},P) \rightarrow {\rm Quot}_{\Lambda}(F_{\mathcal{E}}(\mathcal{G}),X,P)$ (see Lemma \ref{508} and \ref{510}). Next, we prove that the ${\rm Quot}_{\Lambda}(F_{\mathcal{E}}(\mathcal{G}),X,P)$ is a quasi-projective scheme (see Proposition \ref{5101}). The construction of this proposition will be used in \S 6 (Proposition \ref{603}). Finally, these results give the main theorem (Theorem \ref{511}).

Recall that the functor $F_{\mathcal{E}} : {\rm QCoh}(\mathcal{X}) \rightarrow {\rm QCoh}(X)$ is an exact functor (see \S 2.2).  This functor can be generalized to a natural transformation \begin{align*}
\widetilde{{\rm Quot}}_{\Lambda}(\mathcal{G},\mathcal{X},P) \rightarrow \widetilde{{\rm Quot}}_{F_{\mathcal{E}}(\Lambda)}(F_{\mathcal{E}}(\mathcal{G}),X, P),
\end{align*}
where $\widetilde{{\rm Quot}}_{\Lambda}(\mathcal{G},\mathcal{X},P)$ is the quot-functor of $\Lambda$-modules with modified Hilbert polynomial $P$ over $\mathcal{X}$ and $\widetilde{{\rm Quot}}_{F_{\mathcal{E}}(\Lambda)}(F_{\mathcal{E}}(\mathcal{G}),X, P)$ is the quot-functor of $F_{\mathcal{E}}(\Lambda)$-modules with Hilbert polynomial $P$ over $X$. Note that $F_{\mathcal{E}}(\Lambda)$ is still a sheaf of graded algebras, and the modified Hilbert polynomial is fixed under the functor $F_{\mathcal{E}}$. We still use the same notation $F_{\mathcal{E}}$ for this natural transformation.

\begin{lem}\label{508}
The natural transformation
\begin{align*}
F_{\mathcal{E}}:\widetilde{{\rm Quot}}_{\Lambda}(\mathcal{G},\mathcal{X},P) \rightarrow \widetilde{{\rm Quot}}_{F_{\mathcal{E}}(\Lambda)}(F_{\mathcal{E}}(\mathcal{G}),X, P)
\end{align*}
is a monomorphism.
\end{lem}

\begin{proof}
We have to show that for each $S$-scheme $T$, the morphism
\begin{align*}
F_{\mathcal{E}}(T): \widetilde{{\rm Quot}}_{\Lambda}(\mathcal{G},\mathcal{X},P)(T) \rightarrow \widetilde{{\rm Quot}}_{F_{\mathcal{E}}(\Lambda)}(F_{\mathcal{E}}(\mathcal{G}),X, P)(T)
\end{align*}
is an injection. We will omit $T$ for simplicity.

Let $(\mathcal{F},\Phi: \Lambda \otimes \mathcal{F}  \rightarrow \mathcal{F})$ be an element in $\widetilde{{\rm Quot}}_{\Lambda}(\mathcal{G},\mathcal{X},P)$. We have
\begin{align*}
(F_{\mathcal{E}}(\mathcal{F}),F_{\mathcal{E}}(\Phi): F_{\mathcal{E}}(\Lambda \otimes \mathcal{F}) \rightarrow F_{\mathcal{E}}(\mathcal{F})) \in \widetilde{{\rm Quot}}_{F_{\mathcal{E}}(\Lambda)}(F_{\mathcal{E}}(\mathcal{G}),X, P)
\end{align*}
under the transformation $F_{\mathcal{E}}$. The natural transformation $F_{\mathcal{E}}$ is a monomorphism when restricting to the quot-functor $\widetilde{{\rm Quot}}(\mathcal{G},\mathcal{X},P)$, i.e. the morphism
\begin{align*}
F_{\mathcal{E}}: \widetilde{{\rm Quot}}(\mathcal{G},\mathcal{X},P) \rightarrow \widetilde{{\rm Quot}}(F_{\mathcal{E}}(\mathcal{G}),X, P)
\end{align*}
is an injection \cite[Lemma 6.1]{OlSt}. Thus the coherent sheaf $\mathcal{F}$ corresponds to a unique coherent sheaf $F_{\mathcal{E}}(\mathcal{F})$. Now we will show that $F_{\mathcal{E}}$ is also an injection for the morphism $\Phi: \Lambda \otimes \mathcal{F}  \rightarrow \mathcal{F}$, and we only have to prove that if
\begin{align*}
F_{\mathcal{E}}(\Phi): F_{\mathcal{E}}(\Lambda \otimes F ) \rightarrow F_{\mathcal{E}}(F)
\end{align*}
is trivial, i.e. ${\rm ker}(F_{\mathcal{E}}(\Phi)) \cong F_{\mathcal{E}}(\Lambda \otimes F )$, then $\Phi$ is also a trivial morphism. Note that there is a short exact sequence
\begin{align*}
0 \rightarrow {\rm ker}(\Phi) \rightarrow \Lambda \otimes \mathcal{F}  \xrightarrow{\Phi} \mathcal{F}.
\end{align*}
Applying the exact functor $F_{\mathcal{E}}$ to the above sequence, we have
\begin{align*}
0 \rightarrow F_{\mathcal{E}}({\rm ker}(\Phi)) \cong {\rm ker}(F_{\mathcal{E}}(\Phi)) \rightarrow F_{\mathcal{E}}(\mathcal{F} \otimes \Lambda) \rightarrow F_{\mathcal{E}}(\mathcal{F}).
\end{align*}
Since $F_{\mathcal{E}}$ is an exact functor, ${\rm ker}(F_{\mathcal{E}}(\Phi)) \cong F_{\mathcal{E}}(\mathcal{F} \otimes \Lambda)$ if and only if ${\rm ker}(\Phi) \cong \mathcal{F} \otimes \Lambda$. Therefore the functor $F_{\mathcal{E}}$ is also an injection for the morphism.
\end{proof}

\begin{rem}\label{509}
Recall that $\pi: \mathcal{X} \rightarrow X$ is the natural map to the coarse moduli space $X$, and $\pi_*: {\rm QCoh}(\mathcal{X}) \rightarrow {\rm QCoh}(X)$ is an exact functor. The exact functor induces a natural transformation $\pi_*: \widetilde{{\rm Quot}}(\mathcal{G},\mathcal{X},P) \rightarrow \widetilde{{\rm Quot}}(\pi_*(\mathcal{G}),X, P)$. Note that this natural transformation is not injective in general. This is also one of the reasons why people introduces the generating sheaf $\mathcal{E}$ and define the functor $F_{\mathcal{E}}$.
\end{rem}

\begin{lem}\label{510}
The monomorphism $F_{\mathcal{E}}$ is relatively representable by schemes, and $F_{\mathcal{E}}$ is a $F_{\mathcal{E}}$ is a finitely-presented finite monomorphism.
\end{lem}

\begin{proof}
The proof of this lemma is the same as \cite[Proposition 6.2]{OlSt} by applying Lemma \ref{508}.
\end{proof}

\begin{prop}\label{5101}
Let $X$ be a projective scheme, and let $G$ be a coherent sheaf on $X$. We fix a polynomial $P$. The $\Lambda$-quot-functor $\widetilde{{\rm Quot}}_{\Lambda}(G,X, P)$ is represented by a quasi-projective scheme.
\end{prop}

\begin{proof}
In this proposition, we will construct the scheme representing $\widetilde{{\rm Quot}}_{\Lambda}(G,X, P)$.

We fix a positive integer $k$. Denote by $Q_1$ the quot-scheme ${\rm Quot}(\Lambda_k \otimes G,X,P)$. For each $S$-scheme $f:T \rightarrow S$, the set $Q_1(T)$ parameterizes the isomorphism classes of quotients
\begin{align*}
f^*(\Lambda_k \otimes G) \rightarrow \mathcal{F}_T \rightarrow 0
\end{align*}
with modified Hilbert polynomial $P$, where $\mathcal{F}_T$ is a coherent sheaf over $X_T$.

There exists an open subscheme $Q_2 \subseteq Q_1$ such that any quotient $[\rho_T] \in Q_2(T)$ has the following factorization.
\begin{center}
\begin{tikzcd}
f^*(\Lambda_k \otimes G) \arrow[rd, "1_{f^*(\Lambda_k)} \otimes \rho'_T"] \arrow[rr, "\rho_T"] &  & \mathcal{F}_T  \\
& f^*(\Lambda_k) \otimes \mathcal{F}_T \arrow[ru,"\Phi_k"]&
\end{tikzcd}
\end{center}
where the induced map $\rho'_T: f^*G \rightarrow \mathcal{F}_T$ is a quotient in ${\rm Quot}(G,P)(T)$.

If a quotient map is in $Q_2$, the coherent sheaf has a $\Lambda_k$-structure. Next, we will construct a scheme of $Q_3$ such that the coherent sheaf has a $\Lambda$-structure.

Let $[\rho_T: f^*(\Lambda_k \otimes G) \rightarrow \mathcal{F}_T]$ be a point in $Q_2(T)$. Denote by $\rho'_T: f^*(G) \rightarrow \mathcal{F}_T$ the quotient map in the factorization of $\rho_T$. Let $\mathcal{K}$ be the kernel of the quotient map
\begin{align*}
0 \rightarrow \mathcal{K} \rightarrow f^*(G) \xrightarrow{\rho'_T} \mathcal{F}_T \rightarrow 0.
\end{align*}
The quotient map $\rho_T$ induces the morphism $f^*(\Lambda_1 \otimes G) \rightarrow \mathcal{F}_T$, which gives us the following map
\begin{align*}
f^*(\Lambda_1) \otimes \mathcal{K} \rightarrow f^*(\Lambda_1 \otimes G) \rightarrow \mathcal{F}_T.
\end{align*}
There exists a closed subscheme $Q_3 \subseteq Q_2$ such that the induced map $f^*(\Lambda_1) \otimes \mathcal{K} \rightarrow \mathcal{F}_T$ is trivial.

Now let $[\rho_T]$ be a quotient in $Q_3$. By the discussion above, the quotient map $\rho_T$ induces the following one
\begin{align*}
f^*(\Lambda_1 \otimes G) \rightarrow \mathcal{F}_T.
\end{align*}
Therefore we have the following factorization
\begin{center}
\begin{tikzcd}
f^*(\Lambda_1 \otimes V \otimes G) \arrow[rd] \arrow[rr] &  & \mathcal{F}_T  \\
& f^*(\Lambda_1) \otimes \mathcal{F}_T \arrow[ru, "\Phi_1"]&
\end{tikzcd}
\end{center}
For each positive integer $j$, we have a morphism
\begin{align*}
f^*(\overbrace{\Lambda_1 \otimes \dots \otimes \Lambda_1}^{j}) \otimes \mathcal{F}_T \rightarrow \mathcal{F}_T,
\end{align*}
which is induced by the morphism $\Phi_1:f^*(\Lambda_1) \otimes \mathcal{F}_T \rightarrow \mathcal{F}_T$. Denote by $\mathcal{K}_j$ the kernel of the surjection
\begin{align*}
\overbrace{\Lambda_1 \otimes \dots \otimes \Lambda_1}^{j} \rightarrow \Lambda_j \rightarrow 0.
\end{align*}
This gives us a well-defined map
\begin{align*}
f^*(\mathcal{K}_j) \otimes \mathcal{F}_T \rightarrow \mathcal{F}_T.
\end{align*}
Given a positive integer $j$, there exists a closed subscheme $Q_{4,j} \subseteq Q_3$ such that $[\rho_T] \in Q_{4,j}(T)$ if the corresponding map $f^*(\mathcal{K}_j) \otimes \mathcal{F}_T \rightarrow \mathcal{F}_T$ is trivial. Denote by $Q_{4, \infty}$ the intersection of all of these closed subschemes $Q_{4,j}$, $j \geq 1$. The conditions for $Q_3$ and $Q_{4,\infty}$ guarantee that a coherent sheaf $\mathcal{F}$ with a $\Lambda_k$-structure is also a $\Lambda$-module.

Now a quotient $[\rho_T: f^*(\Lambda_k \otimes G) \rightarrow \mathcal{F}_T] \in Q_{4,\infty}(T)$ gives a $f^{*}(\Lambda_k)$-structure on $\mathcal{F}_T$. This structure induces a $f^*(\Lambda_1)$-structure on $\mathcal{F}_T$. We know that $\Lambda_1$ generates $\Lambda$. Thus, a $f^*(\Lambda_1)$-structure will give us a $f^*(\Lambda)$-structure on $\mathcal{F}_T$, which will induce a $f^*(\Lambda_k)$-structure. Note that this $f^*(\Lambda_k)$-structure may not be the same as the previous one. However, there is a closed subset $Q_5 \subseteq Q_{4,\infty}$ such that these two structures are the same.

The locally-closed subset $Q_5 \subseteq Q_1={\rm Quot}(\Lambda_k \otimes G,X,P)$ is the quasi-projective scheme representing $\widetilde{{\rm Quot}}_{\Lambda}(G,X,P)$.

\end{proof}

By Lemma \ref{510}, the functor
\begin{align*}
F_{\mathcal{E}}:\widetilde{{\rm Quot}}_{\Lambda}(\mathcal{G},\mathcal{X},P) \rightarrow \widetilde{{\rm Quot}}_{F_{\mathcal{E}}(\Lambda)}(F_{\mathcal{E}}(\mathcal{G}),X, P)
\end{align*}
is relatively representable by schemes, and the moduli problem $\widetilde{{\rm Quot}}_{F_{\mathcal{E}}(\Lambda)}(F_{\mathcal{E}}(\mathcal{G}),X, P)$ is representable by quasi-projective scheme by Proposition \ref{5101}. Therefore, we have the following theorem.

\begin{thm}\label{511}
Let $S$ be an affine scheme (or a noetherian scheme of finite type) and let $\mathcal{X}$ be a projective Deligne-Mumford stack over $S$. The $\Lambda$-quot-functor $\widetilde{\rm Quot}_{\Lambda}(\mathcal{G},P)$ is represented by a quasi-projective $S$-scheme.
\end{thm}

Denote by ${\rm Quot}_{\Lambda}(\mathcal{G},P)$ the space representing $\widetilde{{\rm Quot}}_{\Lambda}(\mathcal{G},P)$.

\begin{cor}\label{512}
Let $S$ be an affine scheme (or a noetherian scheme of finite type) and let $\mathcal{X}$ be a projective stack over $S$. Then the connected components of ${\rm Quot}_{\Lambda}(\mathcal{G})$ are quasi-projective $S$-schemes.
\end{cor}

\begin{proof}
The connected components of $\widetilde{\rm Quot}_{\Lambda}(\mathcal{G})$ are parameterized by integer polynomials (as modified Hilbert polynomials). This corollary follows from Theorem \ref{511} immediately.
\end{proof}

\subsection{Boundedness of $\Lambda$-modules \Romannum{1}}
In this subsection, we will prove that the family of $p$-semistable $\Lambda$-modules of pure dimension $d$ with a given modified Hilbert polynomial $P$ is bounded.

Let $f: \mathcal{X} \rightarrow T$ be a family of projective stacks with a family of moduli spaces $X \rightarrow T$ over an algebraically closed field $k$, where $T$ is a scheme. Let $\mathcal{E}$ be a generating sheaf of $\mathcal{X}$, and let $\mathcal{O}_{X}(1)$ be an $f$-ample line bundle. We fix an integer polynomial $P$ of degree $d$ and a rational number $\mu_0$.

\begin{cor}\label{513}
Let $\widetilde{\mathfrak{F}}$ be the family of purely $d$-dimensional $\Lambda$-modules with modified Hilbert polynomial $P$ on the fibers of $f:\mathcal{X} \rightarrow T$ such that the maximal slope $\hat{\mu}_{\rm max}(F_{\mathcal{E}}(\mathcal{F})) \leq \mu_0$. The family $\widetilde{\mathfrak{F}}$ is bounded.
\end{cor}

\begin{proof}
By Lemma \ref{510}, it is equivalent to consider the family $F_{\mathcal{E}}(\widetilde{\mathfrak{F}})$ over $X$. The family $F_{\mathcal{E}}(\widetilde{\mathfrak{F}})$ is bounded by \cite[\S 3]{Simp2}. Therefore the family $\widetilde{\mathfrak{F}}$ is bounded.
\end{proof}

\begin{cor}\label{514}
The family of $p$-semistable $\Lambda$-modules of pure dimension $d$ with a given modified Hilbert polynomial $P$ on $\mathcal{X}$ is bounded.
\end{cor}

\begin{proof}
By Corollary \ref{513}, we only have to show that the slope $\hat{\mu}_{\rm max}(F_{\mathcal{E}}(\mathcal{F}))$ is bounded. The property of boundedness of the slope is given by Lemma \ref{406}. This finishes the proof of this corollary.
\end{proof}

\section{Moduli Space of $\Lambda$-modules on Projective Deligne-Mumford Stacks}
In this section, we will construct the moduli space of $\Lambda$-modules on projective Deligne-Mumford stacks. The setup is the same as \S 3.8. Let $S$ be a noetherian scheme of finite type, or an affine scheme, and let $\mathcal{X}$ be a projective Deligne-Mumford stack over $S$. Let $\mathcal{E}$ be a generating sheaf over $\mathcal{X}$ and let $\mathcal{O}_X(1)$ be a polarization over $X$, the coarse moduli space of $\mathcal{X}$.

By Theorem \ref{511}, ${\rm Quot}_{\Lambda}(\mathcal{G},P)$ is a projective scheme. Based on this result, we study the moduli problem of $p$-semistable $\Lambda$-modules over $\mathcal{X}$
\begin{align*}
\widetilde{\mathcal{M}}_{\Lambda}^{ss}(\mathcal{E},\mathcal{O}_X(1),P): (\text{Sch}/S)^{{\rm op}} \rightarrow \text{Set}.
\end{align*}
Given an $S$-scheme $T$, $\widetilde{\mathcal{M}}_\Lambda^{ss}(\mathcal{E},\mathcal{O}_X(1),P)(T)$ is the set of $T$-flat families of $p$-semistable $\Lambda$-modules on $\mathcal{X} \otimes_S T$ of pure dimension $d$ with modified Hilbert polynomial $P$ with respect to the following equivalence relation $``\sim"$. Let $(\mathcal{F}_T,\Phi_T), (\mathcal{F}'_T,\Phi'_{T}) \in \widetilde{\mathcal{M}}_\Lambda^{ss}(\mathcal{E},\mathcal{O}_X(1),P)(T)$ be two elements. We say $(\mathcal{F}_T,\Phi_T) \sim (\mathcal{F}'_T,\Phi'_{T})$ if and only if $\mathcal{F}_T \cong \mathcal{F}'_T \otimes p^* L$ and $\Phi_T \cong \Phi'_T \otimes 1_{p^* L}$ for some $L \in {\rm Pic}(T)$.

\begin{lem}\label{601}
Given a polynomial $P$, there is a positive integer $N_0$ depending on $\Lambda$, $\mathcal{E}$ and $P$ such that for any $m \geq N_0$ and any $p$-semistable $\Lambda$-module $\mathcal{F}$ with Hilbert polynomial $P$ on $\mathcal{X}$, we have
\begin{enumerate}
\item $H^0(\mathcal{X}/S,\mathcal{F}\otimes \mathcal{E}^{\vee} \otimes \pi^* \mathcal{O}_{X}(m))$ is locally free of rank $P(m)$ and $H^i(\mathcal{X}/S,\mathcal{F}\otimes \mathcal{E}^{\vee} \otimes \pi^* \mathcal{O}_{X}(m))=0$ for $i >0$.
\item The map
    \begin{align*}
    H^0(\mathcal{X}/S,\mathcal{F}\otimes \mathcal{E}^{\vee} \otimes \pi^* \mathcal{O}_{X}(m)) \otimes \mathcal{E} \otimes \pi^* \mathcal{O}_{X}(-m) \rightarrow \mathcal{F} \rightarrow 0
    \end{align*}
    is surjective.
\end{enumerate}
\end{lem}

\begin{proof}
By Corollary \ref{514}, we know that the family of $p$-semistable $\Lambda$-modules of pure dimension $d$ with a given modified Hilbert polynomial is bounded. Note that a $\Lambda$-module is also an $\mathcal{O}_{\mathcal{X}}$-module. Thus there is an integer $N_0$ such that when $m \geq N_0$, for any element $\mathcal{F}$ in this family, the $\Lambda$-module $\mathcal{F}$ is $m$-regular. By the equivalent conditions of boundedness and regularity in \S 3.5, the integer $m$ satisfies the requirements in the lemma.
\end{proof}

The $p$-semistability of coherent sheaves and $\Lambda$-modules on a projective scheme are open conditions \cite{HuLe,Simp2}. This statement can be directly generalized to $\Lambda$-modules over $\mathcal{X}$.

\begin{lem}\label{602}
Given an $S$-scheme $T$, let $\mathcal{X}_T \rightarrow T$ be a family of projective stacks over $T$ and let $\mathcal{F}_T$ be a family of $\Lambda$-modules on $\mathcal{X}_T$. There is an open subset $T^{ss} \subseteq T$ such that $\mathcal{F}_t$ is $p$-semistable if and only if $t \in T^{ss}$. The same argument holds for $p$-stable $\Lambda$-modules.
\end{lem}

Now we will construct a quasi-projective scheme parametrizing elements in $\widetilde{\mathfrak{F}}^{ss}_{\Lambda}(P)$, which is the set of $p$-semistable $\Lambda$-modules of pure dimension $d$ with the modified Hilbert polynomial $P$. We first give the idea of the construction and then prove the statements in detail (see Proposition \ref{603}).

By Lemma \ref{601}, we can take an integer $m$ such that for any $\Lambda$-module $(\mathcal{F},\Phi) \in \widetilde{\mathfrak{F}}^{ss}_{\Lambda}(P)$, the coherent sheaf $\mathcal{F}$ is $m$-regular. Moreover, by Proposition \ref{407}, we can choose an integer $N$ such that for any $(\mathcal{F},\Phi) \in \widetilde{\mathfrak{F}}^{ss}_{\Lambda}(P)$, we have
\begin{align*}
P(N) \geq P_{\mathcal{E}}(\mathcal{F},m)=h^0(X/S, F_{\mathcal{E}}(\mathcal{F})(m)).
\end{align*}
Let $V$ be the linear space $S^{P(N)}$. Let $\mathcal{G}$ be the coherent sheaf $\mathcal{E} \otimes \pi^* \mathcal{O}_X(-N)$. The above discussion tells us that each $\Lambda$-module $(\mathcal{F},\Phi) \in \widetilde{\mathfrak{F}}_{\Lambda}^{ss}(P)$ corresponds to a surjection $[V \otimes \mathcal{G} \rightarrow \mathcal{F}]$ together with an isomorphism $V \cong H^0(X,F_{\mathcal{E}}(\mathcal{F})(N))$. Note that this correspondence does not take the $\Lambda$-structure $\Phi$ into account. Therefore, the quot-scheme ${\rm Quot}(V \otimes \mathcal{G},P)$ is so small that it cannot cover all $\Lambda$-modules. We have to find a larger quot-scheme which can cover all $\Lambda$-modules of pure dimension $d$ with Hilbert polynomial $P$.

Let $k$ be a positive integer. We consider the quot-scheme ${\rm Quot}(\Lambda_k \otimes V \otimes \mathcal{G},P)$. Given an element $[\rho: \lambda_k \otimes V \otimes \mathcal{G} \rightarrow \mathcal{F}] \in Q'$, suppose that the quotient map $q$ have the following factorization
\begin{center}
\begin{tikzcd}
\Lambda_k \otimes V \otimes \mathcal{G} \arrow[rd, "1 \otimes \rho'"] \arrow[rr, "\rho"] &  & \mathcal{F}  \\
& \Lambda_k \otimes \mathcal{F} \arrow[ru,"\Phi_k"]&
\end{tikzcd}
\end{center}
such that
\begin{itemize}
\item the induced morphism $\rho':V \otimes \mathcal{G} \rightarrow \mathcal{F}$ is an element in ${\rm Quot}(V \otimes \mathcal{G},P)$;
\item $\Phi_k: \Lambda_k \otimes \mathcal{F} \rightarrow \mathcal{F}$ is a morphism.
\end{itemize}
If a quotient $[\rho:\Lambda_k \otimes V \otimes \mathcal{G} \rightarrow \mathcal{F}]$ has this factorization property, we say that $[\rho]$ \emph{admits a factorization}.

The map $\Phi_k$ in the factorization will give a $\Lambda$-structure on $\mathcal{F}$ under some good conditions. Given a $\Lambda$-module $(\mathcal{F},\Phi)$, the coherent sheaf $\mathcal{F}$ is included in ${\rm Quot}(V \otimes \mathcal{G},P)$ by Lemma \ref{601} and the morphism $\Phi:\Lambda \otimes \mathcal{F} \rightarrow \mathcal{F}$ induces a map $\Phi_k: \Lambda_k \otimes \mathcal{F} \rightarrow \mathcal{F}$ naturally. Therefore a $\Lambda$-module corresponds to an element in ${\rm Quot}(\Lambda_k \otimes V \otimes \mathcal{G},P)$ uniquely. On the other hand, for each element $[\rho: \Lambda_k \otimes V \otimes \mathcal{G} \rightarrow \mathcal{F}] \in {\rm Quot}(\Lambda_k \otimes V \otimes \mathcal{G},P)$, we have a natural map
\begin{align*}
V \otimes \mathcal{G} \rightarrow \Lambda_k \otimes V \otimes \mathcal{G} \rightarrow \mathcal{F}.
\end{align*}
This induces a natural morphism $\alpha: V \rightarrow H^0(X/S, F_{\mathcal{E}}(\mathcal{F})(N))$.

In summary, let $N$ be a large enough integer. We want to find a subset $Q^{ss}_{\Lambda}$ of ${\rm Quot}(\Lambda_k \otimes V \otimes \mathcal{G},P)$ such that $Q^{ss}_{\Lambda}$ contains all elements $[\rho:\Lambda_k \otimes V \otimes \mathcal{G} \rightarrow \mathcal{F}]$ satisfying the following conditions
\begin{itemize}
    \item a quotient $[\rho]$ admits a factorization and induces a unique $\Lambda$-structure on $\mathcal{F}$;
    \item $\mathcal{F}$ is a $p$-semistable $\Lambda$-module;
    \item $V \cong H^0(X/S, F_{\mathcal{E}}(\mathcal{F})(N))$.
\end{itemize}

Here is the formal setup of this problem. We fix a polynomial $P$. Let $N_0$ be the positive integer determined by Lemma \ref{601}. We choose integers $m,N$ as discussed above, and we consider the following moduli problem
\begin{align*}
\widetilde{Q}^{ss}_\Lambda: ({\rm Sch}/S)^{\rm op} \rightarrow {\rm Set},
\end{align*}
and for each $S$-scheme $T$, $\widetilde{Q}^{ss}_\Lambda(T)$ is the set of pairs $(\mathcal{F}_T,\alpha_T)$ such that
\begin{enumerate}
    \item $\mathcal{F}_T$ is a $p$-semistable $\Lambda$-module with the modified Hilbert polynomial $P$ on $\mathcal{X}_T$,
    \item $\alpha_T: V_T \cong H^0(X_T/T, F_{\mathcal{E}}(\mathcal{F}_T)(N))$ is an isomorphism.
\end{enumerate}

The proof of the following proposition is similar to that of Proposition \ref{5101}.
\begin{prop}\label{603}
The functor $\widetilde{Q}^{ss}_\Lambda$ is representable by a quasi-projective scheme $Q^{ss}_\Lambda$ over $S$.
\end{prop}

\begin{proof}
Let $V$ be the linear space $S^{P(N)}$ and let $\mathcal{G}$ be the coherent sheaf $\mathcal{E} \otimes \pi^* \mathcal{O}_X(-N)$. We fix a positive integer $k$. Denote by $Q_1$ the quot-scheme ${\rm Quot}(\Lambda_k \otimes V \otimes \mathcal{G},P)$. For each $S$-scheme $f:T \rightarrow S$, the set $Q_1(T)$ parameterizes the isomorphism classes of quotients
\begin{align*}
f^*(\Lambda_k \otimes V \otimes \mathcal{G}) \rightarrow \mathcal{F}_T \rightarrow 0
\end{align*}
with modified Hilbert polynomial $P$, where $\mathcal{F}_T$ is a coherent sheaf over $\mathcal{X}_T$.

There exists an open subscheme $Q_2 \subseteq Q_1$ such that any quotient map $\rho_T \in Q_2(T)$ admits a factorization. More precisely, let $[\rho_T: f^*(\Lambda_k \otimes V \otimes \mathcal{G}) \rightarrow \mathcal{F}_T]$ be a quotient in $Q_2(T)$. The map $\rho_T$ can be factored in the following way
\begin{center}
\begin{tikzcd}
f^*(\Lambda_k \otimes V \otimes \mathcal{G}) \arrow[rd, "1_{f^*(\Lambda_k)} \otimes \rho'_T"] \arrow[rr, "\rho_T"] &  & \mathcal{F}_T  \\
& f^*(\Lambda_k) \otimes \mathcal{F}_T \arrow[ru,"\Phi_k"]&
\end{tikzcd}
\end{center}
where the induced map $\rho'_T: f^*(V \otimes \mathcal{G}) \rightarrow \mathcal{F}_T$ is a quotient in ${\rm Quot}(V \otimes \mathcal{G},P)(T)$. As we discussed early in this section, a quotient map $\rho_T$ admitting a factorization gives us a quotient map $\rho'_T :f^*(V \otimes \mathcal{G}) \rightarrow \mathcal{F}_T$ and a $f^*(\Lambda_k)$-structure on the coherent sheaf $\mathcal{F}_T$.

If a quotient map is in $Q_2$, then the coherent sheaf will have a $\Lambda_k$-structure. However, this may not give a $\Lambda$-structure for the coherent sheaf. We will explore the conditions, under which a coherent sheaf with a $\Lambda_k$-structure is a $\Lambda$-module.

Let $[\rho_T: f^*(\Lambda_k \otimes V \otimes \mathcal{G}) \rightarrow \mathcal{F}_T]$ be a point in $Q_2(T)$. Denote by $\rho'_T: f^*(V \otimes \mathcal{G}) \rightarrow \mathcal{F}_T$ the quotient map in the factorization of $\rho_T$. Let $\mathcal{K}$ be the kernel of the quotient map
\begin{align*}
0 \rightarrow \mathcal{K} \rightarrow f^*(V \otimes \mathcal{G}) \xrightarrow{\rho'_T} \mathcal{F}_T \rightarrow 0.
\end{align*}
The quotient map $\rho_T$ induces the morphism $f^*(\Lambda_1 \otimes V \otimes \mathcal{G}) \rightarrow \mathcal{F}_T$, which gives us the following map
\begin{align*}
f^*(\Lambda_1) \otimes \mathcal{K} \rightarrow f^*(\Lambda_1 \otimes V \otimes \mathcal{G}) \rightarrow \mathcal{F}_T.
\end{align*}
There exists a closed subscheme $Q_3 \subseteq Q_2$ such that the induced map $f^*(\Lambda_1) \otimes \mathcal{K} \rightarrow \mathcal{F}_T$ is trivial.

Now let $[\rho_T]$ be a quotient in $Q_3$. By the discussion above, the quotient map $\rho_T$ induces the following one
\begin{align*}
f^*(\Lambda_1 \otimes V \otimes \mathcal{G}) \rightarrow \mathcal{F}_T.
\end{align*}
Therefore we have the following factorization
\begin{center}
\begin{tikzcd}
f^*(\Lambda_1 \otimes V \otimes \mathcal{G}) \arrow[rd] \arrow[rr] &  & \mathcal{F}_T  \\
& f^*(\Lambda_1) \otimes \mathcal{F}_T \arrow[ru, "\Phi_1"]&
\end{tikzcd}
\end{center}
For each positive integer $j$, we have a morphism
\begin{align*}
f^*(\overbrace{\Lambda_1 \otimes \dots \otimes \Lambda_1}^{j}) \otimes \mathcal{F}_T \rightarrow \mathcal{F}_T,
\end{align*}
which is induced by the morphism $\Phi_1:f^*(\Lambda_1) \otimes \mathcal{F}_T \rightarrow \mathcal{F}_T$. Denote by $\mathcal{K}_j$ the kernel of the surjection
\begin{align*}
\overbrace{\Lambda_1 \otimes \dots \otimes \Lambda_1}^{j} \rightarrow \Lambda_j \rightarrow 0.
\end{align*}
This gives us a well-defined map
\begin{align*}
f^*(\mathcal{K}_j) \otimes \mathcal{F}_T \rightarrow \mathcal{F}_T.
\end{align*}
Therefore given a positive integer $j$, there exists a closed subscheme $Q_{4,j} \subseteq Q_3$ such that $[\rho_T] \in Q_{4,j}(T)$ if the corresponding map $f^*(\mathcal{K}_j) \otimes \mathcal{F}_T \rightarrow \mathcal{F}_T$ is trivial. Denote by $Q_{4, \infty}$ the intersection of all of these closed subschemes $Q_{4,j}$, $j \geq 1$. The conditions for $Q_3$ and $Q_{4,\infty}$ guarantee that a coherent sheaf $\mathcal{F}$ with a $\Lambda_k$ structure is also a $\Lambda$-module.

The above discussion tells us that a quotient $[\rho_T: f^*(\Lambda_k \otimes V \otimes \mathcal{G}) \rightarrow \mathcal{F}_T] \in Q_{4,\infty}(T)$ gives a $f^{*}(\Lambda_k)$-structure on $\mathcal{F}_T$. This structure induces a $f^*(\Lambda_1)$-structure on $\mathcal{F}_T$. We know that $\Lambda_1$ generates $\Lambda$. Thus a $f^*(\Lambda_1)$-structure will give us a $f^*(\Lambda)$-structure on $\mathcal{F}_T$, which will induce a $f^*(\Lambda_k)$-structure. Note that this $f^*(\Lambda_k)$-structure may not be the same as the previous one. However, there is a closed subset $Q_5 \subseteq Q_{4,\infty}$ such that these two structures are the same.

After that, let $Q_6 \subseteq Q_5$ be the open subset such that if $\mathcal{F} \in Q_6$, then we have $V \cong H^0(X, F_{\mathcal{E}}(\mathcal{F})(N))$.

Finally, by Lemma \ref{602}, there is an open subset $Q^{ss}_\Lambda \subseteq Q_6$ such that $\mathcal{F}$ is a $p$-semistable $\Lambda$-module if and only if $\mathcal{F} \in Q^{ss}_\Lambda$.
\end{proof}

Now we come to the part of GIT. With the same reason as in \S 3.8, we consider the ${\rm SL}(V)$-action. There is a natural embedding
\begin{align*}
    \psi_N: {\rm Quot}(\Lambda_k \otimes V\otimes \mathcal{G},P)\hookrightarrow {\rm Grass}(H^0(X, F_{\mathcal{E}}(\Lambda_k \otimes V\otimes \mathcal{G})(N)),P(N) ),
\end{align*}
Let $\mathscr{L}_N$ be the pullback of the canonical invertible bundle over the Grassmannian, and $\mathscr{L}_N$ is an ample line bundle on ${\rm Quot}(\Lambda_k \otimes V \otimes \mathcal{G},P)$. There is a natural group action ${\rm SL}(V)$ on ${\rm Quot}(\Lambda_k \otimes V\otimes \mathcal{G},P)$, which induces an action on the line bundle $\mathscr{L}_N$. Given a group action ${\rm SL}(V)$ on ${\rm Quot}(\Lambda_k \otimes V \otimes \mathcal{G},P)$ and an ample line bundle $\mathscr{L}_N$ over ${\rm Quot}(\Lambda_k \otimes V \otimes \mathcal{G},P)$, semistable (resp. stable) points of ${\rm Quot}(\Lambda_k \otimes V\otimes \mathcal{G},P)$ are well-defined.

Denote by ${\rm Quot}^{ss}(\Lambda_k \otimes V\otimes \mathcal{G},P)$ the set of semistable points in ${\rm Quot}(\Lambda_k \otimes V\otimes \mathcal{G},P)$ with respect to the group action ${\rm SL}(V)$ and the line bundle $\mathscr{L}_N$. Next, we will show that $Q^{ss}_{\Lambda} \subseteq {\rm Quot}(\Lambda_k \otimes V\otimes \mathcal{G},P)$ (Lemma \ref{605}) and the closure of any ${\rm SL}(V)$-orbit of any element $[\rho] \in Q^{ss}_{\Lambda}$ is also included in $Q^{ss}_{\Lambda}$ (Lemma \ref{606}). Based on these results, we can prove that given two elements $[\rho_i: \Lambda_k \otimes V \otimes \mathcal{G} \rightarrow \mathcal{F}_i] \in Q^{ss}_{\Lambda}$, $i=1,2$, the closures of the orbits of these two elements intersect if $gr^{\rm JH}(\mathcal{F}_1) \cong gr^{\rm JH}(\mathcal{F}_2)$ (Lemma \ref{6071}).

\begin{lem}\label{605}
There is a large enough integer $N$ such that the subscheme $Q^{ss}_\Lambda \subseteq {\rm Quot}(\Lambda_k \otimes V\otimes \mathcal{G},P)$ is included in ${\rm Quot}^{ss}(\Lambda_k \otimes V\otimes \mathcal{G},P)$.
\end{lem}

\begin{proof}
Let $[\rho: \Lambda_k \otimes V \otimes \mathcal{G} \rightarrow \mathcal{F}]$ be a point in ${\rm Quot}(\Lambda_k \otimes V\otimes \mathcal{G},P)$. We will prove that if the $\Lambda$-module $\mathcal{F}$ (induced by $[\rho]$) is a $p$-semistable $\Lambda$-module, then the point $[\rho]$ is also semistable with respect the invertible sheaf $\mathscr{L}_N$ and the action of ${\rm SL}(V)$.

We have a natural monomophism
\begin{align*}
& {\rm Quot}(\Lambda_k \otimes V \otimes \mathcal{G},P) \rightarrow {\rm Quot}(F_{\mathcal{E}}( \Lambda_k \otimes V \otimes \mathcal{G} ) ,P) \\
& [\Lambda_k \otimes V \otimes \mathcal{G} \longrightarrow \mathcal{F}] \rightarrow [F_{\mathcal{E}}( \Lambda_k \otimes V \otimes \mathcal{G} ) \longrightarrow F_{\mathcal{E}}(\mathcal{F})],
\end{align*}
which is induced by the exact functor $F_{\mathcal{E}}$ (see \cite[Lemma 6.1]{OlSt} or \S 3.1). Let $V'$ be a subspace of $V$. Denote by $\mathcal{F}'$ the subsheaf of $\mathcal{F}$ generated by $V'$. Let $\mathcal{F}'_{\rm sat}$ be the saturation of $\mathcal{F}'$. Note that $\mathcal{F}'_{\rm sat}$ is a $\Lambda$-module and its $\Lambda$-structure coincide with the one induced from $\Lambda_k \otimes \mathcal{F}'_{\rm sat} \rightarrow \mathcal{F}'_{\rm sat}$ by the property of $Q_5$. Therefore $\mathcal{F}'_{\rm sat}$ is a $\Lambda$-submodule of $\mathcal{F}$.

In summary, we have $V' \subseteq H^0(X, F_{\mathcal{E}}(\mathcal{F}'_{\rm sat})(N) ) \subseteq H^0(X, F_{\mathcal{E}}(\mathcal{F})(N) )$. This gives us the following inequalities
\begin{align*}
\frac{\dim V'}{r(\mathcal{F}')} \leq \frac{h^0(F_{\mathcal{E}}(\mathcal{F}'_{\rm sat})(N))}{r(\mathcal{F}'_{\rm sat})} \leq \frac{P(N)}{r(\mathcal{F})}.
\end{align*}
Note that $r(\mathcal{F}')=r(\mathcal{F}'_{\rm sat})$.

Now we are working on the coherent sheaves $F_{\mathcal{E}}(\mathcal{F})$ and $F_{\mathcal{E}}(\mathcal{F}'_{\rm sat})$ on $X$.
\begin{itemize}
\item If $\frac{\dim V'}{r(\mathcal{F}')} < \frac{P(N)}{r(\mathcal{F})}$, we can choose $M$ large enough such that $P_{\mathcal{E}}(\bullet ,M)$ approximates $r(\bullet) M^d$, where $\bullet=\mathcal{F}, \mathcal{F}'$. Then, we have
    \begin{align*}
    \frac{\dim V'}{P _{\mathcal{E}}(\mathcal{F}',M) } < \frac{P(N)}{P_{\mathcal{E}}(\mathcal{F},M)}.
    \end{align*}
\item If $\frac{\dim V'}{r(\mathcal{F}')} = \frac{P(N)}{r(\mathcal{F})}$, then we know that the $\Lambda$-submodule $\mathcal{F}'_{\rm sat}$ has the same reduced modified Hilbert polynomial as $\mathcal{F}$. Therefore, $\mathcal{F}'_{\rm sat}$ is a $p$-semistable $\Lambda$-module. Also, the equality $\frac{\dim V'}{r(\mathcal{F}')} = \frac{P(N)}{r(\mathcal{F})}$ implies that
\begin{align*}
h^0(F_{\mathcal{E}}(\mathcal{F}')(N)  )=h^0(F_{\mathcal{E}}(\mathcal{F}_{\rm sat}')(N)  ).
\end{align*}
We have $\mathcal{F}=\mathcal{F}'_{\rm sat}$. In this case, $\mathcal{F}$ and $\mathcal{F}'$ has the same reduced modified Hilbert polynomial.
\end{itemize}
Combining these two cases, the point $[\rho]$ is a semistable point by Theorem \ref{316}.
\end{proof}

\begin{lem}\label{606}
Given any point $[\rho] \in Q^{ss}_\Lambda$, the closure of any ${\rm SL}(V)$-orbit of $[\rho]$ in $Q^{ss}_\Lambda$ is contained in $Q^{ss}_\Lambda$, where the closure is taken in ${\rm Quot}^{ss}(\Lambda_k \otimes V \otimes \mathcal{G},P)$.
\end{lem}

\begin{proof}
Based on the Hilbert-Mumford Criterion \cite[Theorem 2.1]{MumFogKir}, we need to find the limit point of any one-parameter subgroup action on a given point in $Q^{ss}_\Lambda$. Let $\varphi: \mathbb{G}_m \rightarrow {\rm SL}(V)$ be an one parameter-subgroup. The vector space $V$ can be decomposed as $V= \bigoplus\limits_{\alpha} V_{\alpha}$, where $t \cdot v_{\alpha}=t^{\alpha}v_{\alpha}$ for $v_{\alpha} \in V_{\alpha}$. Therefore we can define a filtration $V_{\geq \beta}:=\bigoplus\limits_{\alpha \geq \beta}V_{\alpha}$ of $V$.

Let $[\rho_{\mathcal{F}}: \Lambda_k \otimes V \otimes \mathcal{G} \rightarrow \mathcal{F}]$ be a point in $Q^{ss}_{\Lambda}$. The filtration $\mathcal{F}_{\geq \beta}$ of $\mathcal{F}$ is defined as
\begin{align*}
\mathcal{F}_{\geq \beta}:=\rho_{\mathcal{F}}(\Lambda_k \otimes V_{\geq \beta} \otimes \mathcal{G}),
\end{align*}
and the graded part is
\begin{align*}
\mathcal{F}_{\beta}:=\rho_{\mathcal{F}}(\Lambda_k \otimes V_{\beta} \otimes \mathcal{G}).
\end{align*}
With respect to the one-parameter subgroup $\varphi$ and the point $\rho_{\mathcal{F}}$, the limit point is
\begin{align*}
[\rho_{\mathcal{F}'}:= \Lambda_k \otimes V \otimes \mathcal{G} \rightarrow \mathcal{F}'] \in {\rm Quot}^{ss}(\Lambda_k \otimes V \otimes \mathcal{G},P),
\end{align*}
where $\mathcal{F}'=\bigoplus\limits_{\beta} \mathcal{F}_{\beta}$. It suffices to show that the point $\rho_{\mathcal{F}'}$ is contained in $Q^{ss}_\Lambda$.

Now let $\mathscr{H}_{\geq \beta}$ be the saturation of $\mathcal{F}_{\geq \beta}$. By the property of saturation, $\mathscr{H}_{\geq \beta}$ are $\Lambda$-modules (see \S 4.1). Define $\mathscr{H}_{\beta}:= \mathscr{H}_{\geq \beta}/\mathscr{H}_{\geq \beta+1}$, which are also $\Lambda$-modules. Therefore, proving $[\rho_{\mathcal{F}'}] \in Q^{ss}_\Lambda$, it suffices to prove that
\begin{itemize}
\item $\mathscr{H}_\beta$ is a $p$-semistable $\Lambda$-submodule of $\mathcal{F}$;
\item the reduced modified Hilbert polynomial of $\mathscr{H}_\beta$ is the same as that of $\mathcal{F}$;
\item $\mathcal{F}' \cong \bigoplus\limits_{\beta} \mathscr{H}_{\beta}$.
\end{itemize}

Now we will prove the above statements. Note that there is a natural map $\mathcal{F}_\beta \rightarrow \mathscr{H}_\beta$, and the image of this map is denoted by $\mathscr{I}_\beta$. The composed morphism
\begin{align*}
\Lambda_k \otimes V \otimes \mathcal{G} \rightarrow \mathcal{F} \rightarrow \mathcal{F}_\beta \rightarrow \mathscr{H}_\beta
\end{align*}
induces a map $V \rightarrow H^0(X, F_{\mathcal{E}}(\mathscr{I}_\beta)(N))$. Let $J_\beta$ be the kernel of $V \rightarrow H^0(X, F_{\mathcal{E}}(\mathscr{I}_\beta)(N))$. We have
\begin{align*}
\dim(J_\beta) \geq P(N)-h^0(X, F_{\mathcal{E}}(\mathscr{I}_\beta)(N) ).
\end{align*}
Denote by $\mathscr{J}_\beta$ the subsheaf of $\mathcal{F}$ generated by the image of $\Lambda_k \otimes J_\beta \otimes \mathcal{G}$. The following graph explains the relation of the above notations.
\begin{center}
\begin{tikzcd}
    J_{\beta} \arrow[r] \arrow[d,dotted] & V \arrow[r] \arrow[d,dotted]  & V_{\beta} \arrow[r] \arrow[d,dotted]  & H^0(X, F_{\mathcal{E}}(\mathscr{I}_\beta)(N)) \arrow[d,dotted] \\
    \mathscr{J}_\beta \arrow[r] & \mathcal{F} \arrow[r] & \mathcal{F}_{\beta} \arrow[r] & \mathscr{I}_\beta \subseteq \mathscr{H}_\beta
\end{tikzcd}
\end{center}
Note that $\mathscr{J}_\beta$ maps zero in $\mathscr{H}_\beta$ and also in $\mathscr{I}_\beta$. This implies that
\begin{align*}
r(\mathscr{J}_\beta) \leq r(\mathcal{F})-r(\mathscr{I}_\beta).
\end{align*}
Since $\rho_{\mathcal{F}'} \in {\rm Quot}^{ss}(\Lambda_k \otimes V \otimes \mathcal{G},P)$, by Lemma \ref{3141}, \ref{3142} and Theorem \ref{316}, we have
\begin{align*}
\frac{\dim(J_\beta)}{ \mathscr{J}_\beta} \leq \frac{P(N)}{r(\mathcal{F})}.
\end{align*}
Combing the above three inequalities, we have
\begin{align*}
\frac{P(N)}{r(\mathcal{F})} \leq \frac{h^0(X, F_{\mathcal{E}}(\mathscr{I}_\beta)(N)  ) }{r(\mathscr{I}_\beta)}.
\end{align*}

We will prove that $\mathscr{H}_{\geq \beta}$ is a $p$-semistable $\Lambda$-submodule of $\mathcal{F}$ with the same reduced modified Hilbert polynomial by induction on $\beta$. Suppose that the statement holds for $\mathscr{H}_{\geq \beta+1}$. Then $\mathcal{F}/\mathscr{H}_{\geq \beta+1}$ is a $p$-semistable $\Lambda$-module. Note that $\mathscr{I}_\beta \subseteq \mathscr{H}_\beta \subseteq \mathcal{F}/\mathscr{H}_{\geq \beta+1}$, which implies
\begin{align*}
\frac{ h^0(X, F_{\mathcal{E}}(\mathscr{I}_\beta)(N)   ) }{ r(\mathscr{I}_\beta) } \leq  \frac{ h^0(X, F_{\mathcal{E}}(\mathscr{H}_\beta)(N)   ) }{ r(\mathscr{H}_\beta) } \leq \frac{P(N)}{r(\mathcal{F})}.
\end{align*}
Together with the inequality $\frac{P(N)}{r(\mathcal{F})} \leq \frac{h^0(X, F_{\mathcal{E}}(\mathscr{I}_\beta)(N)  ) }{r(\mathscr{I}_\beta)}$ we discussed above, we have
\begin{itemize}
\item $h^0(X, F_{\mathcal{E}}(\mathscr{I}_\beta)(N)  ) =h^0(X, F_{\mathcal{E}}(\mathscr{H}_\beta)(N)$;
\item the reduced modified Hilbert polynomials of $\mathcal{F}$, $\mathscr{H}_{\beta}$ are the same.
\end{itemize}
Therefore, $\mathscr{H}_{\beta}$ is a $p$-semistable $\Lambda$-module. This also implies that $\mathscr{H}_{\geq \beta}$ is $p$-semistable. This finishes the proof by induction.

The above proof also tells us that
\begin{itemize}
    \item $\mathscr{H}_\beta$ is a $p$-semistale $\Lambda$-submodule of $\mathcal{F}$;
    \item The reduced modified Hilbert polynomials of $\mathscr{H}_{\beta}$ are the same;
    \item $h^0(X, F_{\mathcal{E}}(\mathscr{I}_\beta)(N)  ) =h^0(X, F_{\mathcal{E}}(\mathscr{H}_\beta)(N))$.
\end{itemize}
The only thing left is $\mathcal{F}' \cong \bigoplus\limits_{\beta} \mathscr{H}_{\beta}$. To prove this isomorphism, we show $\mathcal{F}_\beta \cong \mathscr{H}_\beta$.

Since the integer $N$ is large enough, $F_{\mathcal{E}}(\mathscr{I}_\beta)(N)$ and $F_{\mathcal{E}}(\mathscr{H}_\beta)(N)$ are generated by global sections. By Lemma \ref{508}, the equality $h^0(X, F_{\mathcal{E}}(\mathscr{I}_\beta)(N)  ) =h^0(X, F_{\mathcal{E}}(\mathscr{H}_\beta)(N))$ implies that $\mathscr{I}_\beta=\mathscr{H}_\beta$. Recall that $\mathscr{I}_\beta$ is defined as the image of $\mathcal{F}_\beta$ in $\mathscr{H}_\beta$. Therefore we have $\mathcal{F}_\beta \cong \mathscr{H}_\beta$, and
\begin{align*}
\mathcal{F}' = \bigoplus\limits_\beta \mathcal{F}_\beta \cong \bigoplus\limits_\beta \mathscr{H}_\beta.
\end{align*}
This finishes the proof of this lemma.
\end{proof}

\begin{lem}\label{6071}
Let $[\Lambda_k \otimes V \otimes \mathcal{G} \rightarrow \mathcal{F}_i]$, $i=1,2$ be two points in $Q_{\Lambda}^{ss}$. The closures of the corresponding orbits in $Q_{\Lambda}^{ss}$ intersect if and only if $gr^{\rm JH}(\mathcal{F}_1) \cong gr^{\rm JH}(\mathcal{F}_2)$ with respect to the $\Lambda$-Jordan-H\"older filtrations.
\end{lem}

\begin{proof}
The proof of this lemma is the same as the proof of Lemma \ref{3181}.
\end{proof}

By Theorem \ref{314}, let ${\rm Quot}^{ss}(\Lambda_k \otimes V \otimes \mathcal{G},P)/{\rm SL}(V)$ be the good geometric quotient. Denote by
\begin{align*}
\varphi: {\rm Quot}^{ss}(\Lambda_k \otimes V \otimes \mathcal{G},P) \rightarrow {\rm Quot}^{ss}(\Lambda_k \otimes V \otimes \mathcal{G},P)/{\rm SL}(V)
\end{align*}
the morphism.

\begin{lem}\label{6072}
The image $\varphi(Q^{ss}_{\Lambda})$ is a locally closed subscheme.
\end{lem}

\begin{proof}
The proof is same as \cite[Lemma 4.5, 4.6]{Simp2}.
\end{proof}

Let
\begin{align*}
\mathcal{M}_\Lambda^{ss}(\mathcal{E},\mathcal{O}_X(1),P)=Q^{ss}_\Lambda/ {\rm SL}(V)
\end{align*}
be the GIT quotient, and let $\mathcal{M}_\Lambda^{s}(\mathcal{E},\mathcal{O}_X(1),P)$ be the set for stable points in $\mathcal{M}_\Lambda^{ss}(\mathcal{E},\mathcal{O}_X(1),P)$. In this section, we only give the proof for the semistable case. In fact, the stable case can be proved similarly.

The discussion in this section gives the following theorem.
\begin{thm}\label{607}
\leavevmode
\begin{enumerate}
\item The moduli space $\mathcal{M}_\Lambda^{ss}(\mathcal{E},\mathcal{O}_X(1),P)$ is a quasi-projective $S$-scheme.
\item There exists a natural morphism
\begin{align*}
\widetilde{\mathcal{M}}_{\Lambda}^{ss}(\mathcal{E},\mathcal{O}_X(1),P) \rightarrow \mathcal{M}_\Lambda^{ss}(\mathcal{E},\mathcal{O}_X(1),P)
\end{align*}
such that $\mathcal{M}_\Lambda^{ss}(\mathcal{E},\mathcal{O}_X(1),P)$ universally co-represents $\widetilde{\mathcal{M}}_\Lambda^{ss}(\mathcal{E},\mathcal{O}_X(1),P)$, and the points of $\mathcal{M}_\Lambda^{ss}(\mathcal{E},\mathcal{O}_X(1),P)$ represent the $S$-equivalent classes of $p$-semistable $\Lambda$-modules with modified Hilbert polynomial $P$.
\item $\mathcal{M}_\Lambda^{s}(\mathcal{E},\mathcal{O}_X(1),P)$ is a coarse moduli space of $\widetilde{\mathcal{M}}_\Lambda^{s}(\mathcal{E},\mathcal{O}_X(1),P)$, and its points represent isomorphism classes of $p$-stable $\Lambda$-modules.
\end{enumerate}
\end{thm}

\begin{exmp}\label{608}
Let $\mathcal{X}$ be a projective Deligne-Mumford stack over an affine scheme $S$. Let $\mathcal{L}$ be a fixed line bundle over $\mathcal{X}$. Similar to \S 2.5.2, we can define the moduli problem of $p$-semistable $\mathcal{L}$-twisted Hitchin pairs $\widetilde{\mathcal{M}}^{ss}_H(\mathcal{X},\mathcal{L})$. In \S 4.1, $\mathcal{L}$-twisted Hitchin pairs can be considered as $\Lambda$-modules. Therefore, the moduli space of $p$-semistable $\mathcal{L}$-twisted Hitchin pairs $\mathcal{M}^{ss}_H(\mathcal{X},\mathcal{L})$ exists and universally co-represents the moduli problem $\widetilde{\mathcal{M}}^{ss}_H(\mathcal{X},\mathcal{L})$.
\end{exmp}

\bigskip
\noindent\small{\textsc{Department of Mathematics, South China University of Technology}\\
381 Wushan Rd, Tianhe Qu, Guangzhou, Guangdong, China}\\
\emph{E-mail address}:  \texttt{hsun71275@scut.edu.cn}

\end{document}